\documentclass[10pt,leqno]{amsart}
\usepackage[colorlinks,linkcolor=blue,citecolor=blue,urlcolor=blue]{hyperref}
\usepackage[final]{microtype}
\usepackage{amsfonts,amssymb}
\usepackage{color}
\usepackage{graphicx}
\usepackage{mathtools}
\usepackage{extarrows}
\usepackage{algorithmic}   
\usepackage{caption}
\usepackage{subcaption}
\usepackage{tikz}
\usepackage{mathrsfs}
\usepackage[isbn=false,maxbibnames=99]{biblatex}
\usepackage[normalem]{ulem}

\allowdisplaybreaks

\usepackage[noabbrev,nameinlink,capitalize]{cleveref}

\newtheorem{theorem}{Theorem}[section]
\newtheorem{lemma}[theorem]{Lemma}
\newtheorem{corollary}[theorem]{Corollary}
\newtheorem{assumption}[theorem]{Assumption}

\theoremstyle{definition}
\newtheorem{definition}[theorem]{Definition}
\newtheorem{example}[theorem]{Example}
\newtheorem{remark}[theorem]{Remark}

\def\cleverreffix#1{\AddToHook{env/#1/begin}{\crefalias{theorem}{#1}}}
\cleverreffix{lemma}
\cleverreffix{corollary}
\cleverreffix{assumption}
\cleverreffix{definition}
\cleverreffix{example}
\cleverreffix{remark}

\numberwithin{equation}{section}

\usepackage{my_latex_commands}

\newcommand{\lip}{\mathrm{Lip}}
\newcommand{\ran}{\mathrm{ran}}
\newcommand{\BL}{\mathrm{BL}}
\newcommand{\wdist}{\DD}
\newcommand\weaklystar{\stackrel{\star}{\rightharpoonup}}
\newcommand\polar{^\circ}
\newcommand\bipolar{^{\circ\circ}}
\newcommand\cl{\operatorname{cl}}
\DeclarePairedDelimiterX\eqclass[2]{\llbracket}{\rrbracket}{#1,#2}

\newcommand{\dualelement}{\Phi}

\newcommand\mathdh{\textit{\dh}}

\usepackage{enumerate}
\usepackage{enumitem}
\setenumerate{label=(\roman*)}

\def\itemref#1{\itemrefintern#1MYENDING}
\def\itemrefintern#1:#2:#3MYENDING{\cref{#1:#2}\ref{#1:#2:#3}}

\def\Itemrefintern#1:#2:#3MYENDING{\Cref{#1:#2}\ref{#1:#2:#3}}
\DeclarePairedDelimiterX\seq[1](){#1}
\DeclareMathAlphabet{\mathpzc}{OT1}{pzc}{m}{it}
\newcommand\oo{\mathpzc{o}}

\begin{document}
\title[
	Second-order conditions for problems with transport distances
]{
	No-gap second-order conditions for
	optimization problems involving transport distances
}

\author{Nicolas Borchard} \address{BTU Cottbus-Senftenberg, Fakult\"at 1, 
	Fachgebiet Optimale Steuerung, Platz der Deutschen Einheit 1, 03046 Cottbus, Germany}
\email{nicolas.borchard@b-tu.de}

\author{Christian Meyer} \address{Technische Universit\"at Dortmund, Fakult\"at f\"ur
	Mathematik, Lehrstuhl LSX, Vogelpothsweg 87, 44227 Dortmund, Germany}
\email{christian2.meyer@tu-dortmund.de}

\author{Gerd Wachsmuth} \address{BTU Cottbus-Senftenberg, Fakult\"at 1, 
	Fachgebiet Optimale Steuerung, Platz der Deutschen Einheit 1, 03046 Cottbus, Germany}
\email{gerd.wachsmuth@b-tu.de}

\thanks{This research has been funded by the Deutsche Forschungsgemeinschaft (DFG, German research foundation) via grant number 576024132.}



\begin{abstract}
	We consider optimization problems
	in the space of measures.
	As a regularization term,
	the problem includes the transport distance
	to a given prior measure.
	For the derivation of second-order optimality conditions
	of no-gap type, the theory of
	weak-$\star$ second subderivatives is used which will lead to an equivalence
	with quadratic growth under additional assumptions on the smooth part of the objective and 
	on the Kantorovich potential, i.e., the solution of the dual transport problem. 
	Further, the weak-$\star$
	second subderivative is calculated and weak-$\star$ epidifferentiability is
	proven. Finally, the results are applied to optimal control problems in measure space.
\end{abstract}

\maketitle

\section{Introduction}

In this paper, we establish no-gap second-order optimality conditions for optimization problems
on the space of measures
with transport regularization.
The problem under consideration reads
\begin{equation*}
	\tag{P}
	\label{P}
	\min_{u \in \MM(\omega_1)}
	F(u) + \frac \alpha 2 D(u)
	,
\end{equation*}
where $D$ is the transport cost functional that measures the distance of $u \in \MM(\omega_1)$ 
to some prior measure $u_0 \in \MM(\omega_0), u_0 \ge 0$, 
w.r.t.\ the transportation cost $c : \omega_0 \times \omega_1 \to \R$, i.e., it is the optimal value of 
the following \emph{Kantorovich problem}:
\begin{equation}
	\label{def:D}
	D(u)
	:=
	\inf\set*{
		\int_{\omega_0 \times \omega_1} c(x,\eta) \d \pi(x,\eta)
		\given
		\pi \in \Gamma(u_0, u)
	}
	,
\end{equation}
where the set of feasible transport plans is given by
\begin{equation}
	\Gamma(u_0,u)
	:=
	\set*{
		\pi \in \MM(\omega_0 \times \omega_1)
		\given
		\pi \ge 0,
		P_0 \# \pi = u_0,
		P_1 \# \pi = u
	}
	.
\end{equation}
Herein, $\omega_0, \omega_1 \subset \R^d$ are given compact sets. 
Moreover, $\#$ denotes the measure-pushforward and $P_i \colon \omega_0 \times \omega_1 \to \omega_i$
is the projection on $\omega_i$, $i=0,1$.
Furthermore, $\alpha > 0$ and $F: \MM(\omega_1) \to \R$ is a smooth functional.
The precise assumptions on the data are given in \cref{subsec:problem_setting} below.

Optimization problems in the space of regular Borel measures have various applications, e.g., in 
inverse problems \cite{ScherzerWalch2009,BrediesPikkarainen2013,DuvalPeyre2015},
optimal sensor placement \cite{ClasonKunisch2012,NeitzelPieperVexlerWalter2019},
or sparse optimal control \cite{ClasonKunisch2011,PieperVexler2013}.
A motivation for considering the transport distance as regularizer instead of, for instance,
the usually employed Radon norm
lies in its advantageous continuity properties: In contrast to the Radon norm, the transport distance is 
sequentially weak-$\star$ continuous. Moreover, the consideration of $D$ might be motivated by modeling aspects.
For instance, the relocation of a resource $u_0$ to minimize the functional $F$, while 
at the same time keeping the transportation costs as small as possible,
directly leads to a problem of type \eqref{P}. 
Quite recently, problems of type \eqref{P}
have been considered in \cite{MeyerWachsmuth2025}.
Therein,
existence of solutions and first-order optimality conditions
were studied.
Moreover, it was shown that optimal controls
possess higher regularity
for specific choices of the functional $F$,
e.g., if $F$ is the reduced objective functional
of a linear-quadratic optimal control problem.

As mentioned above, we aim at deriving no-gap second-order optimality conditions meaning that we search 
for second-order conditions that are \emph{equivalent} to the quadratic growth of the 
objective in the neighborhood of a local minimizer. In the recent paper \cite{BorchardWachsmuth2024}, 
such conditions have been characterized in an abstract setting.
They involve two conditions, namely 
\begin{enumerate}[label=(\roman*)]
    \item the so-called \emph{non-degeneracy condition} and
    \item coercivity of the second derivative of the objective.
\end{enumerate}
Since one part of the objective in \cite{BorchardWachsmuth2024} is non-smooth as it is the case for 
the transport distance $D$, the definition of a ``second derivative'' is not immediate and relies 
on the concept of the \emph{weak-$\star$ second subderivative}.
The application of the general theory to our problem \eqref{P} poses two major challenges.
First, one has to verify the non-degeneracy condition to be able to apply the general theory at all.
Secondly, in order to obtain a manageable second-order condition, 
one has to precisely compute the weak-$\star$ second subderivative of the transport distance. 
Following an idea of \cite{WachsmuthWachsmuth2022}, we apply a duality argument 
for both tasks, based on the pre-conjugate of the transport distance $D$.
This approach is also used in \cite{WachsmuthWalter2024}, where a problem 
similar to \eqref{P} is investigated with the Radon norm as regularizer instead of the transport distance. 
In order to put the general theory into practice, the authors of \cite{WachsmuthWalter2024} 
embedded their problem into the dual of the space of continuously differentiable functions.
In this space, the non-degeneracy condition can be verified and the weak-$\star$ second 
subderivative of the Radon norm can be computed. As a results, one obtains quadratic 
growth of the objective in the norm of the dual of $C^1$. 
Here, we will pursue the same strategy and treat \eqref{P} as a problem in $C^1(\omega_1)\dualspace$.

Let us put our work into perspective. The derivation of no-gap second-order optimality conditions 
has a long tradition, especially in the context of optimal control problems. 
The classical smooth theory requires that the regularizer is twice Fr\'echet-differentiable 
and that the second derivative is a Legendre form, see, e.g., \cite{CasasTroeltzsch2012} and the references therein. 
The transport distance $D$ however is non-smooth so that the classical theory is not applicable. 
There are multiple contributions to second-order conditions for non-smooth infinite dimensional 
problems that are not of no-gap type. We only mention \cite{Casas2012,CasasWachsmuthWachsmuth2017}
addressing optimal control problems in absence of a regularization term.
A first contribution to no-gap second-order conditions was made in \cite{ChristofWachsmuth2018}, 
where a bang-bang control problem is analyzed. 
The underlying analysis is based on the concept of the weak-$\star$ second subderivative, which is well-known
in finite-dimensions, see, e.g., \cite[Definition~13.3]{RockafellarWets1998}, and has been adapted
to infinite dimensions in \cite{Do1992,Levy1993}.
Later on, the second-order analysis of \cite{ChristofWachsmuth2018} has been advanced and
adapted to non-smooth integral functionals in \cite{WachsmuthWachsmuth2022}, to 
spatio-temporally sparse problems in \cite{BorchardWachsmuth2024}, and to 
problems involving the Radon norm in \cite{WachsmuthWalter2024}.
At this point, we would like to mention that this abstract second-order theory
does not only provide quadratic growth in a neighborhood of a minimizer,
but it can also be used
to study stability properties \cite[Section~6]{WachsmuthWalter2025:1},
to perform a sensitivity analysis \cite{ChristofWachsmuth2019}
and
to verify fast local convergence of solution methods \cite{WachsmuthWachsmuth2025:1,WachsmuthWalter2025:1}.

To the best of our knowledge, the transport distance has not been considered in this context so far. 
This even concerns a second-order differentiability analysis of the transport distance alone. 
The only contribution in this direction is \cite[Theorem~1.4]{Caja-LopezDelgadinoKitagawa2026}, 
where a formula for the second-order variation of the transport distance is given 
under the assumption that the marginals and their variations
and the domain are sufficiently smooth and the transport costs are quadratic.
This result however is not sufficient for our purpose, since the verification of the non-degeneracy condition 
requires to consider \eqref{P} in $C^1(\omega)\dualspace$ and so, even if we assume that a local 
minimizer $\bar u$ is smooth enough, we need to be able to treat 
very irregular variations of $\bar u$.
Nevertheless, the second-order derivative
from \cite[Theorem~1.4]{Caja-LopezDelgadinoKitagawa2026}
coincides with the expression of the weak-$\star$ second subderivative
obtained in the present paper, see \cref{rem:comparison}.

\subsection*{Plan of the paper}
After this introduction, we present some preliminaries that will be used throughout the paper in \cref{sec:prelim}. 
In \cref{subsec:problem_setting}, we state our standing assumptions and some well-known facts 
about the Kantorovich problem. The abstract second-order theory underlying our analysis is 
presented in \cref{subsec:abstract_second_order}. In this section, we will also give a rigorous definition 
of the non-degeneracy condition and the concept of the weak-$\star$ second subderivative. 
In \cref{sec:NDC}, see in particular \cref{ex:counter_example_inequality},
we will see that 
the non-degeneracy condition cannot hold in $\MM(\omega_1)$ and, for that reason, we 
lift our problem \eqref{P} to the space $C^1(\omega_1)\dualspace$ as indicated above. 
In \cref{subsec:dual_C1} and \cref{subsec:lifted_problem}, we recall some properties of $C^1(\omega)\dualspace$ 
from \cite{WachsmuthWalter2024} and introduce the lifted problem, respectively.
\cref{sec:NDC} is then dedicated to the verification of the non-degeneracy condition. 
As mentioned above, it is based on the pre-conjugate of $D$ and, following an idea of 
\cite{WachsmuthWachsmuth2022}, we will establish a descent property of the pre-conjugate, which 
will ultimately imply the non-degeneracy condition. The derivation of this descent property 
requires an additional regularity assumption on the Kantorovich potential associated with 
the local minimizer, cf.~\cref{assu:regularity_assumptions_on_T} and \cref{assu:weaker_regularity_assumptions_on_T}.
Afterwards, in \cref{sec:second_subderivative}, we compute the weak-$\star$ second subderivative of $D$ 
under slightly more restrictive assumptions.
Again, we pursue the strategy of \cite{WachsmuthWalter2024} and argue by means of the 
second derivative of the pre-conjugate functional. Finally, in \cref{sec:second_order_condition}, 
we collect our findings to formulate no-gap second order conditions. Since we assume the second derivative 
of $F$ to be sequentially weak-$\star$ continuous, the non-degeneracy condition alone is 
enough to establish no-gap second-order conditions involving the weak-$\star$ second subderivative of $D$. 
If one assumes the more restrictive assumptions from \cref{sec:second_subderivative}, then the 
second-order conditions can be formulated in a more manageable form. 
The paper ends with an example from optimal control, where $F$ involves the solution operator 
of an elliptic partial differential equation.

\section{Preliminaries}
\label{sec:prelim}

Our original problem \eqref{P} is posed in the space of measures $\MM(\omega_1)$ which can be identified with $C(\omega_1)\dualspace$,
the dual space of the space of continuous functions on the compact set $\omega_1$.
However,
this space is not suitable for our purposes.
The problem is that we are interested in quadratic growth properties which cannot be obtained in the space of measures.
Thus, we work in the space $C^1(\omega_1)\dualspace$. This has multiple reasons.
First,
$C^1(\omega_1)\dualspace$ is actually bigger than $\MM(\omega_1)$,
see \Cref{subsec:dual_C1},
which means that we can still get all solutions from $\MM(\omega_1)$.
Second,
we need the non-degeneracy condition \eqref{NDC}
that will be introduced in \Cref{subsec:abstract_second_order}
and it turns out that the condition holds true in $C^1(\omega_1)\dualspace$, cf.\ \Cref{thm:NDC_D_inequality} and \Cref{lem:ndc_sufficient}.
On the other hand,
\eqref{NDC} will be proven using the descent lemma \Cref{lem:weak_descent_lemma}
and \Cref{ex:counter_example_inequality} shows
that there is no descent lemma in $\MM(\omega_1)$.
Moreover, one can directly argue that \eqref{NDC} cannot be satisfied in $\MM(\omega_1)$,
see \cref{ex:no_NDC_in_M}.

\subsection{Problem setting}
\label{subsec:problem_setting}

In this subsection, we state our standing assumptions
that will be tacitly assumed in the following.
Moreover, we fix the notation and recall basic results from optimal transport.

\begin{assumption}[Standing assumptions]
	\label{assu:standing_assumptions}
	We assume the following.
	\begin{enumerate}
		\item
			\label{assu:standing_assumptions:1}
			Let $\omega_0, \omega_1 \subset \R^d$ be compact sets, $d \in \N$.
	\item
			\label{assu:standing_assumptions:2}
		We assume that $\omega_1$ is the closure of its interior
		and
		that $\interior \omega_1$ is uniformly locally quasiconvex,
		see \cref{def:ulq}.
		\item
			\label{assu:standing_assumptions:3}
			The cost function $c \colon \omega_0 \times \omega_1 \to \R$ is continuous.
		\item
			\label{assu:standing_assumptions:4}
			The prior measure $u_0 \in \MM(\omega_0)$ satisfies $u_0 \ge 0$
			and $u_0(\omega_0) > 0$.
			Moreover, we assume $\supp(u_0) = \omega_0$.
		\item
			\label{assu:standing_assumptions:5}
			Finally, we fix
			$\bar u \in \MM(\omega_1)$ with $\bar u \ge 0$ and $\bar u(\omega_1) = u_0(\omega_0)$,
			which is a candidate for a minimizer.
	\end{enumerate}
\end{assumption}
Note that this situation is rather simple and many results are known for the theory of optimal transport in case of compact sets and continuous cost functions, e.g. in \cite[Thm.~1.4]{santambrogio}.
Further, we mention that the assumption
$\supp(u_0) = \omega_0$ is not restrictive,
since $\supp(u_0) \subset \omega_0$ is always a compact set
and, thus, we can shrink $\omega_0$
to satisfy this assumption.

In this paper, we use the following notations.
The positive real numbers are denoted by $\R^+ := (0,\infty)$
and we use
$\bar\R := \R \cup \set{\infty}$.
The closed ball in a Banach space $X$ with radius $\varepsilon > 0$ centered at $x \in X$ is denoted by $B_\varepsilon(x)$.

The set $C(\omega; \R^d)$ consists of all continuous functions from a compact set $\omega$ to $\R^d$
and $C(\omega) := C(\omega; \R^1)$.
The spaces of scalar and vector-valued Borel measures on $\omega$ are written down as $\MM(\omega)$
and $\MM(\omega; \R^d)$, respectively.
Recall that these are the dual spaces of $C(\omega)$ and $C(\omega; \R^d)$.
The space of Borel measurable functions whose $p$th power is $\mu$-integrable
is denoted by $L^p(\mu)$ for $\mu \in \MM(\omega)$.
Similarly, $L^p(\mu; \R^d)$ denotes the corresponding space of  vector-valued functions.

Lastly, we want to mention how the pushforward measure can be used for integral transformations.
Let $\mu \in \MM(\omega)$ be given and let $T \colon \omega \to \Omega$ be Borel measurable.
Then, the measure $T \# \mu$, which is defined via
\[
	(T \# \mu)(A) := \mu(T^{-1}(A))
	\qquad
	\text{for all Borel measurable sets } A \in \Omega,
\]
is a measure in $\MM(\Omega)$ and for every measurable function $f \colon \Omega \to \R$
the identity
\[
	\int_\Omega f(y) \d(T \# \mu)(y)
	=
	\int_\omega f(T(x)) \d\mu(x)
\]
holds.

Finally,
we note some properties of the transport distance $D$.
\begin{lemma}
	\label{lem:properties_D}
	The following properties hold.
	\begin{enumerate}
		\item
		\label{lem:properties_D:1}
		For all $u \in \MM(\omega_1)$,
		the values of the Kantorovich dual problem and primal problem coincide,
		i.e.,
		\begin{align*}
			\MoveEqLeft
			\sup_{\psi \in C(\omega_1)}
			\int_{\omega_1} \psi(\eta) \d u(\eta)
			+
			\int_{\omega_0} \psi^{\bar c}(x) \d u_0(x)
			\\
			&=
			\inf\set*{
				\int_{\omega_0 \times \omega_1} c(x,\eta) \d \pi(x,\eta)
				\given
				\pi \in \Gamma(u_0, u)
			}
			=
			D(u)
			,
		\end{align*}
		where the $\bar c$-conjugate is defined via
		\begin{align}
			\label{eq:cconjugate}
			\psi^{\bar c}(x)
			:=
			\inf_{\eta \in \omega_1} c(x,\eta) - \psi(\eta)
			\qquad\forall x \in \omega_0
			,
		\end{align}
		cf.\ \cite[Def.~1.10]{santambrogio}.
		\item
		\label{lem:properties_D:2}
		Let the cost function $c$ be given by $c(x,\eta) := \abs{x - \eta}^p$
		for $p > 1$
		and $u_0$ be absolutely continuous w.r.t.\ the Lebesgue measure on $\omega_0$.
		Then,
		$D$ is strictly convex.
	\end{enumerate}
\end{lemma}
\begin{proof}
	The claims are proven in
	\cite[Thm.~1.46]{santambrogio}
	and
	\cite[Prop.~7.19]{santambrogio}.
\end{proof}
Given $u \in \MM(\omega_1)$,
the solutions of the Kantorovich dual problem
(as stated in \itemref{lem:properties_D:1})
are called
Kantorovich potentials,
i.e., $\psi \in C(\omega_1)$ is a Kantorovich potential if and only if
\begin{equation}
	\label{eq:Kantorovich_potential_1}
	\int_{\omega_1} \psi(\eta) \d u(\eta)
	+
	\int_{\omega_0} \psi^{\bar c}(x) \d u_0(x)
	=
	D(u)
	.
\end{equation}

\subsection{Abstract second-order conditions}
\label{subsec:abstract_second_order}

In this subsection, we briefly introduce the framework developed in
\cite{ChristofWachsmuth2018,WachsmuthWachsmuth2022,BorchardWachsmuth2024}
that we will use to obtain second-order conditions
for problems of the form
\begin{equation}
	\label{eq:abstract_problem}
	\min_{x \in X} F(x) + G(x),
\end{equation}
where $X$ is a Banach space, and $F \colon X \to \R$, $G \colon X \to \bar\R$ are given.
The space $X$ is assumed to be the (topological) dual space of a separable Banach space $Y$.
Clearly, points outside of the set $\dom(G) = \set{x \in X \given G(x) < \infty}$
cannot be solutions of \eqref{eq:abstract_problem},
unless $G \equiv \infty$.
Consequently, the values of $F$ on $X \setminus \dom(G)$
do not play any role in the theory
and it would be sufficient of $F$ is just defined on $\dom(G)$,
as it is done in, e.g., \cite[Section~2]{WachsmuthWachsmuth2022}.

The functional $G$ is allowed to be nonsmooth
and
we introduce its weak-$\star$ second subderivative.
\begin{definition}[{\cite[Def.~2.2,~2.7]{WachsmuthWachsmuth2022}}]
	\label{def:wsss}
	Let $x \in \dom(G)$ and $w \in Y$ be given.
	Then, the \emph{weak-$\star$ second subderivative} $G''(x,w;\cdot) : X \to [-\infty, \infty]$ of $G$ at $x$ for $w$ is defined by
	\[
		G''(x,w;h)
		:=
		\inf\set*{
			\liminf_{k\to\infty} \frac{G(x + t_k h_k) - G(x) - t_k \dual{w}{h_k}}{t_k^2 / 2}
			\given
			t_k \searrow 0,
			h_k \weaklystar h
		}
		.
	\]
	Here, the infimum is taken w.r.t.\ all sequences
	$\seq{t_k} \subset \R^+$ and $\seq{h_k} \subset X$
	with
	$t_k \to 0$ and $h_k \weaklystar h$ in $X$.
	The functional $G$ is weak-$\star$ twice epi-differentiable
	(respectively strictly twice epi-differentiable, respectively strongly)
	at $x$ for $w$ in a direction $h \in X$, if for all $\seq{t_k} \subset \R^+$ with $t_k \to 0$
	there exists a sequence $\seq{h_k} \subset X$
	(called ``\emph{recovery sequence}'')
	satisfying
	$h_k \weaklystar h$ (respectively $h_k \weaklystar h$ and $\norm{h_k}_X \to \norm{h}_X$, respectively $h_k \to h$)
	and
	\[
		G''(x,w;h)
		=
		\lim_{k \to \infty} \frac{G(x + t_k h_k) - G(x) - t_k \dual{w}{h_k}}{t_k^2 / 2}
		.
	\]
	The functional $G$ is called weak-$\star$/strictly/strongly twice epi-differentiable at $x$ for $w$ if it is
	weak-$\star$/strictly/strongly twice epi-differentiable at $x$ for $w$ in all directions $h \in X$.
\end{definition}

Concerning the functional $F$,
we assume that it is twice differentiable at $\bar x$ in the following sense.
There exist $F'(\bar x) \in Y$ and a bounded bilinear form $F''(\bar x) \colon X \times X \to \R$
such that
\begin{equation}
	\label{eq:abstract_differentiable_F}
	\lim_{k \to \infty}
	\frac
	{F(\bar x + t_k h_k) - F(\bar x) - t_k F'(\bar x) h_k - \frac 1 2 t_k^2 F''(\bar x) h_k^2}
	{t_k^2}
	=
	0
\end{equation}
holds for all sequences $\seq{ t_k } \subset \R^+$ and $\seq{ h_k } \subset X$
with $t_k \to 0$, $h_k \weaklystar h$ and $\bar x + t_k h_k \in \dom(G)$.
Here,
$F''(\bar x) h_k^2 := F''(\bar x) [h_k, h_k]$.

Additionally, we need the following condition.

\begin{definition}[{\cite[Thm.~2.9]{WachsmuthWachsmuth2022}}]
	\label{def:ndc}
	Let $F'(\bar x) \in Y$ and
	a bounded linear functional
	$F''(\bar x) \colon X \times X \to \R$ be given.
	We say that the \emph{non-degeneracy condition} is satisfied
	if
	\begin{equation*}
		\tag{NDC}
		\label{NDC}
		\begin{split}
			\text{for all }
			\seq{t_k} \subset \R^+,
			\seq{h_k} \subset X
			\text{ with }
			t_k \to 0,
			h_k \weaklystar 0,
			\norm{h_k}_X = 1,
			\text{ we have}
			\\
			\liminf_{k\to\infty}
			\parens*{
			\frac{1}{t_k^2} \parens*{ G(\bar x + t_k h_k) - G(\bar x) }
			+
			\frac 1 {t_k} \dual{F'(\bar x)}{h_k}
			+
			\frac 1 2 F''(\bar x) h_k^2
			}
			>
			0
			.
		\end{split}
	\end{equation*}
\end{definition}

Now, we are in position to state the second-order optimality result.
\begin{theorem}[{\cite[Thm~2.20]{BorchardWachsmuth2024}}]
	\label{thm:abstract_second_order_result}
	Let $\bar x \in \dom(G)$ be given such that $F$ is twice differentiable at $\bar x$
	in the sense of \eqref{eq:abstract_differentiable_F}.
	Further, we assume that the map $X \ni h \mapsto F''(\bar x) h^2 \in \R$
	is sequentially weak-$\star$ continuous.
	Then, the following two conditions are equivalent.
	\begin{enumerate}
		\item
			There exist $\kappa,\varepsilon > 0$ such that
			the quadratic growth condition
			\begin{equation}
				\label{eq:abstract_QG}
				F(x) + G(x)
				\ge
				F(\bar x) + G(\bar x) + \frac \kappa 2 \norm{x - \bar x}_X^2
				\quad
				\forall x \in B_\varepsilon(\bar x)
			\end{equation}
			holds.
		\item
			The condition \eqref{NDC} and the following second-order sufficient condition
			\begin{align}
				F''(\bar x) h^2 + G''(\bar x, -F'(\bar x); h) > 0
				\quad
				\forall h \in X \setminus \set{0}
			\end{align}
			are satisfied.
	\end{enumerate}
\end{theorem}
At this point we mention that
it is sufficient to assume that
$h \mapsto F''(\bar x) h^2$ is sequentially weak-$\star$ lower semicontinuous
if the functional $G$ is strongly twice epi-differentiable at $\bar x$
for $-F'(\bar x)$.
However, in our situation, this is not applicable since $D$ is not
expected to be strongly twice epi-differentiable.

Finally, we recall a result which can be used to verify \eqref{NDC}.
\begin{lemma}[{\cite[Lem.~2.12]{WachsmuthWachsmuth2022}}]
	\label{lem:ndc_sufficient}
	Suppose that $h \mapsto F''(x) h^2$ is
	sequentially weak-$\star$ lower semicontinuous
	and that there exist $c, \varepsilon > 0$
	such that
	\[
		G(x) - G(\bar x) + F'(\bar x)(x - \bar x)
		\ge
		\frac c 2 \norm{x - \bar x}_X^2
		\qquad
		\forall x \in B_\varepsilon(\bar x)
		.
	\]
	Then, \eqref{NDC} is satisfied.
\end{lemma}

At the moment, this theory is formulated with abstract spaces $X$ and $Y$.
In order to apply the theory to our problem \eqref{P},
we need to find suitable spaces.
As the transport distance $D$
(which takes the role of $G$, up to scaling with $\frac\alpha2$)
is defined on $\MM(\omega_1)$,
$Y = C(\omega_1)$ (resulting in $X = Y\dualspace = \MM(\omega_1)$) seems to be a natural choice.
However, we briefly argue that the quadratic growth \eqref{eq:abstract_QG}
and, consequently, \eqref{NDC},
cannot be satisfied in this setting under a mild assumption on $F$.

\begin{example}
	\label{ex:no_NDC_in_M}
	Assume that the functional $F \colon \MM(\omega_1) \to \R$
	satisfies the assumptions of \cref{thm:abstract_second_order_result}
	and, additionally,
	suppose that $F$
	is sequentially weak-$\star$ continuous on $\MM(\omega_1)$.

	For an arbitrary $k \in \N$,
	we partition $\R^d$ into cubes of side length $k^{-1}$, i.e.,
	$\R^d = \bigcup_{j\in \N} A_{k,j}$,
	where $A_{k,j} = k^{-1}( [0,1)^d + z_j )$
	and $\set{z_j}_{j \in \N}$ is an enumeration of  $\Z^d$.

	Consider $k,j \in \N$ such that $\omega_1 \cap A_{k,j} \ne \emptyset$.
	Using that $\omega_1$ is the closure of its interior,
	we can find $x_{k,j} \in \omega_1$
	such that
	$\dist( x_{k,j}, A_{k,j} ) \le k^{-1}$
	and
	$\bar u(\set{x_{k,j}}) = 0$.

	We define
	\begin{equation*}
		u_k
		:=
		\sum_{j \in \N, \omega_1 \cap A_{k,j} \ne \emptyset} \bar u( \omega_1 \cap A_{k,j} ) \delta_{x_{k,j}}
		,
	\end{equation*}
	where $\delta_{x_{k,j}}$ is the Dirac measure at $x_{k,j}$.
	Note that $u_k$ is a nonnegative measure on $\omega_1$
	with $u_k(\omega_1) = \bar u(\omega_1) = u_0(\omega_0)$.
	Using the uniform continuity of continuous functions on the compact set $\omega_1$, one readily verifies that 
	$u_k \weaklystar \bar u$ as $k \to \infty$.
	Since the nonnegative measures $\bar u$ and $u_k$
	have disjoint support and equal mass,
	we directly get
	$ \norm{ \bar u - u_k }_{\MM(\omega_1)} = 2 \bar u(\omega_1)$.
	Now, we choose $\varepsilon \in (0, 2 \bar u(\omega_1))$
	and consider
	$u_{\varepsilon,k} := \bar u + \frac{\varepsilon}{2 \bar u(\omega_1)} (u_k - \bar u)$.
	This directly yields
	that $u_{\varepsilon,k}$ is a nonnegative measure on $\omega_1$,
	$u_{\varepsilon, k}(\omega_1) = u_0(\omega_0)$,
	$u_{\varepsilon,k} \weaklystar \bar u$ as $k \to \infty$
	and
	$\norm{ \bar u - u_{\varepsilon, k} }_{\MM(\omega_1)} = \varepsilon$.

	The sequential weak-$\star$ continuity of $F$ implies $F(u_{\varepsilon, k}) \to F(\bar u)$.
	From \cite[Thm.~1.51]{santambrogio}
	we get $D(u_{\varepsilon, k}) \to D(\bar u)$.
	In total,
	$F(u_{\varepsilon, k}) + \frac \alpha 2 D(u_{\varepsilon, k}) \to F(\bar u) + \frac \alpha 2 D(\bar u)$
	and, by construction, we have
	$\norm{u_{\varepsilon, k} - \bar u}_{\MM(\omega_1)} = \varepsilon$.
	From these two observations, it easily follows that the quadratic growth condition
	\eqref{eq:abstract_QG}
	(with $G = \frac\alpha2 D$)
	cannot be satisfied for any constants $\kappa,\varepsilon > 0$
	in the space $X = \MM(\omega_1)$.
	Consequently, \cref{thm:abstract_second_order_result}
	implies that
	the condition \eqref{NDC} is violated as well.
\end{example}
In turn,
we need to weaken the space $X$ and we will see that the choices
$Y = C^1(\omega_1)$ and $X = C^1(\omega_1)\dualspace$ are suitable.

\subsection{The dual space of \texorpdfstring{$C^1(\omega_1)$}{C1(omega\_1)}}
\label{subsec:dual_C1}
In this subsection, we recall some facts from \cite[§5.2]{WachsmuthWalter2024}.
The assumption on $\omega_1$ to be uniformly locally quasiconvex is important
for some of the nice properties of the dual space $C^1(\omega_1)\dualspace$,
e.g., \Cref{lem:BLnorm_C1*norm_equivalent} could fail to be true.
We start with its definition.

\begin{definition}[{\cite[Def.~5.1]{WachsmuthWalter2024}}]
	\label{def:ulq}
	A set $A \subset \R^d$ is called
	\emph{uniformly locally quasiconvex},
	if
	there exist $r > 0$ and $C \ge 1$
	such that for all $x,y \in A$ with $\abs{x-y} \le r$,
	there exists a curve $\gamma \in C([0,1];\omega_1)$ with
	$\gamma(0) = x, \gamma(1) = y$ and
	Lipschitz constant at most $C \abs{x - y}$,
	i.e., $\abs{\gamma(s) - \gamma(t)} \le C \abs{x - y} \abs{s - t}$
	for all $s,t \in [0,1]$.
\end{definition}

We define the space $C^1(\omega_1)$ as the space of
all functions $\varphi \in C(\omega_1)$
that are continuously differentiable on $\interior(\omega_1)$
such that $\nabla\varphi$ can be continuously extended from $\interior(\omega_1)$ to $\omega_1$,
cf. \cite[p.~5]{WachsmuthWalter2024}.
Equipped with the norm
\begin{equation*}
	\norm{ \varphi }_{C^1(\omega_1)}
	:=
	\max\set{
		\norm{ \varphi }_{C(\omega_1)}
		,
		\norm{ \nabla\varphi }_{C(\omega_1; \R^d)}
	}
	,
\end{equation*}
$C^1(\omega_1)$ becomes a Banach space.
On order to describe the dual space of $C^1(\omega_1)$,
we use the identification
\begin{equation*}
	C^1(\omega_1)
	\cong
	\CC
	:=
	\set{(\varphi, \nabla \varphi) \given \varphi \in C^1(\omega_1)}
	\subset
	C(\omega_1) \times C(\omega_1; \R^d)
	.
\end{equation*}
Note that this implies the separability of $C^1(\omega_1)$.
If we equip this latter product space with the norm
\begin{equation*}
	\norm{(\varphi,\phi)}_{C \times C^d}
	:=
	\max\set{\norm{\varphi}_{C(\omega_1)}, \norm{\phi}_{C(\omega_1;\R^d)}}
	\qquad
	\forall (\varphi, \phi) \in C(\omega_1) \times C(\omega_1;\R^d)
	,
\end{equation*}
the map
$C^1(\omega_1) \ni \varphi
\mapsto
(\varphi, \nabla \varphi) \in C(\omega_1) \times C(\omega_1; \R^d)$
becomes an isometry.
It is easy to see that the dual space of $C(\omega_1) \times C(\omega_1; \R^d)$ is
$\MM(\omega_1) \times \MM(\omega_1; \R^d)$ with the duality pairing
\begin{equation*}
	\dual{(\mu,\nu)}{(\varphi,\phi)}_{C \times C^d}
	:=
	\dual{\mu}{\varphi}_{C}
	+
	\dual{\nu}{\phi}_{C^d}
	:=
	\dual{\mu}{\varphi}_{C(\omega_1)}
	+
	\dual{\nu}{\phi}_{C(\omega_1;\R^d)}
	.
\end{equation*}
Due to a standard result from functional analysis,
the dual space of $C^1(\omega_1)$ can be identified with
a quotient space of $\MM(\omega_1) \times \MM(\omega_1; \R^d)$.

\begin{lemma}[{\cite[Prop.~11.10]{Brezis2011}, \cite[Lemma~5.7]{WachsmuthWalter2024}}]
	The dual space of $C^1(\omega_1)$ satisfies
	\begin{equation}
		(\MM(\omega_1) \times \MM(\omega_1; \R^d)) / \CC^\bot
		\cong
		C^1(\omega_1)\dualspace
		,
	\end{equation}
	where the isometric isomorphism is given by
	\begin{equation*}
		(\MM(\omega_1) \times \MM(\omega_1; \R^d)) / \CC^\bot
		\ni
		(\mu,\nu) + \CC^\bot
		\mapsto
		\parens*{
			C^1(\omega_1) \ni \varphi \mapsto \dual{\mu}{\varphi}_C + \dual{\nu}{\nabla \varphi}_{C^d}
		}
	\end{equation*}
	and the annihilator $\CC^\bot$ is defined via
	\[
		\CC^\bot
		:=
		\set{(\mu,\nu) \in \MM(\omega_1) \times \MM(\omega_1; \R^d)
		\given
		\dual{\mu}{\varphi}_C + \dual{\nu}{\nabla \varphi}_{C^d} = 0
		\quad
		\forall \varphi \in C^1(\omega_1)
		}
		.
	\]
\end{lemma}

This shows that the dual space $C^1(\omega_1)\dualspace$ is (isomorphic to) a quotient space
and its elements can be identified with equivalence classes.
We denote these by
\begin{equation*}
	\eqclass{\mu}{\nu}
	:=
	(\mu,\nu) + \CC^\bot
\end{equation*}
and the dual pairing becomes
\begin{equation*}
	\dual{\eqclass{\mu}{\nu}}{\varphi}_{C^1(\omega_1)}
	=
	\dual{\mu}{\varphi}_C + \dual{\nu}{\nabla \varphi}_{C^d}
	.
\end{equation*}
Note that by definition of the set $\CC^\bot$,
the duality pairing
$\dual{\eqclass{\mu}{\nu}}{\varphi}_{C^1(\omega_1)}$
is well defined,
i.e.,
it is independent of the representative of the equivalence class $\eqclass{\mu}{\nu}$.

For a function $\varphi \colon \omega_1 \to \R$,
we denote by
\[
	\lip(\varphi)
	:=
	\sup\set*{
		\frac{\varphi(\eta_1) - \varphi(\eta_2)}{\abs{\eta_1 - \eta_2}}
		\given
		\eta_1, \eta_2 \in \omega_1, \eta_1 \ne \eta_2
	}
	\in
	[0,\infty]
\]
its smallest possible Lipschitz constant.
Due to the uniform local quasiconvexity of $\interior \omega_1$,
$C^1$ functions are automatically Lipschitz by the following lemma.

\begin{lemma}[{\cite[Lem.~5.4]{WachsmuthWalter2024}}]
	\label{lem:Lipschitz_constant_bounded_by_C1_norm}
	There exists a constant $C_L \ge 1$ such that
	\[
		\lip(\varphi)
		\le
		C_L \norm{\varphi}_{C^1(\omega_1)}
		\qquad
		\forall
		\varphi \in C^1(\omega_1)
		.
	\]
\end{lemma}

Using the Lipschitz constant, we can define the Lipschitz norm
\[
	\norm{\varphi}_\lip
	:=
	\max\set{
		\norm{\varphi}_C,
		\lip(\varphi)}
\]
and finally define the bounded Lipschitz norm on measures, i.e.,
\begin{equation}
	\norm{u}_\BL
	:=
	\sup
	\set{\dual{u}{\varphi}_C
		\given
		\varphi \in C(\omega_1),
		\norm{\varphi}_\lip \le 1
	}
	.
\end{equation}
Interestingly,
the bounded Lipschitz norm is equivalent with the norm of $C^1(\omega_1)\dualspace$ on measures.
\begin{lemma}[{\cite[Lem.~5.8]{WachsmuthWalter2024}}]
	\label{lem:BLnorm_C1*norm_equivalent}
	On $\MM(\omega_1)$,
	$\norm{\eqclass{\cdot}{0}}_{C^1(\omega_1)\dualspace}$ and
	$\norm{\cdot}_\BL$ are equivalent.
\end{lemma}

We lastly introduce a short notation.
\begin{definition}
	For $\rho \in L^2(\bar u; \R^d)$,
	we define
	$\eqclass{0}{\rho \bar u} \in C^1(\omega_1)\dualspace$
	via
	\begin{equation}
		\label{eq:short_notation_element_C1*}
		\dual{\eqclass{0}{\rho \bar u}}{\varphi}
		:=
		\int_{\omega_1}
			\rho(\eta)^\top \nabla\varphi(\eta)
		\d\bar u(\eta)
		\qquad\forall \varphi \in C^1(\omega_1)
		.
	\end{equation}
\end{definition}
Note that (up to boundary terms),
the functional
$\eqclass{0}{\rho \bar u}$
coincides with the distribution
$-\div(\rho \bar u)$.

\subsection{The lifted problem}
\label{subsec:lifted_problem}

After having introduced a characterization of $C^1(\omega_1)\dualspace$,
we can lift our problem \eqref{P} to this dual space.
First, we consider the canonical injection $\Pi : \MM(\omega_1) \to C^1(\omega_1)\dualspace$ which is defined via
\begin{equation}
	\Pi \mu := \eqclass{\mu}{0}
	.
\end{equation}
Note that by the Stone-Weierstraß theorem
$C^1(\omega_1)$ is dense in $C(\omega_1)$
and this
implies that $\Pi$ is injective.
The range of $\Pi$ can be characterized in the following way.

\begin{lemma}[{\cite[Lem.~5.13]{WachsmuthWalter2024}}]
	\label{lem:characterization_range_Pi}
	Let $\dualelement \in C^1(\omega_1)\dualspace$.
	Then,
	\[
		\dualelement \in \ran(\Pi)
		\qquad \Longleftrightarrow \qquad
		\exists C > 0:
		\forall \psi \in C^1(\omega_1):
		\abs{\dual{\dualelement}{\psi}} \le C \norm{\psi}_{C(\omega_1)}
		.
	\]
\end{lemma}

For the lifting, we have to extend the functionals $F$ and $D$ appearing in \eqref{P}.
We simply extend $D$ by $\infty$ outside of $\ran(\Pi)$.
To be precise, we define
the generalized transport distance $\wdist : C^1(\omega_1)\dualspace \to \bar\R$ by
\begin{equation}
	\label{eq:Dlifted}
	\wdist(\dualelement)
	:=
	\begin{cases}
		D(u)
		&\text{if } \dualelement \in \ran(\Pi) \text{ with } \dualelement = \Pi u,
		\\
		\infty
		&\text{otherwise.}
	\end{cases}
\end{equation}
Note that $\wdist$ is well defined since $\Pi$ is injective.
Similarly, $F$ is extended to obtain $\FF \colon C^1(\omega_1)\dualspace \to \R$,
the precise value of the extension is not important, since
$\FF + \frac\alpha2 \wdist \equiv \infty$ outside of $\ran(\Pi)$.
Consequently, the problem \eqref{P} is lifted to $C^1(\omega_1)\dualspace$
and we obtain
\begin{equation*}
	\tag{P'}
	\label{P'}
	\min_{\dualelement \in C^1(\omega_1)\dualspace}
	\FF(\dualelement) + \frac \alpha 2 \wdist(\dualelement)
	.
\end{equation*}
The following result states the equivalence to problem \eqref{P}.
\begin{theorem}
	The following characterization of minimizers of \eqref{P} and \eqref{P'} holds true.
	\begin{enumerate}
		\item
		Let $\bar u \in \MM(\omega_1)$ minimize \eqref{P}.
		Then, $\eqclass{\bar u}{0}$ minimizes \eqref{P'}.
		\item
		Let $\bar\dualelement \in C^1(\omega_1)\dualspace$ minimize \eqref{P'}.
		Then, there exists $\bar u \in \MM(\omega_1)$ such that $\bar\dualelement = \eqclass{\bar u}{0}$ and $\bar u$ minimizes \eqref{P}.
	\end{enumerate}
\end{theorem}
The proof follows directly from the definitions of $\FF$ and $\DD$.
For brevity of the notation,
we identify the fixed measure $\bar u \in \MM(\omega_1)$,
see \itemref{assu:standing_assumptions:5},
with its image under $\Pi$, i.e.,
\begin{equation}
	\label{eq:identification_bar_u}
	\bar u
	=
	\Pi \bar u
	=
	\eqclass{\bar u}{0}
	\qquad\text{in } C^1(\omega_1)\dualspace
	.
\end{equation}

\section{The non-degeneracy condition}
\label{sec:NDC}
In order to apply \cref{thm:abstract_second_order_result}
to our lifted problem \eqref{P'},
we have to verify \eqref{NDC}.
To this end, we apply \cref{lem:ndc_sufficient}
and the corresponding inequality is checked
via a duality approach
similar to \cite{WachsmuthWachsmuth2022}.
Therefore, we introduce the functional
$J \colon C^1(\omega_1) \to \R$
via
\begin{equation}\label{eq:predualobj}
	J(\psi)
	:=
	-\int_{\omega_0} \psi^{\bar c}(x) \d u_0(x)
	,
\end{equation}
where the $\bar c$-conjugate was defined in \eqref{eq:cconjugate}.
As shown in \cite[Lemma~A.2]{MeyerWachsmuth2025}, $J$ is convex and continuous.
We prove that the convex conjugate of $J$, i.e.,
$J\conjugate \colon C^1(\omega_1)\dualspace \to \bar\R$,
actually is
$\wdist$,
using ideas from \cite{WachsmuthWalter2024}.

\begin{theorem}
	\label{thm:J*_is_distance}
	The convex conjugate of $J$ is $\wdist$.
\end{theorem}
\begin{proof}
	Let $\dualelement \in C^1(\omega_1)\dualspace$ be given.
	We first consider the case $\dualelement \in C^1(\omega_1)\dualspace \setminus \ran(\Pi)$ which is the otherwise-case in \eqref{eq:Dlifted}.
	From \cref{lem:characterization_range_Pi}, we obtain functions $\psi_n \in C^1(\omega_1)$ such that
	$\abs{\dual{\dualelement}{\psi_n}} > n \norm{\psi_n}_{C(\omega_1)}$ for every $n \in \N$.
	Without loss of generality, $\abs{\dual{\dualelement}{\psi_n}} = \dual{\dualelement}{\psi_n}$ and $\norm{\psi_n}_{C(\omega_1)} = 1$.
	This implies
	\begin{align*}
		J\dualspace(\dualelement)
		&\ge
		\sup_{n \in \N}
		\parens*{
			\dual{\dualelement}{\psi_n}_{C^1(\omega_1)}
			+ \int_{\omega_0} \inf_{\eta \in \omega_1} c(x,\eta) - \psi_n(\eta) \d u_0(x)
		}
		\\
		&\ge
		\sup_{n \in \N}
		\parens*{
			n - u_0(\omega_0) \parens*{ \norm{c}_{C(\omega_0 \times \omega_1)} + 1 }
		}
		=
		\infty
		=
		\DD(\dualelement)
		.
	\end{align*}
	If otherwise $\dualelement \in \ran(\Pi)$, we have
	$\dualelement = \eqclass{u}{0}$ for some $u \in \MM(\omega_1)$.
	Consequently, we have
	$\dual{\eqclass{u}{0}}{\psi}_{C^1(\omega_1)} = \int_{\omega_1} \psi(\eta) \d u(\eta)$
	and thus
	\begin{align*}
		J\dualspace(\dualelement)
		&=
		\sup_{\psi \in C^1(\omega_1)}
		\parens*{
			\int_{\omega_1} \psi(\eta) \d u(\eta)
			+
			\int_{\omega_0} \psi^{\bar c}(x) \d u_0(x)
		}
		,
	\end{align*}
	which is almost the Kantorovich dual problem,
	in which the supremum is taken over $C(\omega_1)$
	instead of $C^1(\omega_1)$.
	However, the suprema coincide as $C^1(\omega_1)$ is dense in $C(\omega_1)$ by compactness of $\omega_1$
	and
	$C(\omega_1) \ni \psi \mapsto \int_{\omega_0} \psi^{\bar c}(x) \d u_0(x)$
	is continuous, see \cite[Lemma~A.2]{MeyerWachsmuth2025}.
	Finally, \itemref{lem:properties_D:1} yields
	\begin{equation*}
		J\dualspace(\dualelement)
		=
		D(u)
		=
		\DD(\eqclass{u}{0}).
	\end{equation*}
	This finishes the proof.
\end{proof}
This duality has an interesting consequence concerning the Kantorovich potentials.
Given $u \in \MM(\omega_1)$, a function $\psi \in C^1(\omega_1)$
is a Kantorovich potential if and only if
\begin{equation}
	\label{eq:Kantorovich_potential}
	\int_{\omega_1} \psi(\eta) \d u(\eta)
	=
	\DD(\eqclass{u}{0})
	+
	J(\psi)
	,
\end{equation}
see \eqref{eq:Kantorovich_potential_1}.
By the Fenchel--Young identity, this is
equivalent to
$\psi \in \partial\DD(\eqclass{u}{0})$
and to
$\eqclass{u}{0} \in \partial J(\psi)$.

The first part of this section is dedicated to the derivation of 
the (descent) inequality
\begin{equation*}
	J(\psi)
	\le
	J(\bar\varphi) + \dual{\bar u}{\psi - \bar\varphi}
	+ \frac{L}{2} \norm{\psi-\bar\varphi}_{C^1(\omega_1)}^2
	\qquad\forall \psi \in C^1(\omega_1)
\end{equation*}
for some $L \ge 0$,
where $\bar\varphi \in C^1(\omega_1)$ is a Kantorovich potential of $\bar u$.
Afterwards, arguing as in \cite[Lemma~4.9]{WachsmuthWachsmuth2022}
yields
\begin{equation*}
	J\conjugate(v)
	\ge
	J\conjugate(\bar u)
	+
	\dual{\bar\varphi}{v - \bar u}
	+
	\frac{1}{2 L} \norm{v - \bar u}_{C^1(\omega_1)\dualspace}^2
	\qquad\forall v \in C^1(\omega_1)\dualspace
	,
\end{equation*}
where we used the identification \eqref{eq:identification_bar_u}.
Then, \Cref{lem:ndc_sufficient} yields \eqref{NDC}.

The Kantorovich potential $\bar\varphi$ and perturbations thereof will play a prominent role 
in the upcoming analysis. 
It will be convenient to define a transport map $T(\psi)$
associated to an arbitrary $\psi \in C^1(\omega_1)$
that maps each $x \in \omega_0$ to a minimizer of
$c(x, \cdot) - \psi$ on $\omega_1$.
Note that a minimizer exists by the Weierstraß theorem but it does not have to be unique.
Of course, it is essential that the map $T(\psi)$
is measurable.
In our case, the existence of such a measurable selection
of the set-valued map
$\omega_0 \ni x \mapsto \argmin_{\eta \in \omega_1} c(x, \eta) - \psi(\eta)$
follows from \cite[Thm.~8.1.3,~Thm.~8.2.11]{AubinFrankowska2008}.

\begin{definition}
	\label{def:T_operator}
	Let $T : C^1(\omega_1) \to \set{ f : \omega_0 \to \omega_1 \given f \text{ is Borel-measurable} }$ be an operator that maps a function $\psi \in C^1(\omega_1)$
	to a measurable transport map,
	i.e.,
	$T(\psi)$ is Borel-measurable
	and $T(\psi)(x) \in \argmin_{\eta \in \omega_1} c(x,\eta) - \psi(\eta)$
	for all $x \in \omega_0$.
\end{definition}

Let us argue that
$T(\psi)$ is an optimal transport map from $u_0$
to $u(\psi) := T(\psi) \# u_0$.
The definition of $T(\psi)$ implies
\begin{equation*}
	\psi^{\bar c}(x) = c(x, T(\psi)(x)) - \psi(T(\psi)(x))
	\qquad\forall x \in \omega_0
	.
\end{equation*}
By moving the last term to the left-hand side and by integrating against $u_0$,
we get
\begin{align*}
	\int_{\omega_0} \psi^{\bar c} \d u_0
	+
	\int_{\omega_1} \psi(\eta) \d(u(\psi))(\eta)
	&=
	\int_{\omega_0} \psi^{\bar c} \d u_0
	+
	\int_{\omega_0} \psi(T(\psi)(x)) \d u_0(x)
	\\&
	=
	\int_{\omega_0} c(x, T(\psi)(x)) \d u_0(x)
	=
	\int_{\omega_0 \times \omega_1} c \d \pi(\psi),
\end{align*}
where $\pi(\psi) := (\id, T(\psi)) \# u_0$.
Now, \itemref{lem:properties_D:1}
implies that $\pi(\psi)$ and $\psi$
are primal and dual solutions of the Kantorovich problem
for $u(\psi)$.
Consequently, $T(\psi)$
is an optimal transport map from $u_0$ to $u(\psi)$
and $\eqclass{u(\psi)}{0} \in \partial J(\psi)$,
cf.\ \eqref{eq:Kantorovich_potential}.

For one Kantorovich potential associated to $\bar u$
(which is the fixed measure from \itemref{assu:standing_assumptions:5}),
we need some regularity assumption.

\begin{assumption}[Regularity assumption]
	\label{assu:regularity_assumptions_on_T}
	We assume that (associated to $\bar u$)
	there exists a Kantorovich potential
	$\bar\varphi \in C^1(\omega_1)$,
	a map $\bar T \colon \omega_0 \to \omega_1$
	and $\gamma > 0$
	such that
	\begin{equation}
		\label{eq:regularity}
		c(x,\eta) - \bar\varphi(\eta)
		\ge
		c(x,\bar T(x)) - \bar\varphi(\bar T(x)) + \frac \gamma 2 \abs*{\eta - \bar T(x)}^2
		\qquad
		\forall
		x \in \omega_0,
		\eta \in \omega_1
		.
	\end{equation}
\end{assumption}

Note that this assumption implies that $\bar T(x)$
is the unique minimizer of
$c(x,\cdot) - \bar\varphi$ on $\omega_1$ for all $x \in \omega_0$.
Consequently,
$\bar T = T(\bar \varphi)$
and, in particular, $\bar T$ is Borel measurable.
Similar to the argument above, we verify $\bar u = \bar T \# u_0$.
Since $\bar\varphi$ is a Kantorovich potential associated with $\bar u$,
\eqref{eq:Kantorovich_potential} implies
\begin{equation*}
	\int_{\omega_1} \bar\varphi \d \bar u
	=
	D(\bar u)
	-
	\int_{\omega_0} \bar\varphi^{\bar c} \d u_0.
\end{equation*}
If we denote by $\bar\pi \in \Gamma(u_0, \bar u)$
an optimal transport plan, the previous identity implies
\begin{equation*}
	\int_{\omega_0 \times \omega_1}
	c(x, \eta) - ( \bar\varphi^{\bar c}(x) + \bar\varphi(\eta) ) \d \bar\pi(x,\eta)
	=
	0.
\end{equation*}
The definition of the $\bar c$-conjugate implies that the integrand is
nonnegative.
Since the integrand is also continuous,
this implies
\begin{equation*}
	\supp(\bar\pi)
	\subset
	\set{
		(x,\eta) \in \omega_0 \times \omega_1
		\given
		c(x, \eta) - ( \bar\varphi^{\bar c}(x) + \bar\varphi(\eta) ) = 0
	}
	=
	\set{ (x, \bar T(x)) \given x \in \omega_0 },
\end{equation*}
where the latter identity follows from
the fact that $\bar T(x)$ is the unique minimizer of $c(x,\cdot) - \bar\varphi$.
Consequently, $\bar\pi = (\id, \bar T) \# u_0$
and $\bar T$ is the unique optimal transport map from $u_0$ to $\bar u$.

We mention that the continuity of $\bar T$
follows from \eqref{eq:regularity} by a simple argument.
\begin{lemma}
	\label{lem:continuity_bar_T}
	Under \cref{assu:regularity_assumptions_on_T}, $\bar T$ is continuous.
\end{lemma}
\begin{proof}
	For $x_1, x_2 \in \omega_0$,
	we apply \eqref{eq:regularity}
	with $x = x_i$ and $\eta = \bar T(x_{3-i})$.
	Adding the resulting inequalities yields
	\begin{equation*}
		c(x_1, \bar T(x_2))
		+
		c(x_2, \bar T(x_1))
		\ge
		c(x_1, \bar T(x_1))
		+
		c(x_2, \bar T(x_2))
		+
		\gamma \abs{\bar T(x_1) - \bar T(x_2)}^2
		.
	\end{equation*}
	Since the continuous function $c$
	is uniformly continuous on the compact set
	$\omega_0 \times \omega_1$,
	this implies continuity of $\bar T$.
\end{proof}

In the next result,
we give a possibility to verify \cref{assu:regularity_assumptions_on_T}.
\begin{lemma}
	\label{lem:Hessian_spd_uniformly_implies_regularity}
	Let $\omega_1$ be convex and
	let $c \in C^2(\omega_0 \times \omega_1)$ and $\bar\varphi \in C^2(\omega_1)$.
	Further, for all $x \in \omega_0$,
	the minimizer of $\omega_1 \ni \eta \mapsto c(x,\eta) - \bar\varphi(\eta)$ shall be unique
	and it is denoted by $\bar T(x)$.
	Additionally,
	we assume that for all $x \in \omega_0$
	the Hessian
	$\nabla_{\eta\eta} c(x,\bar T(x)) - \nabla^2 \bar\varphi(\bar T(x))$
	is positive definite.
	Then, \Cref{assu:regularity_assumptions_on_T} holds.
\end{lemma}
\begin{proof}
	First, we prove that $\bar T$ is continuous.
	Let $\seq{x_n} \subset \omega_0$ with
	$x_n \to x$ be given.
	Then, by compactness of $\omega_0$,
	we can extract a subsequence (without renaming) such that $\bar T(x_n) \to \bar\eta$.
	By optimality of the $\bar T(x_n)$ we get
	\begin{equation*}
		c(x_n, \bar T(x_n)) - \bar\varphi(\bar T(x_n))
		\le
		c(x_n, \eta) - \bar\varphi(\eta)
		\qquad\forall \eta \in \omega_1.
	\end{equation*}
	Passing to the limit $n \to \infty$
	shows that $\bar\eta$ is a minimizer of $c(x, \cdot) - \bar\varphi$
	and, consequently,
	$\bar\eta = \bar T(x)$.
	Since the limit is independent of the chosen subsequence,
	a usual subsequence-subsequence argument yields continuity of $\bar T$.

	This continuity of $\bar T$ implies the existence of $\mu > 0$
	such that the minimal eigenvalue of the Hessian satisfies
	$\lambda_{\textup{min}}(\nabla_{\eta\eta} c(x,\bar T(x)) - \nabla^2 \bar\varphi(\bar T(x))) \ge \mu$
	for all $x \in \omega_0$.

	Next, we prove
	that \eqref{eq:regularity} holds for all $x \in \omega_0$ and all $\eta \in \omega_1 \cap B_\varepsilon(\bar T(x))$ for some $\varepsilon > 0$.
	Therefore, we define the functions
	$f : \omega_0 \times \omega_1 \to \R$ and
	$g : \omega_0 \times \omega_1 \to \R^{d \times d}$ via
	$f(x,\eta) := c(x,\eta) - \bar\varphi(\eta)$ and
	$g(x,\eta) := \nabla_{\eta\eta}^2 c(x,\eta) - \nabla^2 \bar\varphi(\eta)$.
	The function $g$ is uniformly continuous as $\omega_0 \times \omega_1$ is compact.
	Consequently,
	there exists $\varepsilon > 0$ such that
	$\lambda_{\textup{min}}( g(x, \eta ) ) \ge \frac{\mu}{2}$
	for all $x \in \omega_0$ and $\eta \in \omega_1 \cap B_\varepsilon(\bar T(x))$.

	Since $\bar T(x)$ is the unique minimizer of $f(x, \cdot)$
	on the convex set $\omega_1$,
	the first-order optimality condition implies
	\begin{equation*}
		\nabla_\eta f(x, \bar T(x))^\top (\eta - \bar T(x))
		\ge
		0 
		.
	\end{equation*}
	Together with
	a Taylor expansion,
	we obtain
	for some $\eta_x$ between $\eta$ and $\bar T(x)$
	\begin{align*}
		\MoveEqLeft
		f(x,\eta)
		-
		f(x,\bar T(x))
		\\
		&=
		\nabla_\eta f(x,\bar T(x))^\top (\eta - \bar T(x))
		+
		\frac 1 2 (\eta - \bar T(x))^\top g(x,\eta_x) (\eta - \bar T(x))
		\\
		&\ge
		0
		+
		\frac \mu 4 \abs{\eta - \bar T(x)}^2
		.
	\end{align*}
	This shows that \eqref{eq:regularity}
	holds for every $x \in \omega_0$ and all $\eta \in \omega_1 \cap B_\varepsilon(\bar T(x))$.

	It remains to extend this inequality to all $\eta \in \omega_1$.
	For that, we first show that there exists some $\delta > 0$ such that $f(x,\eta) \ge f(x,\bar T(x)) + \delta$ for all $x \in \omega_0$ and all $\eta \in \omega_1 \setminus B_\varepsilon(\bar T(x))$.

	We define the set $U := \set{(x,\eta) \in \omega_0 \times \omega_1) \given \abs{\bar T(x) - \eta} \ge \varepsilon}$.
	The set $U$ is closed (using continuity of $\bar T$) and bounded, thus compact.
	The function 
	$\mathdh : U \to \R, \mathdh(x,\eta) := f(x,\eta) - f(x,\bar T(x))$ is continuous
	and positive, since $\bar T(x)$ is the unique minimizer of $f(x,\cdot)$.
	Consequently it is uniformly positive,
	i.e., there exists $\delta > 0$ such that
	$\mathdh(x, \eta) \ge \delta$
	for all $(x,\eta) \in U$.
	
	Finally, we define $\gamma := \min(\mu/2, 2\delta/\diam(\omega_1)^2)$.
	As seen above, the wanted inequality in \eqref{eq:regularity} holds 
	with this value of $\gamma$
	for $x \in \omega_0$ and
	$\eta \in \omega_1 \cap B_\varepsilon(\bar T(x))$.
	On the other hand,
	if $\eta \in \omega_1 \setminus B_\varepsilon(\bar T(x))$,
	we have
	\[
		f(x,\eta) - f(x,\bar T(x))
		=
		\mathdh(x,\eta)
		\ge
		\delta
		\ge
		\frac{\delta}{\diam(\omega_1)^2} \abs{\eta - \bar T(x)}^2
		\ge
		\frac{\gamma}{2} \abs{\eta - \bar T(x)}^2
		.
		\qedhere
	\]
\end{proof}

As an auxiliary result, we show that
\cref{assu:regularity_assumptions_on_T} implies some stability
of the objective values of minimizers.
\begin{lemma}
	\label{lem:regularity_implies_nice}
	Let the Kantorovich potential $\bar\varphi \in C^1(\omega_1)$ satisfy \Cref{assu:regularity_assumptions_on_T}.
	For all
	$\psi \in C^1(\omega_1)$
	and
	all $x \in \omega_0$,
	every minimizer $\eta_\psi^x \in \omega_1$
	of $c(x,\cdot) - \psi$
	satisfies
	\begin{equation}
		\label{eq:perturbed_minimizer_inequalities}
		\abs*{
			\bar\varphi(\bar T(x)) - \bar\varphi(\eta_\psi^x)
			- \psi(\bar T(x)) + \psi(\eta_\psi^x)
		}
		\le
		\frac {2 C_L^2} \gamma \norm{\psi - \bar\varphi}_{C^1(\omega_1)}^2
		,
	\end{equation}
	where $C_L$ is the constant from \cref{lem:Lipschitz_constant_bounded_by_C1_norm}.
\end{lemma}
\begin{proof}
	The definition of $\eta_\psi^x$ and \eqref{eq:regularity} yield the inequalities
	\begin{align*}
		c(x,\eta_\psi^x) - \bar\varphi(\eta_\psi^x)
		&\ge
		c(x,\bar T(x)) - \bar\varphi(\bar T(x)) + \frac \gamma 2 \abs{\eta_\psi^x - \bar T(x)}^2
		,
		\\
		c(x,\bar T(x)) - \psi(\bar T(x))
		&\ge
		c(x,\eta_\psi^x) - \psi(\eta_\psi^x)
		.
	\end{align*}
	We add up these inequalities and arrive at
	\begin{align*}
		\frac \gamma 2 \abs{\eta_\psi^x - \bar T(x)}^2
		&\le
		\bar\varphi(\bar T(x)) - \bar\varphi(\eta_\psi^x)
		- \psi(\bar T(x)) + \psi(\eta_\psi^x)
		\\&
		=
		(\bar\varphi - \psi)(\bar T(x)) - (\bar\varphi - \psi)(\eta_\psi^x)
		\\&
		\le
		\lip( \bar\varphi - \psi) \abs{ \eta_\psi^x - \bar T(x) }
		\\&
		\le
		C_L \norm{ \bar\varphi - \psi }_{C^1(\omega_1)} \abs{ \eta_\psi^x - \bar T(x) }
		,
	\end{align*}
	where we used \Cref{lem:Lipschitz_constant_bounded_by_C1_norm} at the end.
	This implies
	$ \frac \gamma 2 \abs{\eta_\psi^x - \bar T(x)} \le C_L \norm{ \bar\varphi - \psi }_{C^1(\omega_1)}$
	and the claim follows.
\end{proof}

With this result at hand,
we can show the desired inequality.
\begin{lemma}
	\label{lem:descent_lemma}
	Let the Kantorovich potential $\bar\varphi \in C^1(\omega_1)$ satisfy \Cref{assu:regularity_assumptions_on_T}.
	Then, the descent condition
	\begin{equation}
		\label{eq:descent_inequality}
		J(\psi)
		\le
		J(\bar\varphi) + \dual{\bar u}{\psi - \bar\varphi}
		+ \frac {2 C_L^2 u_0(\omega_0)} \gamma \norm{\psi-\bar\varphi}_{C^1(\omega_1)}^2
	\end{equation}
	holds for all $\psi \in C^1(\omega_1)$.
\end{lemma}
\begin{proof}
	We recall that \cref{assu:regularity_assumptions_on_T} implies
	$\bar u = \bar T \# u_0$, see above.
	Note that
	\begin{equation*}
		\dual{\bar u}{\psi - \bar\varphi}
		=
		\int_{\omega_1} (\psi - \bar\varphi)(\eta) \d \bar u(\eta)
		=
		\int_{\omega_0} (\psi - \bar\varphi)(\bar T(x)) \d u_0(x).
	\end{equation*}
	For every $x \in \omega_0$,
	let $\eta_\psi^x$ denote a minimizer of $c(x,\cdot) - \psi$.
	Then,
	by
	$\psi^{\bar c}(x) = c(x,\eta_\psi^x) - \psi(\eta_\psi^x)$ and
	$\bar\varphi^{\bar c}(x) \le c(x,\eta_\psi^x) - \bar\varphi(\eta_\psi^x)$,
	we obtain
	\begin{align*}
		-\psi^{\bar c}(x) + \bar\varphi^{\bar c}(x) - (\psi - \bar \varphi)(\bar T(x))
		&\le
		\psi(\eta_\psi^x) - \bar\varphi(\eta_\psi^x) - (\psi - \bar \varphi)(\bar T(x))
		\\&
		\le
		\frac {2 C_L^2} \gamma \norm{\psi - \bar\varphi}_{C^1(\omega_1)}^2
		,
	\end{align*}
	where we used \cref{lem:regularity_implies_nice} in the last inequality.
	Finally, we arrive at
	\begin{align*}
		J(\psi) - J(\bar\varphi) - \dual{\bar u}{\psi - \bar\varphi}
		&=
		\int_{\omega_0}
		-\psi^{\bar c}(x) + \bar\varphi^{\bar c}(x) - (\psi - \bar \varphi)(\bar T(x))
		\d u_0(x)
		\\
		&\le
		\frac {2 C_L^2} \gamma \norm{\psi - \bar\varphi}_{C^1(\omega_1)}^2 u_0(\omega_0)
		.
		\qedhere
	\end{align*}
\end{proof}

The \Cref{assu:regularity_assumptions_on_T} seems a little bit too strong,
as it is often enough to have a unique minimizer for $u_0$-a.a.\ $x \in \omega_0$.
Thus,
we weaken \Cref{assu:regularity_assumptions_on_T}
by only demanding quadratic growth of $c(x,\cdot) - \bar \varphi$
on a uniform ball around the minimizers $\bar T(x)$;
and demanding the function value to be above the minimal value plus some threshold everywhere else
while controlling that threshold in some way.
Note that this is similar to a frequently used regularity assumption
for bang-bang optimal control,
see
\cite[(2.10)]{DeckelnickHinze2012},
\cite[(3.3)]{WachsmuthWachsmuth2009}.

\begin{assumption}[Weaker regularity assumption]
	\label{assu:weaker_regularity_assumptions_on_T}
	We assume that (associated to $\bar u$)
	there exists a Kantorovich potential
	$\bar\varphi \in C^1(\omega_1)$,
	a Borel-measurable map $\bar T \colon \omega_0 \to \omega_1$,
	constants $\varepsilon, \gamma, C > 0$
	and a function
	$\delta \colon \omega_0 \to [0,\infty)$
	such that
	for all $x \in \omega_0$
	with $\bar\eta := \bar T(x)$
	we have
	\begin{subequations}
		\label{eq:weak_regularity}
		\begin{align}
			\label{eq:weak_regularity:1}
			c(x,\eta) - \bar\varphi(\eta)
			&\ge
			c(x,\bar\eta) - \bar\varphi(\bar\eta)
			+
			\frac \gamma 2 \abs*{\eta - \bar\eta}^2
			&& \forall \eta \in \omega_1 \cap B_\varepsilon(\bar\eta),
			\\
			\label{eq:weak_regularity:2}
			c(x,\eta) - \bar\varphi(\eta)
			&\ge
			c(x,\bar\eta) - \bar\varphi(\bar\eta) + \delta(x)
			&& \forall \eta \in \omega_1 \setminus B_\varepsilon(\bar\eta),
			\\
			\label{eq:weak_regularity:3}
			u_0(\set{x \in \omega_0 \given \delta(x) \le \kappa})
			&\le C \kappa
			&&\forall \kappa \ge 0.
		\end{align}
	\end{subequations}
\end{assumption}
We remark that
the minimizer of $c(x,\cdot) - \bar \varphi$ is not unique
if $\delta(x) = 0$,
however, this can only happen on a $u_0$-null set.
In case that the function $\delta$ is uniformly positive,
we recover \cref{assu:regularity_assumptions_on_T},
see the last part of the proof of \cref{lem:Hessian_spd_uniformly_implies_regularity}.

The next result shows that the assertion of \cref{lem:descent_lemma}
holds under the weaker regularity assumption
(with a different constant).
\begin{lemma}
	\label{lem:weak_descent_lemma}
	Let $\bar\varphi \in C^1(\omega_1)$ satisfy
	\Cref{assu:weaker_regularity_assumptions_on_T}.
	Then, there exists $\tilde C > 0$ such that
	\begin{equation}
		\label{eq:weak_descent_inequality}
		J(\psi)
		\le
		J(\bar\varphi) + \dual{\bar u}{\psi - \bar\varphi}
		+
		\frac{\tilde C}{2} \norm{\psi - \bar\varphi}_{C^1(\omega_1)}^2
		\qquad
		\forall \psi \in C^1(\omega_1)
		.
	\end{equation}
\end{lemma}
\begin{proof}
	Let $\psi \in C^1(\omega_1)$ be arbitrary.
	We define $\kappa := 3 \norm{\psi - \bar\varphi}_{C(\omega_1)}$
	and $M_\kappa := \set{x \in \omega_0 \given \delta(x) \le \kappa}$.
	For all $x \in \omega_0 \setminus M_\kappa$,
	we have $\delta(x) \ge \kappa$
	and thus for every $\eta \in \omega_1 \setminus B_\varepsilon(\bar T(x))$
	\eqref{eq:weak_regularity:2} yields
	\begin{align*}
		c(x,\bar T(x)) - \psi(\bar T(x))
		&
		\le
		c(x,\bar T(x)) - \bar\varphi(\bar T(x)) + \frac \kappa 3
		\\
		&\le
		c(x,\eta) - \bar\varphi(\eta) - \frac{2\kappa}3
		\le
		c(x,\eta) - \psi(\eta) - \frac{\kappa}3
		,
	\end{align*}
	i.e., $\eta$ cannot be a minimizer of $c(x, \cdot) - \psi$.
	Hence, every minimizer $\eta_\psi^x$ of $c(x,\cdot) - \psi$ belongs to $B_\varepsilon(\bar T(x))$.
	Thus, we can employ \eqref{eq:weak_regularity:1}
	for $\eta_\psi^x$ and, consequently,
	we obtain that \eqref{eq:perturbed_minimizer_inequalities} still holds.
	By arguing as in the proof of
	\cref{lem:descent_lemma},
	we find
	\begin{align*}
		\int_{\omega_0 \setminus M_\kappa}
		-\psi^{\bar c}(x) + \bar\varphi^{\bar c}(x) - (\psi - \bar \varphi)(\bar T(x))
		\d u_0(x)
		\le
		\frac {2 C_L u_0(\omega_0)} \gamma \norm{\psi - \bar\varphi}_{C^1(\omega_1)}^2
		.
	\end{align*}
	
	Next, we consider $x \in M_\kappa$.
	The definition of the $\bar c$-conjugate implies
	\begin{equation*}
		\bar\varphi^{\bar c}(x) \le c(x, \eta_\psi^x) - \bar\varphi(\eta_\psi^x)
		\qquad\text{and}\qquad
		\psi^{\bar c}(x) = c(x, \eta_\psi^x) - \psi(\eta_\psi^x)
		,
	\end{equation*}
	where $\eta_\psi^x$ is a minimizer of $c(x,\cdot) - \psi$.
	Consequently,
	\begin{align*}
		-\psi^{\bar c}(x) + \bar\varphi^{\bar c}(x)
		-
		(\psi - \bar\varphi)(\bar T(x))
		&\le
		(\psi - \bar\varphi)(\eta_\psi^x)
		-
		(\psi - \bar\varphi)(\bar T(x))
		\\&
		\le
		2 \norm{\psi - \bar\varphi}_{C(\omega_1)}
		.
	\end{align*}
	Integration and the application of \eqref{eq:weak_regularity:3} yield
	\begin{align*}
		\MoveEqLeft
		\int_{M_\kappa}
		-\psi^{\bar c}(x) + \bar\varphi^{\bar c}(x)
		-
		(\psi - \bar\varphi)(\bar T(x))
		\d u_0(x)
		\le
		2 \norm{\psi - \bar\varphi}_{C(\omega_1)} u_0(M_\kappa)
		\\
		&\le
		2 \norm{\psi - \bar\varphi}_{C(\omega_1)} C \kappa
		\le
		6 C \norm{\psi - \bar\varphi}_{C^1(\omega_1)}^2
		.
	\end{align*}
	Together with the inequality on $\omega_0 \setminus M_\kappa$,
	the claim follows.
\end{proof}

Under the (weak) regularity assumption,
we can characterize the subdifferential of $J$.
\begin{corollary}[Subdifferential of $J$]
	\label{lem:subdiff_J}
	Let \cref{assu:weaker_regularity_assumptions_on_T} be satisfied.
	Then,
	$\partial J(\bar\varphi) = \set{\bar u} \subset C^1(\omega)\dualspace$,
	see \eqref{eq:identification_bar_u},
	and $J$ is Gâteaux-differentiable at $\bar\varphi$.

	Moreover, for every $\varphi \in C^1(\omega_1)$,
	we have
	$\eqclass{T(\varphi) \# u_0}{0} \in \partial J(\varphi)$.
	In case that $c(x,\eta) = \abs{x - \eta}^p$ with $p \in (1,\infty)$
	and $u_0$ being absolutely continuous w.r.t.\ the Lebesgue measure on $\omega_0$,
	$\partial J(\varphi) = \set{\eqclass{T(\varphi) \# u_0}{0}}$.
\end{corollary}
\begin{proof}
	Let $\varphi \in C^1(\omega_1)$ be arbitrary.
	We already have seen that $\varphi$ is a Kantorovich potential for $u := T(\varphi) \# u_0$
	and $\eqclass{u}{0} \in \partial J(\varphi)$,
	see the discussion after \cref{def:T_operator}.
	This yields the second claim.

	To see the first claim,
	we first note that
	$\bar u = \eqclass{\bar u}{0} \in \partial J(\bar\varphi)$
	due to the first part of the proof.
	Consequently, we have
	$J'(\bar\varphi; \psi) \geq \dual{\bar u}{\psi}_{C(\omega_1)}$
	for all $\psi \in C^1(\omega_1)$. On the other hand, \eqref{eq:weak_descent_inequality} implies 
	$J'(\bar\varphi; \psi) \leq \dual{\bar u}{\psi}_{C(\omega_1)}$ and thus 
	$J'(\bar\varphi; \psi) = \dual{\bar u}{\psi}_{C(\omega_1)}$
	for all $\psi \in C^1(\omega_1)$.
	This implies that $J$ is Gâteaux differentiable at $\bar\varphi$
	and, consequently, the uniqueness of the subgradient.

	Finally, in the case of $c(x,\eta) = \abs{x - \eta}^p$ and $u_0$ being absolutely continuous 
	w.r.t.\ the Lebesgue measure,
	the cost functional $D$ is strictly convex, cf.\ \itemref{lem:properties_D:2}.
	
		This implies uniqueness of subgradients of $J$ as follows.
		Let $\varphi \in C^1(\omega_1)$ and $\Phi_1, \Phi_2 \in \partial J(\varphi)$ be given.
		This implies $\varphi \in \partial \DD(\Phi_1) \cap \partial \DD(\Phi_2)$.
		In particular, this shows $\Phi_1, \Phi_2 \in \dom(\DD)$,
		i.e., $\Phi_1 = \Pi u_1$ and $\Phi_2 = \Pi u_2$ for some measures $u_1, u_2 \in \MM(\omega_1)$.
		Consequently,
		\begin{equation*}
			D(u)
			=
			\DD(\Pi u)
			\ge
			\DD(\Phi_i) + \dual{ \Pi u - \Phi_i }{ \varphi }_{C^1(\omega_1)}
			=
			D(u_i) + \dual{ u - u_i }{ \varphi }_{C(\omega_1)}
		\end{equation*}
		holds for any measure $u \in \MM(\omega_1)$ and all $i \in \set{1,2}$.
		This implies that $D$ is an affine function on the line segment from $u_1$ to $u_2$.
		The strict convexity of $D$ implies that $u_1 = u_2$.
\end{proof}

Using \cref{lem:weak_descent_lemma},
we can show the satisfaction of the crucial inequality which is needed for
the verification of the non-degeneracy condition,
cf.\ \cref{lem:ndc_sufficient}.
\begin{theorem}
	\label{thm:NDC_D_inequality}
	Let \Cref{assu:weaker_regularity_assumptions_on_T} be satisfied.
	Then,
	\begin{equation*}
		\wdist(u)
		\ge
		\wdist(\bar u) + \dual{u - \bar u}{\bar\varphi}
		+
		\frac{1}{2 \tilde C} \norm{u - \bar u}_{C^1(\omega_1)\dualspace}^2
		\qquad
		\forall u \in C^1(\omega_1)\dualspace
		.
	\end{equation*}
\end{theorem}
\begin{proof}
	As $\wdist = J\dualspace$ by \cref{thm:J*_is_distance},
	$J$ is Gâteaux-differentiable at $\bar\varphi$
	and
	the descent inequality for $J$ holds by \Cref{lem:weak_descent_lemma},
	we can argue exactly as in \cite[Lemma~4.9]{WachsmuthWachsmuth2022}
	in order to obtain the ascent property.
\end{proof}

\begin{remark}\label{rem:strongerassu}
    We emphasize that the above theorem clearly also holds under the stronger
    \cref{assu:regularity_assumptions_on_T}. Thus, at a fist glance, 
    \cref{assu:regularity_assumptions_on_T} now seems to be superfluous in our context, 
    and this is indeed true as long as one is only interested in the non-degeneracy condition 
    \eqref{NDC}. For the computation of a second directional derivative of $J$ 
    (which will ultimately allow us to compute the weak-$\star$ second subderivative of $D$)
    however, we need to resort to the stronger \cref{assu:regularity_assumptions_on_T}, 
    see \cref{rem:J''} below.
\end{remark}

Finally,
we present an example which indicates that the space $C^1(\omega_1)$
cannot be replaced by less regular spaces
in the analysis above.

\begin{example}
	\label{ex:counter_example_inequality}
	In this counterexample, we show that an inequality of type
	\begin{equation*}
		J(\psi) - J(\bar\varphi) - \dual{\bar u}{\psi - \bar\varphi}
		\le
		C \norm{\psi - \bar\varphi}_{Y}^2
	\end{equation*}
	cannot hold
	even in very regular situations
	for the spaces $Y = C(\omega_1)$,
	$Y = C^{0,\alpha}(\omega_1)$ in case $\alpha \in (0,1)$,
	and $Y = W^{1,q}(\omega_1)$ in case $q \le d$ and $d \ge 2$.
	We consider $c(x,\eta) = \abs{x - \eta}^2$,
	$\bar\varphi \equiv 0$,
	$\omega_0 = \omega_1$ is an arbitrary compact set 
	with positive Lebesgue measure
	(satisfying our standing assumption)
	and $u_0$ is the scaled Lebesgue measure on $\omega_0$,
	i.e., $u_0(\omega_0) = 1$.
	In view of $\bar T(x) = \argmin_{\eta \in \omega_1} |x - \eta|^2 - \bar\varphi(\eta)$,
	it is clear that $x = \bar T(x)$,
	i.e., $\bar u = u_0$,
	and
	\Cref{assu:regularity_assumptions_on_T} is satisfied (with equality) with $\gamma = 2$.
	
	We start by dealing with the simplest case $Y = C(\omega_1)$.
	Let $P \subset \omega_1$ be a finite set, such that
	$\dist(P, \cdot)$ is bounded on $\omega_1$ by a small constant $r > 0$.
	For constants $\varepsilon = 2 r^2$ and $M > 0$, we consider the function
	$\psi \in C(\omega_1)$ defined via
	\begin{equation*}
		\psi(\eta) := \max\set{0, \varepsilon - M \dist(P, \eta) }.
	\end{equation*}
	Note that $\norm{\psi}_{C(\omega_1)} = \varepsilon$.
	For an arbitrary $x \in \omega_0$,
	let $p \in P$ be a projection of $x$ onto $P$.
	Then
	$ \psi^{\bar c}(x) = \inf_{\eta \in \omega_1} c(x,\eta) - \psi(\eta) \le c(x,p) - \psi(p) \le r^2 - \varepsilon = -\varepsilon / 2$.
	Consequently,
	\begin{align*}
		\MoveEqLeft
		J(\psi) - J(\bar\varphi) - \dual{\bar u}{\psi - \bar\varphi}
		\\
		&=
		\int_{\omega_0}
		-\psi^{\bar c}(x)
		+
		\bracks{ c(x, \bar T(x)) - \bar\varphi(\bar T(x)) }
		-
		(\psi - \bar\varphi)(\bar T(x))
		\d u_0(x)
		\\
		&\ge
		\int_{\omega_0}
		-\bracks{ r^2 - \varepsilon }
		-
		\psi(\bar T(x))
		\d u_0(x)
		=
		\frac \varepsilon 2
		-
		\int_{\omega_0}
		\psi(x)
		\d u_0(x)
		.
	\end{align*}
	By choosing $M > 0$ large enough,
	the integral over $\psi$ becomes smaller than $\varepsilon / 4$.
	This shows that
	\begin{equation*}
		J(\psi) - J(\bar\varphi) - \dual{\bar u}{\psi - \bar\varphi}
		\le
		C \norm{\psi - \bar\varphi}_{C(\omega_1)}^2
	\end{equation*}
	cannot hold, since the left-hand side is bounded from below by $\varepsilon / 4$,
	the right-hand side is of order $\varepsilon^2$,
	and $\varepsilon = 2 r^2$ can be made arbitrarily small.

	In order to handle the Hölder continuous functions
	$Y = C^{0,\alpha}(\omega_1)$, $\alpha \in (0,1)$,
	we have to be a little bit more careful concerning the choice of $P$.
	In fact, it is possible to select $P \subset \omega_1$ as above
	such that it contains at most $C_d r^{-d}$ points. 
	Here, $C_d$ is some constant (depending on the dimension and on the fixed set $\omega_1$).
	Consequently,
	the graph of the function $\psi$ (defined as above)
	consists of (parts of) $C_d r^{-d}$ small cones
	with height $\varepsilon = 2 r^2$
	and a spherical base with radius $\varepsilon / M$.
	The volume of each cone is $\hat C_d \varepsilon (\varepsilon/M)^{d}$
	(again, $\hat C_d$ depends on the dimension $d$).
	If there is an overlap among cones,
	the integral only becomes smaller,
	thus
	\begin{equation*}
		\int_{\omega_0} \psi \d u_0
		\le
		C_d r^{-d} \hat C_d \varepsilon (\varepsilon / M)^{d}
		=
		C_d \hat C_d 2^{d/2}
		\varepsilon^{-d/2 + 1 + d} M^{- d}
		=:
		C \varepsilon^{d/2 + 1} M^{ - d}
	\end{equation*}
	By choosing $M := \tilde C \varepsilon^{1/2} := \sqrt[d]{4C} \varepsilon^{1/2}$ we can achieve
	\begin{equation*}
		J(\psi) - J(\bar\varphi) - \dual{\bar u}{\psi - \bar\varphi}
		\ge
		\frac\varepsilon2
		-
		\int_{\omega_0} \psi \d u_0
		\ge
		\frac \varepsilon 2
		-
		C \varepsilon^{d/2 + 1} M^{ - d}
		=
		\frac\varepsilon4
		.
	\end{equation*}
	It remains to estimate the Hölder norm of $\psi$.
	It can be checked that
	\begin{equation*}
		\psi(\eta)
		=
		\max\set*{
			\max\set{ 0, \varepsilon - M \abs{\eta - p} }
			\given
			p \in P
		}
	\end{equation*}
	and the maximum of Hölder continuous functions is
	again Hölder continuous
	with Hölder constant bounded by the maximum of the Hölder constants.
	Further, it is easy to check that the Hölder constant of the function $\psi_p \colon \R^d \to \R$,
	defined via
	$\psi_p(\eta) = \max\set{ 0, \varepsilon - M \abs{\eta - p} }$, is
	\begin{equation*}
		\abs{\psi_p}_{C^{0,\alpha}(\R^d)}
		=
		\sup\set*{
			\frac{\psi_p(\eta_1) - \psi_p(\eta_2)}{\abs{\eta_1 - \eta_2}^\alpha}
			\given
			\eta_1 \ne \eta_2
		}
		=
		\frac{ \varepsilon - 0 }{ \abs{\varepsilon / M}^\alpha }
		=
		\tilde C^\alpha \varepsilon^{1-\alpha/2}
		.
	\end{equation*}
	Consequently,
	$\abs{\psi}_{C^{0,\alpha}(\R^d)}$ is bounded by the same constant
	and
	\begin{equation*}
		\norm{\psi}_{C^{0,\alpha}(\omega_1)}
		=
		\max\set*{ \norm{\psi}_{C(\omega_1)}, \abs{\psi}_{C^{0,\alpha}(\omega_1)} }
		=
		\tilde C^\alpha \varepsilon^{1-\alpha/2}
	\end{equation*}
	for $\varepsilon > 0$ small enough.
	This implies
	that the inequality
	\begin{equation*}
		J(\psi) - J(\bar\varphi) - \dual{\bar u}{\psi - \bar\varphi}
		\le
		C \norm{\psi - \bar\varphi}_{C(\omega_1)}^2
	\end{equation*}
	cannot hold since the left-hand side is bounded from below by $\varepsilon/4$
	while the right-hand side is of order $\varepsilon^{2 - \alpha}$
	with
	$2 - \alpha > 1$ due to $\alpha \in (0,1)$.

	Finally, we address the case of
	Sobolev spaces.
	We assume that $d \ge 2$ and that $\omega_1 \subset \R^d$ is regular enough,
	i.e., the interior $\interior(\omega_1)$ is a Lipschitz domain
	and $\omega_1$ equals the closure of its interior.
	We set $Y = W^{1,q}(\omega_1) := W^{1,q}(\interior(\omega_1))$ with $q \le d$.
	Next, we fix $p \in \interior(\omega_1)$
	and denote by $D$ the diameter of $\omega_1$.
	Due to $q \le d$,
	a standard construction
	similar to \cite[Section~9.3, Remark~16]{Brezis2011}
	yields a function $\psi \in C_c^\infty(\interior(\omega_1)) \subset W^{1,q}(\omega_1)$
	with arbitrary small $\norm{\psi}_{W^{1,q}(\omega_1)}$
	such that $\psi(p) = 2 D^2$ and $0 \le \psi \le 2 D^2$ on $\omega_1$.
	By arguing as above, we find
	$ \psi^{\bar c}(x) = \inf_{\eta \in \omega_1} c(x,\eta) - \psi(\eta) \le c(x,p) - \psi(p) \le -D^2$
	and
	\begin{equation*}
		J(\psi) - J(\bar\varphi) - \dual{\bar u}{\psi - \bar\varphi}
		\ge
		\int_{\omega_0}
		-\psi^{\bar c}(x)
		-
		\psi (\bar T(x))
		\d u_0(x)
		\ge
		D^2
		-
		\int_{\omega_0}
		\psi (x)
		\d u_0(x)
		.
	\end{equation*}
	Note that the integral of $\psi$
	can be bounded (up to a constant) by its $W^{1,q}(\omega_1)$ norm.
	Thus, we see again that
	\begin{equation*}
		J(\psi) - J(\bar\varphi) - \dual{\bar u}{\psi - \bar\varphi}
		\le
		C \norm{ \psi - \bar\varphi }_{W^{1,q}(\omega_1)}^2
	\end{equation*}
	cannot hold, since 
	$\norm{ \psi - \bar\varphi }_{W^{1,q}(\omega_1)} = \norm{ \psi }_{W^{1,q}(\omega_1)}$
	can be made arbitrarily small.
\end{example}

\section{Computation of the weak-\texorpdfstring{$\star$}{*} second subderivative and epi-differentiability}
\label{sec:second_subderivative}

The aim of this section is to compute the weak-$\star$ second subderivative of the transport cost $\DD$ and to show its weak-$\star$ epi-differentiability.
To this end, we follow an idea of \cite{ChristofWachsmuth2019},
which is based on the following result.

\begin{lemma}[{\cite[Lem.~3.2]{ChristofWachsmuth2019}, \cite[Lem.~2.8]{BorchardWachsmuth2024}}]
	\label{lem:density_lemma}
	Let $x \in \dom(G)$ and $w \in Y$ be given.
	We suppose the existence of a set $V \subset X$
	and a functional $Q\colon X \to [-\infty,\infty]$ such that
	\begin{enumerate}
		\item\label{lem:density_lemma:1}
		for all $h \in X$ we have $G''(x,w;h) \ge Q(h)$,
		\item\label{lem:density_lemma:2}
		for all $h \in V$ and all $\seq{t_k} \subset \R^+$ with $t_k \to 0$,
		there exists a sequence $\seq{h_k} \subset X$ satisfying
		$h_k \weaklystar h$, $\norm{h_k}_X \to \norm{h}_X$, and
		\begin{align*}
			Q(h) = \lim_{k \to \infty} \frac{G(x + t_k h_k) - G(x) - \dual{w}{h_k}}{t_k^2 / 2} \in [-\infty,\infty],
		\end{align*}
		\item\label{lem:density_lemma:3}
		for all $h \in X$
		with $Q(h) < \infty$
		there exists a sequence $\seq{h^l} \subset V$
		with $h^l \weaklystar h$,
		$\norm{h^l}_X \to \norm{h}_X$
		and $Q(h) \ge \liminf_{l \to \infty} Q(h^l)$.
	\end{enumerate}
	Then, $Q = G''(x,w;\cdot)$
	and $G$ is strictly twice epi-differentiable at $x$ for $w$.
\end{lemma}

As we will see below, we cannot directly apply this result, since we were not able 
to establish the existence of sequences $\seq{h_k}$, $\seq{h^l}$ fulfilling
the convergence of norms required in 
\itemref{lem:density_lemma:2} and \ref{lem:density_lemma:3}.
Nevertheless, it turns out that the weak-$\star$ convergence
of the sequences
together with some uniform bound on the norms
is sufficient for our purpose and \cref{lem:density_lemma} serves as a good guide for the 
computation of the weak-$\star$ second subderivative of $\DD$.
A crucial ingredient in \cref{lem:density_lemma} is the functional $Q$ which provides a lower bound.
We construct this functional via the pre-conjugate functional $J$.
To be precise, we
verify that $J$ is twice directionally Hadamard differentiable at $\bar\varphi$, see \Cref{subsec:J''},
and
\cite[Lemma~4.4]{WachsmuthWachsmuth2022} will provide us with the function $Q$.
Afterwards,
the weak-$\star$ second subderivative of $\DD$ is calculated in \cref{subsec:W''}.

Throughout this section,
we make use of the following assumption.

\begin{assumption}
	\label{assu:varphireg}
	The transportation cost $c$ is twice continuously differentiable on $\omega_0 \times \omega_1$. 
	The optimal control satisfies $\supp(\bar u) \subset \interior(\omega_1)$ and 
	the Kantorovich potential $\bar\varphi \in C^1(\omega_1)$ is twice continuously differentiable on
	$\supp(\bar u) + B_{\hat \varepsilon}(0) \subset \interior(\omega_1)$ for some $\hat \varepsilon > 0$.
	Moreover, we assume that the regularity condition \Cref{assu:regularity_assumptions_on_T} is satisfied
	for $\bar\varphi$.
\end{assumption}

Since $\bar T$ is continuous due to
\cref{lem:continuity_bar_T},
we can apply
\cite[Lemma~B.3]{MeyerWachsmuth2025}
and obtain
\begin{equation}
	\label{eq:support}
	\supp(\bar u)
	=
	\bar T(\supp(u_0))
	=
	\bar T(\omega_0)
	,
\end{equation}
where we used
$\bar u = \bar T \# u_0$ and
\itemref{assu:standing_assumptions:4}.

\subsection{Computing \texorpdfstring{$J''$}{J''}}
\label{subsec:J''}
We return to the pre-conjugate function $J$ from \eqref{eq:predualobj}.
	We already have seen in \cref{lem:subdiff_J}
	that $J$ is Gâteaux differentiable at $\bar\varphi$
	and $\bar u = J'(\bar\varphi)$.
	In what follows, we show that $J$ is even
	\emph{twice directionally Hadamard differentiable}
	at $\bar\varphi$
	in the sense
	that the limit
	\begin{equation*}
		J''(\bar\varphi; \psi)
		:=
		\liminf_{k\to\infty}
		\frac{J(\bar \varphi + t_k \psi_k) - J(\bar \varphi) - t_k \dual{\bar u}{\psi_k}}{t_k^2 /2}
	\end{equation*}
	exists for all sequences
	$\seq{t_k} \subset \R^+$ and $\seq{\psi_k} \subset C^1(\omega_1)$
	such that
	$t_k \to 0$ and $\psi_k \to \psi$ in $C^1(\omega_1)$.
	Note that
	the notation slightly differs from
	the weak-$\star$ second subderivative
	introduced in \cref{def:wsss},
	which would be denoted as $J''(\bar\varphi, \bar u; \cdot)$.

The term $J(\bar\varphi + t_k \psi_k)$ contains the $\bar c$-conjugate
of $\bar\varphi + t_k \psi_k$, see \eqref{eq:predualobj}.
Hence, we have to investigate the perturbed optimization problems
$\min_{\eta \in \omega_1} c(x,\eta) - (\bar\varphi + t_k \psi_k)(\eta)$.
For this,
the following lemma will be essential.

\begin{lemma}
	\label{lem:sensitivity}
    Let $\omega \subset \R^d$ be compact and $f \in C^1(\omega)$.
    Suppose that $\bar\eta \in \interior(\omega)$ is the unique solution of 
    \begin{equation*}
        \min_{\eta \in \omega} \; f(\eta).
    \end{equation*}
    Additionally,
    we assume that $f$ is
    twice differentiable at $\bar\eta$ and
    strongly convex with constant $\hat\gamma > 0$ on 
    $B_\varepsilon(\bar\eta) \subset \interior(\omega)$ for some $\varepsilon > 0$, i.e., 
    \begin{equation}\label{eq:funiconv}
        \dual{\nabla f(\eta_1) - \nabla f(\eta_2)}{\eta_1 - \eta_2}
        \geq
        \hat\gamma \, \abs{\eta_1 - \eta_2}^2
        \quad
        \forall\, \eta_1, \eta_2 \in B_\varepsilon(\bar\eta).
    \end{equation}
    Finally, we assume the existence of a constant $\delta > 0$ such that 
    \begin{equation}\label{eq:xiglobopt}
        f(\eta) \ge f(\bar\eta) + \delta
        \quad
        \forall\, \eta \in \omega \setminus B_\varepsilon(\bar\eta).
    \end{equation}
    We define the set-valued mapping $S \colon C^1(\omega) \rightrightarrows \omega$
    via
    \begin{equation*}
        S(v)
        :=
        \argmin_{\eta \in \omega} v(\eta) \subset \omega
        \qquad\forall v \in C^1(\omega)
        .
    \end{equation*}
    Then,
    $S$ is calm at $f$,
    i.e.,
    for arbitrary $v \in C^1(\omega)$ with $\norm{v}_{C^1(\omega)} < \delta/2$
    and arbitrary $\eta \in S(f + v)$ it holds
    \begin{equation}
    	\label{eq:calm_solution_operator}
    	\abs{\eta - \bar \eta}
    	\le
			\frac 1 {\hat\gamma}
    	\norm{v}_{C^1(\omega)}
    	.
    \end{equation}
    Further, let a null sequence $\seq{t_k}  \subset \R^+$ and
    a sequence $\seq{h_k} \subset C^1(\omega)$ with $h_k \to h$ in $C^1(\omega)$ be given 
    and, for every $k\in \N$,
    pick an arbitrary $\eta_k \in S(f + t_k h_k)$.
    Then,
    \begin{equation}\label{eq:taylorsensitivity}
        \eta_k = \bar\eta + t_k\, \theta + \oo(t_k)
        \qquad\text{as } k \to \infty
    \end{equation}
    with $\theta = - \nabla^2 f(\bar\eta)^{-1} \nabla h(\bar\eta)$.
\end{lemma}
\begin{proof}
    Clearly, $S(v) \neq \emptyset$ for all $v\in C^1(\omega)$ due to the compactness of $\omega$.
    To prove the calmness of $S$ at $f$,
    let $v \in C^1(\omega)$ with $\norm{v}_{C^1(\omega)} < \delta/2$ be given. 
    Then, thanks to \eqref{eq:xiglobopt}, we have
    \begin{equation*}
    \begin{aligned}
        (f + v)(\eta)
        &\geq f(\eta) - \norm{v}_{C(\omega)} 
        \geq
        f(\bar\eta) + \delta - \norm{v}_{C(\omega)} \\
        &\geq
        (f+v)(\bar\eta) + \delta - 2 \norm{v}_{C(\omega)}
        >
        (f+v)(\bar\eta)
        \quad
        \forall\, \eta \in \omega \setminus B_\varepsilon(\bar\eta).
    \end{aligned}
    \end{equation*}
    Therefore, $S(f+v) \subset B_\varepsilon(\bar\eta) \subset \interior(\omega)$.
    Consequently, an arbitrary $\eta \in S(f+v)$
    satisfies the first-order necessary optimality condition,
    i.e., $\nabla (f+v)(\eta) = 0$.
    Together with $\nabla f(\bar\eta) = 0$ and
    \eqref{eq:funiconv} we get
    \begin{equation*}
      \hat\gamma \abs{\eta - \bar\eta}^2
      \le
      \dual{\nabla f(\eta) - \nabla f(\bar\eta)}{\eta - \bar\eta}
      =
      \dual{-\nabla v(\eta)}{\eta - \bar\eta}
      \le
      \norm{v}_{C^1(\omega)} \abs{\eta - \bar\eta}.
    \end{equation*}
    This implies the calmness \eqref{eq:calm_solution_operator}.
    
    Consider now an arbitrary null sequence $\seq{t_k} \subset \R^+$ and an arbitrary sequence $\seq{h_k} \subset C^1(\omega)$ with $h_k \to h$ in $C^1(\omega)$, 
    as in the statement of the lemma. For the rest of the proof, we assume w.l.o.g.\  that $\norm{t_k h_k}_{C^1(\omega)} < \delta/2$ for all $k \in \N$.
    For each $k\in \N$, pick an arbitrary $\eta_k \in S(f + t_k h_k)$. Then the calmness implies
    \begin{equation}\label{eq:estdiffquot}
        \frac{\abs{\eta_k - \bar\eta}}{t_k}
        \leq
        \frac{1}{\hat\gamma} \norm{h_k}_{C^1(\omega)}
        \leq
        \widetilde C
        <
        \infty
        \quad
        \forall \, k\in \N
    \end{equation}
    and thus there exists a converging subsequence, denoted by the same symbol to ease notation, i.e., 
    $(\eta_k - \bar\eta)/t_k \to \theta$ for some $\theta \in \R^d$. Furthermore, since $\eta_k \to \bar\eta \in \interior(\omega)$, we have $\eta_k \in \interior(\omega)$ for 
    $k\in \N$ sufficiently large and consequently, the first-order optimality conditions for $\bar\eta$ and $\eta_k$ read $\nabla f(\bar\eta) = 0$ and 
    $\nabla (f+t_k h_k)(\eta_k) = 0$. This leads to 
    \begin{equation*}
        \frac{\nabla f(\eta_k) -  \nabla f(\bar\eta)}{t_k}
        =
        -\nabla h_k(\eta_k)
    \end{equation*}
    and the twice differentiability of $f$ at $\bar\eta$ implies
    \begin{equation}\label{eq:tayloreta}
        \nabla^2 f(\bar\eta)\,\frac{\eta_k - \bar\eta}{t_k}
        +
        \frac{\oo(\abs{\eta_k - \bar\eta})}{t_k}
        =
        -\nabla h_k(\eta_k) .
    \end{equation}
    Now, since $\abs{\eta_k - \bar\eta} \leq \tilde C t_k$ and $\nabla h_k \to \nabla h$  in $C(\omega;\R^d)$, we can pass to the limit $k \to \infty$ and obtain
    \begin{equation}\label{eq:etagleichung}
        \nabla^2 f(\bar\eta) \theta = - \nabla h(\bar\eta)
        .
    \end{equation}
    From \eqref{eq:funiconv} it follows that the Hessian $\nabla^2 f(\bar\eta)$ is positive definite.
    Thus, the limit $\theta$ is uniquely determined by \eqref{eq:etagleichung}.
    Consequently,
    a subsequence-subsequence argument shows that
    the entire sequence of difference quotients $\{(\eta_k - \bar\eta)/t_k\}$ converges to $\theta$, which finally gives the result.
\end{proof}

\begin{remark}
	\label{rem:S_Hadamard}
    Note that, if $S$ were single-valued, then the above proof just shows that $S$ is Hadamard differentiable in $\bar\eta$.
\end{remark}

In the next result, we show that \cref{assu:varphireg}
allows to apply \cref{lem:sensitivity}
pointwise.

\begin{lemma}
	\label{lem:assumptions_imply_main_lemma's_assumptions}
	Let \Cref{assu:varphireg} be satisfied.
	Then,
	there exist $\hat\gamma, \varepsilon, \delta > 0$
	with $\varepsilon \in (0,\hat\varepsilon]$
	such that
	for every $x \in \omega_0$,
	the assumptions of \Cref{lem:sensitivity}
	hold for the function
	$c(x,\cdot) - \bar \varphi$
	with these constants $\hat\gamma, \varepsilon, \delta$.
	Moreover, there exists a constant $C_H$,
	such that the norm of the Hessian
	$\nabla^2_{\eta\eta}c(x,\eta) - \nabla^2\bar\varphi(\eta)$
	is bounded by $C_H$
	for all $x \in \omega_0$ and $\eta \in B_\varepsilon(\bar T(x))$.
\end{lemma}
\begin{proof}
	Let $x \in \omega_0$ be arbitrary.
	Since $\bar \eta := \bar T(x) \in \supp(\bar u)$,
	see \eqref{eq:support},
	\cref{assu:varphireg}
	implies that
	$c(x, \cdot) - \bar\varphi$
	is twice continuously differentiable in $\bar \eta$ and 
	\eqref{eq:regularity} yields
	that the smallest eigenvalue of
	its Hessian
	$\nabla^2_{\eta\eta} c(x,\bar\eta) - \nabla^2\bar\varphi(\bar\eta)$
	is at least $\gamma$, where we also used that 
	$\nabla_{\eta} c(x,\bar\eta) - \nabla\bar\varphi(\bar\eta) = 0$ due to $\supp(\bar u) \subset \interior(\omega_1)$ 
	by \cref{assu:varphireg}.
	For any $\hat\gamma \in (0,\gamma)$,
	we can choose $\varepsilon \in (0,\hat\varepsilon]$
	such that the smallest eigenvalue of
	$\nabla^2_{\eta\eta} c(x,\eta) - \nabla^2\bar\varphi(\eta)$
	is at least $\hat\gamma$
	for all $x \in \omega_0$ and $\eta \in B_\varepsilon(\bar T(x))$
	due to uniform continuity of the Hessian on the compact set $\omega_0 \times \omega_1$.
	In particular,
	\eqref{eq:funiconv} is satisfied by
	$f = c(x, \cdot) - \bar\varphi$
	for all $x \in \omega_0$.
	Moreover, \eqref{eq:regularity}
	gives that \eqref{eq:xiglobopt} is fulfilled for $\delta = \frac \gamma 2 \varepsilon^2$ and the uniqueness of the minimizer.
	The upper bound for the Hessian follows
	since
	$\omega_0 \times ( \supp(\bar u) + B_\varepsilon(0))$
	is a compact set and the Hessian is continuous on this set
	due to \cref{assu:varphireg}.
\end{proof}

Before we use \Cref{lem:sensitivity} to calculate $J''$,
we first introduce a substitution rule
that allows us to rewrite an integral w.r.t.\ $u_0$
as an integral w.r.t.\ $\bar u$.
We recall that the definition of the push-forward
implies
\begin{equation*}
	\int_{\omega_0} f(\bar T(x)) \d u_0(x)
	=
	\int_{\omega_1} f(\eta) \d \bar u(\eta)
\end{equation*}
for all $\bar u$-integrable functions $f \colon \omega_1 \to \R$,
since $\bar u = \bar T \# u_0$.
However, this does not work if the integrand in the former integral would also depend on $x$.
For this, one can employ the celebrated disintegration theorem from \cite[Thm.~5.3.1]{AmbrosioGigliSavare2005}.
This ensures the existence of a family $\set{\hat u_\eta}_{\eta \in \omega_1}$
of probability measures on $\omega_0$,
such that
$\eta \mapsto \hat u_\eta(B)$ is measurable for all Borel sets $B \subset \omega_0$,
$\hat u_\eta( \omega_0 \setminus \bar T^{-1}(\set{\eta})) = 0$ for $\bar u$-a.e.\ $\eta \in \omega_1$
and
\begin{equation*}
	\int_{\omega_0} f(x) \d u_0(x)
	=
	\int_{\omega_1} \int_{\bar T^{-1}(\set{\eta})} f(x) \d \hat u_\eta(x) \d \bar u(\eta)
\end{equation*}
for all Borel measurable $f \colon \omega_0 \to [0,\infty]$.
In particular, if $f \colon \omega_0 \times \omega_1 \to [0,\infty]$
is Borel measurable, this implies
\begin{equation}
	\label{eq:disintegration}
	\begin{aligned}
		\int_{\omega_0} f(x, \bar T(x)) \d u_0(x)
		&=
		\int_{\omega_1} \int_{\bar T^{-1}(\set{\eta})} f(x, \bar T(x)) \d \hat u_\eta(x) \d \bar u(\eta)
		\\
		&=
		\int_{\omega_1} \int_{\bar T^{-1}(\set{\eta})} f(x, \eta) \d \hat u_\eta(x) \d \bar u(\eta)
		\\
		&=
		\int_{\omega_1} F(\eta) \d \bar u(\eta)
	\end{aligned}
\end{equation}
with
$F(\eta) := \int_{\bar T^{-1}(\set{\eta})} f(x, \eta) \d \hat u_\eta(x)$.
In a certain sense,
the measure $\hat u_\eta$ indicates which points in $\omega_0$
have been transported to the point $\eta$
and
$F(\eta)$ is a weighted average over those points.

\begin{theorem}
	\label{thm:J''}
	Let \Cref{assu:varphireg} be true.
	Then,
	$J$ is twice directionally Hadamard differentiable at $\bar\varphi$ with
	\begin{equation}
		\label{eq:strong_second_subderivative_J}
		J''(\bar\varphi; \psi)
		=
		\int_{\omega_1}
		\nabla \psi(\eta)^\top
		H(\eta)^{-1}
		\nabla \psi(\eta)
		\d \bar u(\eta)
		\qquad\forall \psi \in C^1(\omega_1)
		,
	\end{equation}
	where
	$H \colon \omega_1 \to \R^{d \times d}$
	defined via
	\begin{equation}
		\label{eq:H}
		H(\eta)
		:=
		\parens*{
			\int_{\bar T^{-1}(\set{\eta})}
			\big[\nabla^2_{\eta\eta} c(y, \eta) - \nabla^2 \bar\varphi(\eta)\big]^{-1}
			\d \hat u_\eta(y)
		}^{-1}
	\end{equation}
	is a bounded Borel-measurable map
	and the family $\set{\hat u_\eta}_{\eta \in \omega_1}$ of probability measures on $\omega_0$
	is provided by the disintegration theorem (see above).
	For $\bar u$-a.a.\ $\eta \in \omega_1$, the eigenvalues of $H$ are in $[\gamma,C_H]$.
\end{theorem}
\begin{proof}
	Let sequences $\seq{t_k} \subset \R^+$
	and $\seq{\psi_k} \subset C^1(\omega)$
	with
	$t_k \to 0$ and
	$\psi_k \to \psi$ in $C^1(\omega_1)$ be given.
	Then,
	\eqref{eq:predualobj} implies
	\begin{equation*}
		\frac{J(\bar\varphi + t_k \psi_k) - J(\bar\varphi) - t_k \dual{\bar u}{\psi_k}}{t_k^2 / 2}
		=
		\int_{\omega_0}
		\frac{
			-( \bar\varphi + t_k \psi_k)^{\bar c}
			+\bar\varphi^{\bar c}
			-
			t_k (\psi_k \circ \bar T)
		}{t_k^2/2}
		\d u_0
		.
	\end{equation*}
	We prove the convergence of this integral via the dominated convergence theorem.
	To this end,
	let $x \in \omega_0$ be arbitrary.
	We define $\bar\eta := \bar T(x)$,
	which is the unique minimizer of $c(x,\cdot) - \bar\varphi$,
	and denote by $\eta_k$
	an arbitrary minimizer of $c(x, \cdot) - (\bar\varphi + t_k \psi_k)$,
	i.e., we have
	$\eta_k \in S(c(x,\cdot) - \bar\varphi - t_k \psi_k)$
	with $S$ from \cref{lem:sensitivity}.
	Then, the definition of the $\bar c$-conjugate, see \eqref{eq:cconjugate},
	implies that the value of the above integrand at $x$ is precisely
	\begin{equation*}
		\frac{
			- c(x, \eta_k) + (\bar\varphi + t_k \psi_k)(\eta_k)
			+ c(x, \bar\eta) - \bar\varphi(\bar\eta)
			- t_k \psi_k(\bar\eta)
		}{t_k^2/2}
		=
		I_{k}(x) + 2 J_k(x)
		,
	\end{equation*}
	where
	\begin{equation*}
		I_{k}(x) :=
		\frac{
			- c(x, \eta_k) + \bar\varphi(\eta_k)
			+ c(x, \bar\eta) - \bar\varphi(\bar\eta)
		}{t_k^2/2}
		\quad\text{and}\quad
		J_k(x) :=
		\frac{
			\psi_k(\eta_k)
			- \psi_k(\bar\eta)
		}{t_k}
		.
	\end{equation*}

	\Cref{lem:assumptions_imply_main_lemma's_assumptions} lets us apply
	\Cref{lem:sensitivity} to the function $f = c(x, \cdot) - \bar\varphi$ and the direction $h = -\psi$.
	In turn, we get
	\begin{equation}
		\label{eq:expansion_eta_k}	
		\eta_k = \bar \eta + t_k \theta(x) + \oo(t_k) 
		\quad \text{with} \quad 
		\theta(x) = [\nabla^2_{\eta\eta} c(x, \bar \eta) - \nabla^2 \bar\varphi(\bar \eta)]^{-1} \nabla \psi(\bar \eta).
	\end{equation}
	We mention that the $\oo(t_k)$ depends on $x$,
	but this will not cause any problems, since we only need pointwise convergence for
	the application of Lebesgue's dominated convergence theorem below.
	Let us denote by $C > 0$ an upper bound on $\norm{\psi_k}_{C^1(\omega_1)}$.
	Note that $C$ is independent of $x$.
	Next, we pick $k_0$ such that
	$t_k \le \min\set{\delta/(3 C), (\varepsilon \hat\gamma)/C }$
	for all $k \ge k_0$.
	In particular, $\norm{t_k \psi_k}_{C^1(\omega_1)} \le \delta/3 < \delta/2$
	for all $k \ge k_0$
	allows us to apply
	\eqref{eq:calm_solution_operator} with $v = - t_k \psi_k$
	and we obtain
	\begin{equation}
		\label{eq:etak_to_eta_uniformly}
		\abs{\eta_k - \bar \eta}
		\le
		\frac {t_k \norm{\psi_k}_{C^1(\omega_1)}}{\hat\gamma}
		\le
		\frac{C t_k}{\hat\gamma}
		\le
		\varepsilon
		\qquad\forall k \ge k_0
		.
	\end{equation}
	For $J_k$, this implies the uniform bound
	\begin{equation*}
		\abs{J_k(x)}
		\le
		\frac{ \lip( \psi_k ) \abs{\eta_k - \bar\eta} }{t_k}
		\le
		\frac{C_L \norm{\psi_k}_{C^1(\omega_1)} C}{\hat\gamma}
		\le
		\frac{C_L C^2}{\hat\gamma}
		\qquad\forall k \ge k_0
		,
	\end{equation*}
	where we used \cref{lem:Lipschitz_constant_bounded_by_C1_norm}.
	Moreover, we have the pointwise convergence
	\begin{equation*}
		J_k(x)
		=
		\frac{
			\psi(\eta_k)
			- \psi(\bar\eta)
		}{t_k}
		+
		\frac{
			(\psi_k - \psi)(\eta_k)
				- (\psi_k - \psi)(\bar\eta)
		}{t_k}
		\to
		\nabla \psi(\bar\eta)^\top \theta(x),
	\end{equation*}
	due to \eqref{eq:expansion_eta_k},
	since the second addend can be bounded
	by $\lip(\psi_k - \psi) \abs{\eta_k - \eta}/t_k \to 0$.
	To treat $I_{k}$, we note that the optimality condition for $\bar\eta \in \interior(\omega_1)$
	reads
	$\nabla \bar\varphi(\bar \eta) - \nabla_\eta c(x, \bar \eta) = 0$.
	Consequently,
	a second-order Taylor expansion
	and \eqref{eq:expansion_eta_k} yield
	\begin{align*}
		I_{k}(x)
		&
		=
		\frac{
			\bar\varphi(\eta_k)
			- \bar\varphi(\bar\eta)
			- \nabla\bar\varphi(\bar\eta)^\top (\eta_k - \bar\eta)
			- c(x, \eta_k)
			+ c(x, \bar\eta)
			+ \nabla_\eta c(x,\bar\eta)^\top (\eta_k - \bar\eta)
		}{t_k^2/2}
		\\&
		\to
		\theta(x)^\top
		\parens*{ \nabla^2\bar\varphi(\bar\eta) - \nabla^2_{\eta\eta}c(x,\bar\eta) }
		\theta(x)
		=
		-\theta(x)^\top
		\parens*{ \nabla^2_{\eta\eta}c(x,\bar\eta) - \nabla^2\bar\varphi(\bar\eta) }
		\theta(x)
	\end{align*}
	as $k \to \infty$.
	It remains to check the boundedness of $I_k(x)$.
	By \cref{assu:varphireg} and \cref{lem:assumptions_imply_main_lemma's_assumptions},
	we have $B_\varepsilon(\bar\eta) \subset B_{\hat\varepsilon}(\bar\eta) \subset \interior(\omega_1)$
	and \eqref{eq:etak_to_eta_uniformly} implies
	$\eta_k \in B_\varepsilon(\bar\eta)$ for all $k \ge k_0$.
	By compactness and continuity of $\nabla^2 \bar\varphi - \nabla^2_{\eta\eta}c$,
	we find a constant $C_H \ge 0$ (independent of $x$)
	such that the matrix norm of
	$\nabla^2 \bar\varphi(\eta) - \nabla^2_{\eta\eta}c(x,\eta)$
	is bounded by $C_H$
	for all $x \in \omega_0$ and all $\eta \in B_\varepsilon(\bar T(x))$.
	Consequently, a second-order Taylor estimate
	and \eqref{eq:etak_to_eta_uniformly} imply
	\begin{equation*}
		\abs{ I_k(x) }
		\le
		\frac{ C_H \abs{\eta_k - \eta}^2 }{t_k^2}
		\le
		C_H \parens*{ \frac{C}{\hat\gamma} }^2
	\end{equation*}
	for all $k \ge k_0$.
	Since $u_0$ is a finite measure,
	the uniform bound for $I_k + 2 J_k$ is integrable.
	Consequently, the dominated convergence theorem applies.
	By inserting the equation for $\theta(x)$ in the limiting expression,
	we obtain
	\begin{align*}
		\MoveEqLeft
		\lim_{k \to \infty}
		\frac{J(\bar\varphi + t_k \psi_k) - J(\bar\varphi) - t_k \dual{\bar u}{\psi_k}}{t_k^2 / 2}
		\\
		&=
		\int_{\omega_0}
		-\theta(x)^\top
		\parens*{ \nabla^2_{\eta\eta}c(x,\bar T(x)) - \nabla^2\bar\varphi(\bar T(x)) }
		\theta(x)
		+ 2
		\nabla \psi(\bar T(x))^\top \theta(x)
		\d u_0(x)
		\\
		&=
		\int_{\omega_0}
		\nabla\psi(\bar T(x))^\top
		\parens*{ \nabla^2_{\eta\eta}c(x,\bar T(x)) - \nabla^2\bar\varphi(\bar T(x)) }^{-1}
		\nabla\psi(\bar T(x))
		\d u_0(x)
		.
	\end{align*}

	Finally, we want to rewrite this integral as an integral w.r.t.\ $\bar u$.
	As explained above, this can be achieved by the disintegration theorem
	which provides the family of probability measures $\set{\hat u_\eta}_{\eta \in \omega_1}$.
	Since the above integrand is nonnegative,
	we can argue as in \eqref{eq:disintegration}.
	This yields
	\begin{align*}
		\MoveEqLeft
		\int_{\omega_0}
		\nabla \psi(\bar T(x))^\top
		\big[\nabla^2_{\eta\eta} c(x, \bar T(x)) - \nabla^2 \bar\varphi(\bar T(x))\big]^{-1}
		\nabla \psi(\bar T(x))
		\d u_0(x)
		\\&
		=
		\int_{\omega_1}
		\nabla \psi(\eta)^\top
		H(\eta)^{-1}
		\nabla \psi(\eta)
		\d \bar u(\eta)
		,
	\end{align*}
	with $H$ as in \eqref{eq:H}.
	In \cref{lem:assumptions_imply_main_lemma's_assumptions}
	we have seen
	that the eigenvalues of
	$\nabla^2_{\eta\eta} c(y, \eta) - \nabla^2\varphi(\eta)$
	are in $[\gamma, C_H]$
	for all
	$\eta \in \supp(\bar u)$
	and
	$y \in \bar T^{-1}(\set{\eta})$.
	Consequently, the integrand in the definition of $H$
	is bounded and
	the measurability of $H$ follows from \cite[(5.3.1)]{AmbrosioGigliSavare2005}.
	Since $\hat u_\eta$ is a probability measure,
	we infer that the eigenvalues of $H$
	also belong to $[\gamma, C_H]$
	for $\bar u$-a.a.\ $\eta \in \omega_1$.
	This shows the claim.
\end{proof}

\begin{remark}\label{rem:J''}
    We were not able to prove this theorem under the weaker assumptions in
    \Cref{assu:weaker_regularity_assumptions_on_T}. The current proof depends
    on the calmness \eqref{eq:calm_solution_operator}
    of the solution mapping from \Cref{lem:sensitivity} which
    (under the weaker assumption)
    only works if $\norm{v}_{C^1(\omega_1)} < \delta(x)/2$ which depends on $x$.
    In the setting of the proof of \cref{thm:J''},
    the number $k_0$ then also depends on $x$
    and this invalidates the uniform upper bound on the integrand.
    Some preliminary computations show
    that the structure of $J''(\bar\varphi; \cdot)$
    (if it exists)
    is much more complicated under
    the weaker assumption
    and this is subject to future research.
\end{remark}

In the important case of the quadratic cost, i.e., $c(x,\eta) = \abs{x - \eta}^2$,
we have $\nabla^2_{\eta\eta} c(y, \eta) = 2 \Id$,
which is independent of $y$.
In this case, formula \eqref{eq:H} simplifies to
\begin{equation}
	\label{eq:H_simple}
	H(\eta)
	=
	2 \Id - \nabla^2 \bar\varphi(\eta)
	.
\end{equation}
Consequently, we have the simple expression
\begin{equation*}
	J''(\bar\varphi; \psi)
	=
	\int_{\omega_1}
	\nabla \psi(\eta)^\top
	\parens*{
		2 \Id - \nabla^2 \bar\varphi(\eta)
	}^{-1}
	\nabla \psi(\eta)
	\d \bar u(\eta)
\end{equation*}
for the second directional derivative.
A similar simplification is valid whenever the partial Hessian
$\nabla^2_{\eta\eta} c(y, \eta)$ is independent of $y$.

In the next section,
we will see that the weak-$\star$ second subderivative
$\frac 12 \wdist''(\bar u, \bar\varphi; \cdot)$
is the convex conjugate of
$\frac 12 J''(\bar\varphi; \cdot)$.
Since these are quadratic functionals,
we get
\begin{equation*}
	\frac 12 J''(\bar\varphi; \psi)
	=
	\frac 12 \dual{\psi\dualspace}{\psi}
	=
	\frac 12 \wdist(\bar u, \bar\varphi; \psi\dualspace)
	,
\end{equation*}
whenever $\psi\dualspace \in \frac12 \partial J''(\bar\varphi; \cdot)(\psi)$.
Thus, we characterize this subdifferential.
\begin{corollary}
	\label{lem:Subdiff_Jpp_third_component}
	The functional $\frac 1 2 J''(\bar\varphi; \cdot)$ is convex
	and
	its subdifferential is
	\begin{equation}
		\label{eq:Hinvgradpsidu}
		\frac 1 2 \partial J''(\bar\varphi; \cdot)(\psi)
		=
		\set{\psi\dualspace},
		\quad\text{where}\quad
		\psi\dualspace := \eqclass{0}{(H^{-1} \nabla \psi) \bar u}
	\end{equation}
	for all $\psi \in C^1(\omega_1)$,
	cf. \eqref{eq:short_notation_element_C1*}.
\end{corollary}
\begin{proof}
	This follows directly from \eqref{eq:strong_second_subderivative_J}
	and the properties of $H$.
\end{proof}

\subsection{Computing \texorpdfstring{$\wdist''$}{D''}}
\label{subsec:W''}

Having calculated $J''(\bar\varphi; \cdot)$, we can employ
\cite[Lemma~4.4]{WachsmuthWachsmuth2022}
in order to get a lower bound for the weak-$\star$ second subderivative of $\wdist$
at $\bar u = \eqclass{\bar u}{0}$
for $\bar\varphi$.
This lower bound corresponds to \itemref{lem:density_lemma:1}.

\begin{lemma}
	\label{lem:lower_bound_subderivative}
	Let \Cref{assu:varphireg} be satisfied.
	Then,
	$\frac12 \wdist''(\bar u, \bar\varphi; \cdot)
	\ge \parens*{ \frac12 J''(\bar\varphi;\cdot)}\dualspace$
	and
	\begin{equation*}
		\parens*{ \frac12 J''(\bar\varphi;\cdot)}\dualspace(\dualelement)
		=
		\inf\set*{
			\frac 1 2 \int_{\omega_1} \rho^\top H \rho \d \bar u
			\given
			\rho \in L^2(\bar u; \R^d),
			\dualelement = \eqclass{0}{\rho \bar u}
		}
	\end{equation*}
	for all $\dualelement \in C^1(\omega_1)\dualspace$.
	Further, the infimum is attained whenever the set on the right-hand side
	is not empty.
\end{lemma}
\begin{proof}
	Due to \Cref{thm:J*_is_distance},
	we have $\wdist = J\conjugate$.
	By utilizing
	that $J$ is twice directionally Hadamard differentiable at $\bar\varphi$, see \cref{thm:J''},
	we can argue as in \cite[Lemma~4.4]{WachsmuthWachsmuth2022}
	to obtain
	$\frac 1 2 \wdist''(\bar u,\bar\varphi;\cdot) \ge (\frac 1 2 J''(\bar\varphi; \cdot))\dualspace$.
	In order to evaluate this convex conjugate,
	we define the following maps
	\begin{align*}
		g \colon C(\omega_1; \R^d) \to \R
		,
		\qquad
		g(w)
		&:=
		\frac 1 2 \int_{\omega_1} w(\eta)^\top H(\eta)^{-1} w(\eta) \d \bar u(\eta)
		,
		\\
		L \colon C^1(\omega_1) \to C(\omega_1; \R^d)
		,
		\qquad
		L\phi
		&:=
		\nabla \phi
		.
	\end{align*}
	By \Cref{thm:J''}, $\frac12 J''(\bar\varphi; \cdot) = g \circ L$ holds.
	Since $\dom g = C(\omega_1; \R^d)$, we can apply \cite[Theorem~3]{Rockafellar1971}.
	This yields
	\begin{equation*}
		(g \circ L)\dualspace(\dualelement)
		=
		\inf\set*{g\dualspace(\nu) \given \nu \in \MM(\omega_1;\R^d), L\dualspace \nu = \dualelement}
	\end{equation*}
	and the infimum is attained
	whenever $\dualelement$ belongs to the range of the adjoint $L\dualspace$.

	It remains to compute the adjoint $L\dualspace$ and the convex conjugate $g\dualspace$.
	The definition of the adjoint $L\dualspace \colon \MM(\omega_1; \R^d) \to C^1(\omega_1)\dualspace$
	gives
	\begin{align*}
		\dual{L\dualspace\nu}{\varphi}_{C^1(\omega_1)}
		=
		\dual{\nu}{L \varphi}_{C(\omega_1;\R^d)}
		=
		\dual{\nu}{\nabla \varphi}_{C(\omega_1; \R^d)}
		=
		\dual{\eqclass{0}{\nu}}{\varphi}_{C^1(\omega_1)}
	\end{align*}
	for all $\nu \in \MM(\omega_1; \R^d)$ and $\varphi \in C^1(\omega_1)$.
	For the convex conjugate of $g$,
	we use
	\cite[Corollary~4A]{Rockafellar1971}
	and get
	\begin{equation}
		\label{eq:g*}
		g\dualspace(\nu)
		=
		\begin{cases}
			\frac 1 2 \int_{\omega_1} \rho_{\nu}(\eta)^\top H(\eta) \rho_{\nu}(\eta) \d \bar u
			& \nu \ll \bar u, \frac{\d\nu}{\d\bar u} = \rho_{\nu},
			\\
			\infty
			& \text{otherwise}
		\end{cases}
	\end{equation}
	for all $\nu \in \MM(\omega_1; \R^d)$.
	Here, $\nu \ll \bar u$ means that $\nu$ is absolutely continuous w.r.t.\ $\bar u$
	and $\frac{\d\nu}{\d\bar u}$ denotes the Radon--Nikodým derivative.
	Putting everything together, we arrive at
	\begin{align*}
		\frac 1 2 \wdist''(\bar u,\bar\varphi;\dualelement)
		&\ge
		\parens*{ \frac 1 2 J''(\bar\varphi; \cdot) }\dualspace(\dualelement)
		=
		(g \circ L)\dualspace(\dualelement)
		\\&
		=
		\inf\set*{g\dualspace(\nu) \given \nu \in \MM(\omega_1;\R^d), \eqclass{0}{\nu} = \dualelement}
		\\&
		=
		\inf\set*{
			\frac 1 2 \int_{\omega_1} \rho^\top H \rho \d \bar u
			\given
			\rho \in L^1(\bar u; \R^d),
			\dualelement = \eqclass{0}{\rho \bar u}
		}
		.
	\end{align*}
	Finally, we utilize \cref{thm:J''} which implies that
	$H$ is uniformly positive definite
	$\bar u$-a.e.\ and this implies
	that we can use $\rho \in L^2(\bar u; \R^d)$
	inside the infimum.
\end{proof}

In the next lemma, we show that a difference quotient of subgradients
(recall \cref{lem:subdiff_J})
converges weak-$\star$
and this will be used to show the property similar to \itemref{lem:density_lemma:2}.
\begin{lemma}
	\label{lem:recovery}
	Suppose that \cref{assu:varphireg} holds.
	Let $\seq{t_k}$ with $t_k \searrow 0$ be arbitrary. Let $\psi, \psi_k \in C^1(\omega_1)$ be given with $\psi_k \to \psi$.
	We define the sequence
	$\seq{\psi_k\dualspace} \subset C^1(\omega_1)\dualspace$ via
	\begin{equation}
		\label{eq:psi_k^*}
		\psi_k\dualspace
		:=
		\eqclass*{\frac{T(\bar\varphi+t_k \psi_k) \# u_0 - T(\bar\varphi) \# u_0}{t_k}}{0}
	\end{equation}
	Then,
	$\psi_k\dualspace \weaklystar \psi\dualspace := \eqclass{0}{(H^{-1} \nabla \psi) \bar u}$,
	cf. \eqref{eq:Hinvgradpsidu}.
	Further,
	\begin{align}
		\label{eq:replacement_for_norm_convergence}
		\limsup_{k\to\infty} \norm{\psi_k\dualspace}_{C^1(\omega_1)\dualspace}
		&\le
		C_L
		\sqrt{\frac{u_0(\omega_0)}{\gamma}}
		\sqrt{
			\int_{\omega_1} \nabla \psi^\top H^{-1} \nabla \psi \d \bar u(\eta)
		}
		.
	\end{align}
\end{lemma}
\begin{proof}
	Let $\phi \in C^1(\omega_1)$ be arbitrary.
	Using $\bar T = T(\bar\varphi)$
	and using the substitution rule for pushforward measures, we find
	\begin{equation*}
		\dual{\psi_k\dualspace}{\phi}
		=
		\int_{\omega_0}
		\frac{\phi(T(\bar\varphi + t_k \psi_k)(x)) - \phi(\bar T(x))}{t_k}
		\d u_0(x)
		.
	\end{equation*}
	Again, we use the dominated convergence theorem to pass to the limit.
	For each $x \in \omega_0$,
	we apply \cref{lem:sensitivity}
	to $f = c(x, \cdot) - \bar\varphi$ and $h_k = -\psi_k$.
	Then, \eqref{eq:taylorsensitivity}
	together with the differentiability of $\phi$ shows that
	the pointwise limit of the integrand is
	\begin{equation*}
		\nabla \phi(\bar T(x))^\top
		(\nabla^2_{\eta\eta} c(x, \bar T(x)) - \nabla^2 \bar\varphi(\bar T(x)))^{-1}
		\nabla \psi(\bar T(x))
	\end{equation*}
	and \eqref{eq:calm_solution_operator}
	yields the uniform bound
	\begin{equation*}
		\abs*{
			\frac{\phi(T(\bar\varphi + t_k \psi_k)(x)) - \phi(\bar T(x))}{t_k}
		}
		\le
		\frac{
			\lip(\phi) \norm{\psi_k}_{C^1(\omega_1)}
		}{\hat\gamma}
	\end{equation*}
	for all $k$ large enough.
	Thus, we can pass to the limit and obtain
	\begin{align*}
		\dual{\psi_k\dualspace}{\phi}
		&\to
		\int_{\omega_0}
		\nabla \phi(\bar T(x))^\top
		(\nabla^2_{\eta\eta} c(x, \bar T(x)) - \nabla^2 \bar\varphi(\bar T(x)))^{-1}
		\nabla \psi(\bar T(x))
		\d u_0(x)
		\\
		&=
		\int_{\omega_1}
			\int_{\bar T^{-1}(\set{\eta})}
			\nabla \phi(\eta)^\top
			(\nabla_{\eta\eta} c(x, \eta) - \nabla^2 \bar\varphi(\eta))^{-1}
			\nabla \psi(\eta)
			\d \hat u_\eta(x)
		\d\bar u(\eta)
		\\
		&=
		\int_{\omega_1}
		\nabla \phi(\eta)^\top
		H(\eta)^{-1}
		\nabla \psi(\eta)
		\d\bar u(\eta)
		=
		\dual{\psi\dualspace}{\phi}
		,
	\end{align*}
	where we used \eqref{eq:disintegration} and the definition of $H$, see \eqref{eq:H}.
	This proves $\psi_k\dualspace \weaklystar \psi\dualspace$.

	To verify \eqref{eq:replacement_for_norm_convergence},
	we start with
	\begin{align*}
		\norm{\psi_k\dualspace}_{C^1(\omega_1)\dualspace}
		&=
		\sup_{\norm{\phi}_{C^1(\omega_1)} \le 1}
		\dual{\psi_k\dualspace}{\phi}
		\\&
		=
		\sup_{\norm{\phi}_{C^1(\omega_1)} \le 1}
		\int_{\omega_0}  \frac{\phi(T(\bar\varphi + t_k \psi_k)(x)) - \phi(\bar T(x))}{t_k}
		\d u_0(x)
		\\
		&\le
		C_L
		\int_{\omega_0}  \frac{\abs{T(\bar\varphi + t_k \psi_k)(x) - \bar T(x)}}{t_k}
		\d u_0(x)
		,
	\end{align*}
	where we used $\lip(\phi) \le C_L \norm{\phi}_{C^1(\omega_1)}$,
	see \cref{lem:Lipschitz_constant_bounded_by_C1_norm}.
	By arguing as above, we can use the dominated convergence theorem
	to pass to the limit $k \to \infty$.
	This implies
	\begin{equation*}
		\limsup_{k \to \infty}
		\norm{\psi_k\dualspace}_{C^1(\omega_1)\dualspace}
		\le
		C_L
		\int_{\omega_0}
		\abs*{
			(\nabla^2_{\eta\eta} c(x, \bar T(x)) - \nabla^2 \bar\varphi(\bar T(x)))^{-1}
			\nabla \psi(\bar T(x))
		}
		\d u_0(x)
		.
	\end{equation*}
	We abbreviate $h(x) := \nabla^2_{\eta\eta} c(x, \bar T(x)) - \nabla^2 \bar\varphi(\bar T(x))$.
	With Hölder's inequality, we obtain
	\begin{equation*}
		\limsup_{k \to \infty}
		\norm{\psi_k\dualspace}_{C^1(\omega_1)\dualspace}
		\le
		C_L
		\sqrt{
			u_0(\omega_0)
		}
		\sqrt{
			\int_{\omega_0}
			\abs*{
				h(x)^{-1}
				\nabla \psi(\bar T(x))
			}^2
			\d u_0(x)
		}
		.
	\end{equation*}
	Next, we recall the inequality
	\begin{equation*}
		\abs{ A v }^2
		\le
		\norm{A} ( v^\top A v )
	\end{equation*}
	which holds for all symmetric, positive definite matrices $A \in \R^{d \times d}$
	and all vectors $v \in \R^d$.
	In our case, the norm of the matrix $h(x)^{-1}$
	is bounded by $1/\gamma$ due to \cref{assu:regularity_assumptions_on_T}.
	Thus, we find
	\begin{equation*}
		\limsup_{k \to \infty}
		\norm{\psi_k\dualspace}_{C^1(\omega_1)\dualspace}
		\le
		C_L
		\sqrt{
			\frac{u_0(\omega_0)}{\gamma}
		}
		\sqrt{
			\int_{\omega_0}
			\nabla \psi(\bar T(x))^\top
			h(x)^{-1}
			\nabla \psi(\bar T(x))
			\d u_0(x)
		}
		.
	\end{equation*}
	Using the disintegration formula \eqref{eq:disintegration}
	and the definition of $H$
	implies \eqref{eq:replacement_for_norm_convergence}.
\end{proof}

We will show now that
equality holds
in \Cref{lem:lower_bound_subderivative}
for $\psi\dualspace$ as in \eqref{eq:Hinvgradpsidu}.
This corresponds to \itemref{lem:density_lemma:2}.
\begin{theorem}
	\label{thm:Wpp_on_subspace}
	Let \Cref{assu:varphireg} be true.
	Let $\psi \in C^1(\omega_1)$ be given and define $\psi\dualspace := \eqclass{0}{(H^{-1} \nabla \psi) \bar u}$ as in \eqref{eq:Hinvgradpsidu}.
	Then,
	\begin{equation}
		\label{eq:Formel_4_20}
		\begin{aligned}
			\frac 1 2 \wdist''(\bar u, \bar\varphi; \psi\dualspace)
			&=
			\frac 12 \int_{\omega_1} \nabla \psi^\top H^{-1} \nabla \psi \d\bar u
			\\
			&=
			\inf\set*{
				\frac 1 2 \int_{\omega_1} \rho^\top H \rho \d \bar u
				\given
				\rho \in L^2(\bar u; \R^d),
				\psi\dualspace = \eqclass{0}{\rho \bar u}
			}
			.
		\end{aligned}
	\end{equation}
\end{theorem}
\begin{proof}
	We use the idea of \cite[Lemma~4.5]{WachsmuthWachsmuth2022}.
	The application of \Cref{lem:subdiff_J} on $\bar \varphi + t_k \psi$ and $\bar\varphi$
	yields
	$\bar u \in \partial J(\bar\varphi)$
	and
	\begin{equation*}
		\eqclass{T(\bar\varphi + t_k \psi) \# u_0}{0}
		=
		\bar u + t_k \psi_k\dualspace
		\in
		\partial J(\bar\varphi + t_k \psi)
		,
	\end{equation*}
	where we used $\psi_k\dualspace$ from \eqref{eq:psi_k^*} with $\psi_k \equiv \psi$.
	Thus, the Fenchel--Young identity
	implies
	\begin{align*}
		J(\bar\varphi + t_k \psi)
		+
		\wdist(\bar u + t_k \psi_k\dualspace)
		&=
		\dual{\bar u + t_k \psi_k\dualspace}{\bar\varphi + t_k \psi}
		,
		\\
		\mrep{
			J(\bar\varphi)
		}{
			J(\bar\varphi + t_k \psi)
		}
		+
		\mrep{
			\wdist(\bar u)
		}{
			\wdist(\bar u + t_k \psi_k\dualspace)
		}
		&=
		\dual{
			\mrep{
				\bar u
			}{
				\bar u + t_k \psi_k\dualspace
			}
		}{
			\mrep{
				\bar\varphi
			}{
				\bar\varphi + t_k \psi
			}
		}
		.
	\end{align*}
	Taking the difference, rearranging and dividing by $t_k^2$
	yield
	\begin{equation*}
		\frac 1 2
		\frac{J(\bar\varphi+t_k \psi) - J(\bar\varphi) - t_k \dual{\bar u}{\psi}}{t_k^2 / 2}
		+
		\frac 1 2 \frac{\wdist(\bar u + t_k \psi_k\dualspace) - \wdist(\bar u) - t_k \dual{\psi_k\dualspace}{\bar\varphi}}{t_k^2 / 2}
		=
		\dual{\psi_k\dualspace}{\psi}
		.
	\end{equation*}
	The first addend converges towards $\frac12 J''(\bar\varphi; \psi)$,
	cf.\ \Cref{thm:J''}.
	The right-hand side converges towards $\dual{\psi\dualspace}{\psi}$ by \cref{lem:recovery}.
	Together with $\psi_k\dualspace \weaklystar \psi\dualspace$ and
	the definition of the weak-$\star$ second subderivative
	we get
	\begin{align*}
		\frac 1 2 \wdist''(\bar u, \bar\varphi; \psi\dualspace)
		&\le
		\frac 1 2 \liminf_{k\to\infty}
		\frac{\wdist(\bar u + t_k \psi_k\dualspace) - \wdist(\bar u) - t_k \dual{\psi_k\dualspace}{\bar\varphi}}{t_k^2 / 2}
		\\&
		=
		\dual{\psi\dualspace}{\psi}
		-
		\frac 1 2 J''(\bar\varphi; \psi)
		\le
		\parens*{ \frac 1 2 J''(\bar\varphi; \cdot) }\dualspace(\psi\dualspace)
		.
	\end{align*}
	The other inequality was already shown in \Cref{lem:lower_bound_subderivative}.
	Thus, all inequalities in the previous formula are equalities.
	With \cref{thm:J''} and \cref{lem:lower_bound_subderivative},
	we obtain \eqref{eq:Formel_4_20}.
\end{proof}
Formula \eqref{eq:Formel_4_20} is only valid for special elements of $C^1(\omega_1)\dualspace$
as given in \eqref{eq:Hinvgradpsidu}.
Interestingly, it also shows that the infimum is attained by the density
$\rho = H^{-1} \nabla\psi$.
We will use the set of these functionals similar to $V$ in \Cref{lem:density_lemma}.
For dealing with an arbitrary direction,
we need some density argument,
cf. \itemref{lem:density_lemma:3}.

\begin{lemma}
	\label{lem:density_for_second_subderivative}
	Under \cref{assu:varphireg},
	we consider the following subsets of $C^1(\omega_1)\dualspace$,
	\begin{align*}
		M &:=
		\set*{
			\eqclass{0}{\rho \bar u}
			\given
			\rho \in L^2(\bar u; \R^d),
			\int_{\omega_1} \rho^\top H \rho \d\bar u \le 1
		}
		,
		\\
		N &:=
		\set*{
			\eqclass{0}{(H^{-1} \nabla \psi ) \bar u}
			\given
			\psi \in C^1(\omega_1; \R^d),
			\int_{\omega_1} \nabla\psi^\top H^{-1} \nabla\psi \d\bar u \le 1
		}
		.
	\end{align*}
	Then, every element of $M$
	is the weak-$\star$ limit of a sequence from $N$.
\end{lemma}
\begin{proof}
	We consider the polar sets, i.e.,
	\begin{align*}
		M\polar
		&:=
		\set{ \phi \in C^1(\omega_1) \given \dual{\dualelement}{\phi} \le 1 \; \forall \dualelement \in M },
		\\
		M\bipolar
		&:=
		\set{ \dualelement \in C^1(\omega_1)\dualspace \given \dual{\dualelement}{\phi} \le 1 \; \forall \phi \in M\polar }
	\end{align*}
	and $N\polar$, $N\bipolar$ are defined analogously.
	First, we show $N\polar \subset M\polar$.
	Let $g \in N\polar$ be given.
	The function $\psi := g / \sqrt{\int_{\omega_1} \nabla g^\top H^{-1} \nabla g \d\bar u} \in C^1(\omega_1; \R^d)$
	satisfies
	\begin{equation*}
		\int_{\omega_1} \nabla \psi^\top H^{-1} \nabla \psi \d\bar u
		=
		\frac{
			\int_{\omega_1} \nabla g^\top H^{-1} \nabla g \d\bar u
		}{
			\int_{\omega_1} \nabla g^\top H^{-1} \nabla g \d\bar u
		}
		=
		1
		.
	\end{equation*}
	Hence, the corresponding functional
	$\psi\dualspace := \eqclass{0}{(H^{-1}\nabla\psi) \bar u}$
	belongs to $N$.
	From $g \in N\polar$, we thus get
	\begin{equation*}
		1
		\ge
		\dual{\psi\dualspace}{g}
		=
		\int_{\omega_1} \nabla \psi^\top H^{-1} \nabla g \d\bar u
		=
		\sqrt{\int_{\omega_1} \nabla g^\top H^{-1} \nabla g \d\bar u}
		.
	\end{equation*}
	Now, let $\rho \in L^2(\bar u; \R^d)$ with
	$\int_{\omega_1} \rho^\top H \rho \d\bar u \le 1$ be arbitrary.
	Then, the Cauchy--Schwarz inequality
	on $\R^d$ equipped with the inner product induced by the matrix $H(\eta)$
	yields
	\begin{equation*}
		\abs{
			\rho^\top \nabla g
		}
		=
		\abs{
			\rho^\top H (H^{-1} \nabla g)
		}
		\le
		\sqrt{
			\rho^\top H \rho
		}
		\sqrt{
			\nabla g^\top H^{-1} \nabla g
		}
	\end{equation*}
	on $\omega_1$.
	Integration and using Cauchy--Schwarz in $L^2(\bar u)$
	give
	\begin{equation*}
		\dual{\eqclass{0}{\rho \bar u}}{g}
		=
		\int_{\omega_1} \rho^\top \nabla g \d\bar u
		\le
		\sqrt{\int_{\omega_1} \rho^\top H \rho \d\bar u}
		\sqrt{\int_{\omega_1} \nabla g^\top H^{-1} \nabla g \d\bar u}
		\le
		1.
	\end{equation*}
	Since $\rho$ was arbitrary,
	this shows $g \in M\polar$.
	Since $g$ was arbitrary, we have $N\polar \subset M\polar$.

	This implies $M \subset M\bipolar \subset N\bipolar$,
	while the bipolar theorem yields that $N\bipolar$ coincides
	with the closure of $N$ in the weak-$\star$ topology of $C^1(\omega_1)\dualspace$,
	see \cite[Theorem~IV.1.5]{SchaeferWolff1999}.
	Thus, $M \subset \cl_\star N$.
	Note that this does not immediately prove the existence of the desired sequence,
	since the weak-$\star$ topology is not sequential.
	However, the set $N$ is also bounded in $C^1(\omega_1)\dualspace$,
	since, for every $\psi\dualspace \in N$ and $\phi \in C^1(\omega_1)$,
	the lower bound on the eigenvalues of $H$ from \cref{thm:J''} yields that
	\begin{align*}
		\dual{\psi\dualspace}{\phi}
		&=
		\int_{\omega_1} \nabla\psi^\top H^{-1} \nabla \phi \d\bar u
		\le
		\sqrt{
		\int_{\omega_1} \nabla\psi^\top H^{-1} \nabla \psi \d\bar u
		}
		\sqrt{
		\int_{\omega_1} \nabla \phi^\top H^{-1} \nabla \phi \d\bar u
		}
		\\
		&\le
		\sqrt{
			\int_{\omega_1} \nabla \phi^\top H^{-1} \nabla \phi \d\bar u
		}
		\le
		C \norm{\phi}_{C^1(\omega_1)},
	\end{align*}
	where $\psi$ is the function used in the definition of $\psi\dualspace \in N$.
	Together with the separability of $C^1(\omega_1)$,
	we get that the weak-$\star$ topology is metrizable on the bounded set $\cl_\star N$,
	see \cite[Theorem~3.28]{Brezis2011}.
	The existence of the desired sequence follows.
\end{proof}

As a final step, we aim to prove epi-differentiability
which we will do similar to the proof of \Cref{lem:density_lemma}.
Note that we do not have the convergence of the norms as demanded in
\itemref{lem:density_lemma:2} and \ref{lem:density_lemma:3},
cf.\ \cref{lem:recovery,lem:density_for_second_subderivative}.
However, the proof of \cref{lem:density_lemma}
carries over and we finally obtain the weak-$\star$ epi-differentiability.

\begin{theorem}
	\label{thm:second_subdiff_W22_epidiff}
	Let \Cref{assu:varphireg} be true.
	Then,
	\begin{equation}\label{eq:weakstarsecondsubderivD}
		\begin{aligned}
			\frac 1 2 \wdist''(\bar u, \bar\varphi; \dualelement)
			&=
			\inf\set*{
				\frac 1 2 \int_{\omega_1} \rho^\top H \rho \d \bar u
				\given
				\rho \in L^2(\bar u; \R^d),
				\dualelement = \eqclass{0}{\rho \bar u}
			}
		\end{aligned}
	\end{equation}
	holds for all $\dualelement \in C^1(\omega_1)\dualspace$.
	The infimum is attained whenever the set on the right-hand side is not empty.
	Further, $\wdist$ is weak-$\star$ twice epi-differentiable at $\bar u$ for $\bar\varphi$.
\end{theorem}
\begin{proof}
	To abbreviate the right-hand side of the claimed identity, we define
	\begin{equation*}
		Q(\dualelement)
		:=
		\parens*{ \frac12 J''(\bar\varphi;\cdot)}\dualspace(\dualelement)
		=
		\inf\set*{
			\int_{\omega_1} \rho^\top H \rho \d \bar u
			\given
			\rho \in L^2(\bar u; \R^d),
			\dualelement = \eqclass{0}{\rho \bar u}
		}
		.
	\end{equation*}
	Let $\dualelement \in C^1(\omega_1)\dualspace$ be arbitrary.
	From \cref{lem:lower_bound_subderivative}
	we get $\wdist''(\bar u, \bar\varphi; \dualelement) \ge Q(\dualelement)$.
	It remains to show the other inequality.
	Thus, it is sufficient to consider the case that
	$Q(\dualelement)$ is finite.
	Consequently,
	\cref{lem:lower_bound_subderivative}
	yields
	$\rho \in L^2(\bar u; \R^d)$
	such that
	$\dualelement = \eqclass{0}{\rho \bar u}$
	and
	$Q(\dualelement) = \int_{\omega_1} \rho^\top H \rho \d\bar u$.

	Next, we approximate $\eqclass{0}{\rho \bar u}$
	by elements for which \cref{thm:Wpp_on_subspace} is applicable.
	To this end, define $\tilde\rho := \rho / \sqrt{\int_{\omega_1} \rho^\top H \rho \d\bar u} \in M$.
	By \cref{lem:density_for_second_subderivative}, there exists a sequence $\seq{\tilde\psi_n} \subset C^1(\omega_1)$
	such that $\eqclass{0}{(H^{-1} \nabla \tilde\psi_n ) \bar u} \weaklystar \eqclass{0}{\tilde\rho \bar u}$ and 
	$ \int_{\omega_1} \nabla \tilde\psi_n^\top H^{-1} \nabla \tilde\psi_n \d\bar u \leq 1$.
	Therefore, if we set $\psi_n := \tilde\psi_n \,\sqrt{\int_{\omega_1} \rho^\top H \rho \d\bar u} $, we obtain
	$\psi_n\dualspace := \eqclass{0}{(H^{-1} \nabla \psi_n ) \bar u} \weaklystar \eqclass{0}{\rho \bar u} = \dualelement$
	and
	\begin{equation*}
		Q(\dualelement)
		=
		\int_{\omega_1} \rho^\top H \rho \d\bar u
		\ge
		\int_{\omega_1} \nabla \psi_n^\top H^{-1} \nabla \psi_n \d\bar u
		\quad \forall\, n\in \N
		.
	\end{equation*}
	From \eqref{eq:Formel_4_20} and the definition of $Q$, we obtain that the right-hand side equals
	$Q(\psi_n\dualspace)$.
	Taking the limit inferior on both sides
	gives
	\begin{equation}
		\label{eq:Q_inequality}
		\wdist''(\bar u, \bar\varphi; \dualelement)
		\ge
		Q(\dualelement)
		\ge
		\liminf_{n \to \infty} Q(\psi_n\dualspace)
		.
	\end{equation}
	We choose a subsequence without renaming
	that realizes the limit inferior, i.e.,
	$\liminf_{n \to \infty} Q(\psi_n\dualspace) = \lim_{n \to \infty} Q(\psi_n\dualspace)$.
	As in the proof of \Cref{lem:density_lemma},
	we choose a recovery sequence for each of the $\psi_n\dualspace$
	(cf.\ \itemref{lem:density_lemma:2})
	and construct a diagonal sequence.
	As a first step, we note that
	\Cref{lem:recovery}
	implies
	\[
		\psi_{n,k}\dualspace
		:=
		\eqclass*{\frac{T(\bar\varphi + t_k \psi_n) \# u_0 - T(\bar\varphi) \# u_0 }{t_k}}{0}
		\weaklystar
		\eqclass{0}{(H^{-1} \nabla \psi_n) \bar u} = \psi_n\dualspace
	\]
	for $k \to \infty$ for all $n \in \N$.
	By \eqref{eq:replacement_for_norm_convergence} we obtain
	\begin{equation}
		\label{eq:sigmank*_bounded_in_k}
		\limsup_{k\to\infty} \norm{\psi_{n,k}\dualspace}_{C^1(\omega_1)\dualspace}
		\le
		C_L
		\sqrt{\frac{u_0(\omega_0)}{\gamma}}
		\sqrt{
			\int_{\omega_1} \nabla \psi_n^\top H^{-1} \nabla \psi_n \d \bar u
		}
		\le
		C
	\end{equation}
	with $C = C_L \sqrt{u_0(\omega_0) Q(\dualelement) /\gamma}$.
	In the last step of the proof, we construct a diagonal sequence
	which converges weak-$\star$ to $\dualelement$.
	As $C^1(\omega_1)$ is separable, we choose a dense subset $\seq{f_i} \subset C^1(\omega_1)$.
	The weak-$\star$ convergence $\psi_{n,k}\dualspace \weaklystar \psi_n\dualspace$ for $k \to \infty$
	as well as
	the proof of \cref{thm:Wpp_on_subspace}
	allow us to find a (strictly) increasing sequence $\seq{K_n} \subset \N$ with $K_1 := 1$
	such that for all $n \ge 2$
	\begin{equation}
		\label{eq:diag_sequence_boundedness_condition}
		\sum_{i=1}^n \abs{\dual{\psi_{n,k}\dualspace - \psi_n\dualspace}{f_i}}
		+
		\abs*{
			Q(\psi_n\dualspace)
			-
			\frac{\wdist(\bar u + t_k \psi_{n,k}\dualspace) - \wdist(\bar u) - t_k \dual{\psi_{n,k}\dualspace}{\bar\varphi}}{t_k^2 / 2}
		}
		\le
		\frac 1 n
	\end{equation}
	for all $k \ge K_n$
	and
	$\norm{\psi_{n,k}\dualspace}_{C^1(\omega_1)\dualspace} \le C + 1$
	for all $k \ge K_n$ (which is possible due to \eqref{eq:sigmank*_bounded_in_k}).
	Next, we define $n_k := \max\set{n \in \N \given k \ge K_n}$ (this is well-defined by $K_1 = 1$).
	As $\seq{K_n}$ is increasing, $n_k \to \infty$ monotonously for $k \to \infty$.
	This,
	$k \ge K_{n_k}$ and \eqref{eq:diag_sequence_boundedness_condition} show that
	$\hat \psi_k\dualspace := \psi_{n_k, k}\dualspace$ satisfies
	$\dual{\hat \psi_k\dualspace - \psi_{n_k}\dualspace}{f_i} \to 0$ for all $i \in \N$ and
	\begin{equation}
		\label{eq:diag_sequence_Wpp_converges}
		\abs*{
			Q(\psi_{n_k}\dualspace)
			-
			\frac{\wdist(\bar u + t_k \hat\psi_k\dualspace) - \wdist(\bar u) - t_k \dual{\hat \psi_k\dualspace}{\bar\varphi}}{t_k^2 / 2}
		}
		\to
		0
	\end{equation}
	for $k \to \infty$.
	Together with $\psi_{n_k}\dualspace \weaklystar \dualelement$
	we get
	$\dual{\hat \psi_k\dualspace - \dualelement}{f_i} \to 0$ for all $i \in \N$.
	Since we also have the boundedness
	$\norm{\hat\psi_k\dualspace}_{C^1(\omega_1)\dualspace} \le C + 1$,
	this implies
	$\hat\psi_k\dualspace \weaklystar \dualelement$.
	
	Finally, we check that $\seq{\hat \psi_k\dualspace}$ is a recovery sequence.
	As we chose $\psi_n\dualspace$ such that $Q(\psi_n\dualspace)$ converges,
	\eqref{eq:diag_sequence_Wpp_converges} yields
	\begin{equation*}
		\lim_{n \to \infty} Q(\psi_n\dualspace)
		=
		\lim_{k \to \infty} Q(\psi_{n_k}\dualspace)
		=
		\lim_{k\to\infty}
		\frac{\wdist(\bar u + t_k \hat\psi_k\dualspace) - \wdist(\bar u) - t_k \dual{\hat \psi_k\dualspace}{\bar\varphi}}{t_k^2 / 2}
		.
	\end{equation*}
	From \eqref{eq:Q_inequality} and the definition of the weak-$\star$ second subderivative,
	we get
	\begin{equation*}
		\wdist''(\bar u, \bar\varphi, \dualelement)
		\ge
		\lim_{k\to\infty}
		\frac{\wdist(\bar u + t_k \hat\psi_k\dualspace) - \wdist(\bar u) - t_k \dual{\hat \psi_k\dualspace}{\bar\varphi}}{t_k^2 / 2}
		\ge
		\wdist''(\bar u, \bar\varphi, \dualelement)
		.
	\end{equation*}
	This shows that $\seq{\hat\psi_k\dualspace}$ is indeed a recovery sequence.
\end{proof}
Note that the original blueprint from \Cref{lem:density_lemma}
even yields strict twice epi-differentiability.
In our case,
this does not seem to be possible,
since the approximation results
\cref{lem:recovery,lem:density_for_second_subderivative}
do not provide convergence of the norms.
However,
for the application of the second-order theory
strict twice epi-differentiability is not needed,
see \cref{subsec:abstract_second_order}.
We mention that only the strong twice epi-differentiability
allows to replace the
sequential weak-$\star$ continuity of $h \mapsto F''(\bar x) h^2$
in \cref{thm:abstract_second_order_result}
by
sequential weak-$\star$ lower semicontinuity,
see \cite[Theorem~2.20]{BorchardWachsmuth2024}.
However,
it is not expected that our functional $\wdist$
is strongly twice epi-differentiable.

\begin{remark}
	\label{rem:comparison}
	As mentioned in the introduction,
	\cite[Theorem~1.4]{Caja-LopezDelgadinoKitagawa2026}
	provides a second-order differentiability result for the squared $2$-Wasserstein distance,
	i.e., $c(x, \eta) = \abs{x - \eta}^2$,
	under rather restrictive regularity assumptions.
	Translated to our language,
	they assume that the sets
	$\omega_0 = \supp(u_0), \bar\omega := \supp(\bar u)$ are bounded, uniformly convex with $C^{3,\alpha}$ boundary, $\alpha > 0$,
	and
	that the measures $u_0, \bar u$ have a uniformly positive and Hölder regular
	density w.r.t.\ the Lebesgue measure, which we denote by 
	$U_0 := \frac{\d u_0}{\d \lambda^d} \in C^{0,\alpha}(\omega_0)$
	and $\bar U := \frac{\d \bar u}{\d \lambda^d} \in C^{0,\alpha}(\bar\omega)$.
	Given perturbations $h \in C^{0,\alpha}(\bar\omega)$ and  $k \in C^{0,\alpha}(\omega_0)$
	with
		\begin{equation*}
			\int_{\omega_0} U_0 k \d \lambda^d
			=
			\int_{\bar\omega} \bar U h \d \lambda^d
			=
			0
			,
		\end{equation*}
	they prove that the map
	\begin{equation*}
		\KK \colon \R \ni t
		\mapsto
		\inf\set*{
			\int_{\omega_0 \times \bar\omega} \abs{x - \eta}^2 \d\pi(x,\eta)
			\given
			\pi \in
			\Gamma\parens*{
				U_0 (1 + t k) \lambda^d
				,
				\bar U (1 + t h) \lambda^d
			}
		}
	\end{equation*}
	is twice differentiable at $t = 0$
	and the corresponding second-order derivative is computed.
	Note that such a differentiability result is not suited for the derivation
	of second-order sufficient optimality conditions.

	In what follows, we argue that, 
	in the setting of  \cite[Theorem~1.4]{Caja-LopezDelgadinoKitagawa2026},
	the second derivative of $\KK$ at $0$
	coincides with the right-hand side of
	our formula \eqref{eq:weakstarsecondsubderivD}
	for the weak-$\star$ second subderivative, 
	provided that $k = 0$, since we do not consider perturbations in $u_0$.
	From \cite[Lemma~1.7]{Caja-LopezDelgadinoKitagawa2026} we find that
	the Brenier potential $\phi_0$ (which is unique up to constants) for the transport from $\bar u$ to $u_0$
	satisfies $\phi_0 \in C^{2}(\bar\omega)$ and $\nabla^2 \phi_0$ is uniformly positive definite on $\bar\omega$.
	Moreover, the transport map $\bar T \colon \omega_0 \to \bar\omega$ from $u_0$ to $\bar u$
	satisfies $\bar T = (\nabla\phi_0)^{-1}$.

	In order to comply with \cref{assu:varphireg},
	we use Whitney's extension theorem to extend $\phi_0$ to $\phi_0 \in C^2(\R^d)$,
	where the necessary remainder term estimates can be shown using the convexity of $\bar\omega$.
	Next, we choose a convex and compact set $\omega_1 \subset \R^d$ such that $\bar\omega \subset \interior(\omega_1)$
	and such that $\nabla^2\phi_0$ is still uniformly positive definite on $\omega_1$.
	We further define the function $\bar\varphi \colon \omega_1 \to \R$
	via
	$\bar\varphi(\eta) := \abs{\eta}^2 - 2\phi_0(\eta)$.
	Let us check that \eqref{eq:regularity} holds.
	For a fixed $x \in \omega_0$, the derivative of the map
	$\eta \mapsto c(x, \eta) - \bar\varphi(\eta) = \abs{x}^2 - 2\eta^\top x + 2 \phi_0(\eta)$ at $\eta \in \omega_1$
	is
	\(
		2( \nabla\phi_0(\eta) - x)
	\).
	Note that this vanishes for $\eta = \bar T(x) \in \bar\omega$.
	Together with the strong convexity of $\phi_0$, this shows that $\eta = \bar T(x)$
	is the unique minimizer of $c(x, \cdot) - \bar\varphi$
	and that the quadratic growth condition \eqref{eq:regularity} holds.
	This also shows that the function $\bar\varphi$ is a Kantorovich potential (no matter how the 
	extension on $\omega_1\setminus \bar\omega$ precisely looks like), since \eqref{eq:regularity} implies
	$\bar\varphi^{\overline{c}}(x) = c(x, \bar T(x)) - \bar\varphi(\bar T(x))$ for all $x\in \omega_0$ and thus
	\begin{align*}
	    \int_{\omega_1} \bar\varphi(\eta) \d\bar u(\eta)
	    & = \int_{\bar\omega} \bar\varphi(\eta) \d (\bar T \# u_0)(\eta)
	      = \int_{\omega_0} \bar\varphi(\bar T(x)) \d u_0(x)\\
	    &=  \int_{\omega_0} c(x, \bar T(x)) \d u_0(x) - 
	    \int_{\omega_0} \bar\varphi^{\overline{c}}(x) \d u_0(x)
	    = \DD(\eqclass{\bar u}{0}) + J(\bar\varphi)
	\end{align*}
	so that, according to \eqref{eq:Kantorovich_potential}, $\bar\varphi$ is indeed a Kantorovich potential.
	Hence, \cref{assu:regularity_assumptions_on_T} and, consequently, \cref{assu:varphireg} are verified.

	For the upcoming formulas, we mention
	$\nabla^2\phi_0 = \frac12 H$, see \eqref{eq:H_simple}.
	Now, we are in position to evaluate the formula from
	\cite[Theorem~1.4]{Caja-LopezDelgadinoKitagawa2026}.
	This shows that the second derivative of $\KK$ is given by
	\begin{equation}\label{eq:KKprimeprime}
	\begin{aligned}
	    \frac12 \KK''(0)
	    =
		\int_{\bar\omega} \nabla\xi^\top (\nabla^2\phi_0)^{-1} \nabla\xi \,\bar U \d\lambda^d 
		,
	\end{aligned}
	\end{equation}
	where $\xi$ is the weak solution of the following Neumann problem
	\begin{equation*}
		-\div\bracks*{
			\bar U (\nabla^2 \phi_0)^{-1} \nabla \xi
		}
		=
		- h\, \bar U
		\quad\text{in } \bar\omega, 
		\qquad
		n^\top (\nabla^2 \phi_0)^{-1} \nabla \xi = 0
		\quad\text{on } \partial\bar\omega
	\end{equation*}
	with the unit outer normal $n$.
	The weak formulation of this PDE reads
	\begin{equation}\label{eq:neumann}
		\int_{\bar\omega} \nabla \varphi^\top (\nabla^2 \phi_0)^{-1} \nabla\xi \,\bar U \d\lambda^d
		= - \int_{\bar\omega} \varphi h\, \bar U \d\lambda^d
		\quad \forall\, \varphi \in H^1(\interior(\bar\omega)).
	\end{equation}
    Note that the uniform positive definiteness of $\nabla^2 \phi_0$ 
    and the uniform positivity of $\bar U$ imply that 
		this PDE admits a solution $\xi \in H^1(\interior(\bar\omega))$, which is unique up to a constant.

    In the following, we show that the above expression for $\frac12 \KK''(0)$ 
    equals the infimum on the right hand side of \eqref{eq:weakstarsecondsubderivD}, 
    if we set $\Phi := \eqclass{h \bar U \lambda^d}{0}$. 
    On the space $L^2(\bar u; \R^d)$, we define the equivalent inner product
    \begin{equation*}
        (v, w)_H := \int_{\bar\omega} v^\top \nabla^2\phi_0 \, w \,\bar U \d\lambda^d 
        .
    \end{equation*}
    Moreover, we define the closed subspace $\UU \subset L^2(\bar u; \R^d)$ by 
    \begin{equation*}
        \UU = \set*{ (\nabla^2\phi_0)^{-1}\nabla \zeta \given \zeta \in H^1(\interior(\bar\omega))}
    \end{equation*}        
    and employ the decomposition $L^2(\bar u; \R^d) = \UU \oplus \UU^\perp$, where $\UU^\perp$ is the orthogonal 
    complement of $\UU$ w.r.t.\ the $H$-inner product.
    Next,
    we abbreviate the feasible set of the minimization problem in \eqref{eq:weakstarsecondsubderivD} by
    \begin{equation*}
        R
        :=
        \set{
            \rho \in L^2(\bar u; \R^d)
            \given
            \Phi = \eqclass{0}{\rho \bar u}
        }.
    \end{equation*}
    Together with $\Phi = \eqclass{h \bar U \lambda^d}{0}$,
    we find that $\rho \in R$ is equivalent to
    \begin{equation*}
        \int_{\omega_1} \varphi \, h\, \bar U \d \lambda^d 
        =
        \dual{\Phi}{\varphi}
        =
        \dual{\eqclass{0}{\rho\bar u}}{\varphi}
        =
        \int_{\omega_1} \rho^\top \nabla \varphi \, \bar U \d\lambda^d
        \qquad\forall \varphi \in C^1(\omega_1).
    \end{equation*}
    Here, all integrals can be changed to integrals over $\bar\omega$, since $\bar U = 0$ on $\omega_1 \setminus \bar\omega$.
    Now, we employ the above orthogonal decomposition, i.e.,
    every $\rho \in L^2(\bar u; \R^d)$ can be written as
    $\rho = -(\nabla^2\phi_0)^{-1} \nabla \zeta + w$
    with
    $\zeta \in H^1(\interior(\bar\omega))$ and $w \in \UU^\perp$.
    Consequently, $\rho \in R$ is equivalent to
    \begin{align*}
        \int_{\bar\omega} \varphi \, h\, \bar U \d \lambda^d 
        &=
        -\int_{\bar\omega} \nabla\zeta^\top (\nabla^2\phi_0)^{-1} \nabla \varphi \, \bar U \d\lambda^d
        +
        \int_{\bar\omega} w^\top \nabla \varphi \, \bar U \d\lambda^d
        \\&
        =
        -\int_{\bar\omega} \nabla\zeta^\top (\nabla^2\phi_0)^{-1} \nabla \varphi \, \bar U \d\lambda^d
        +
        (w, (\nabla^2 \phi_0)^{-1}\nabla \varphi)_H
    \end{align*}
    for all $\varphi \in C^1(\omega_1)$.
    Due to $w \in \UU^\perp$, the last inner product vanishes.
    By density, the last equality is also satisfied for all
    $\varphi \in H^1(\interior(\bar\omega))$
    and we arrive at the weak formulation \eqref{eq:neumann}.
    Since the latter has a unique solution (up to constants),
    this shows
    \begin{equation*}
        R
        =
        \set{
            -(\nabla^2 \phi_0)^{-1}\nabla\xi + w
            \given
            w \in \UU^\perp
        }
        =
        -(\nabla^2 \phi_0)^{-1}\nabla\xi
        +
        \UU^\perp
        ,
    \end{equation*}
    where $\xi \in H^1(\interior(\bar\omega))$ is an arbitrary, fixed solution of \eqref{eq:neumann}.
    In particular, this shows that the set $R$ is not empty.
    Now,
    using the $H$-inner product
    and $\frac12 H = \nabla^2\phi_0$,
    the minimization problem in \eqref{eq:weakstarsecondsubderivD} becomes
    \begin{equation*}
        \inf \set*{
            (\rho, \rho)_H
            \given
            \rho \in R
        }
        .
    \end{equation*}
    Due to the above representation of $R$,
    which is orthogonal w.r.t.\ the $H$-inner product,
    it is immediate that the unique solution of this problem is given by
    $\rho = -(\nabla^2 \phi_0)^{-1}\nabla\xi$
    and the infimum is precisely
    \begin{equation*}
        \int_{\bar\omega} \nabla \xi^\top (\nabla^2 \phi_0)^{-1} \nabla \xi\, \bar U \d \lambda^d
        ,
    \end{equation*}
    which equals the value of $\frac12 \KK''(0)$ from \eqref{eq:KKprimeprime}.
    
    To summarize, we have shown that
    \begin{equation*}
        \wdist''(\bar u, \bar\varphi; \eqclass{h \bar U \lambda^d}{0} )
        =
        \KK''(0)
        =
        \frac{\d^2}{\d t^2}
        D(\bar u + t h \bar U \lambda^d) |_{t = 0}
        .
    \end{equation*}
    We further observe that the results of \cite[Theorem~1.4]{Caja-LopezDelgadinoKitagawa2026} 
    and \cref{thm:second_subdiff_W22_epidiff} complement each other.
    The data in \cite[Theorem~1.4]{Caja-LopezDelgadinoKitagawa2026} is supposed to be
    much smoother compared to our setting.
    This allows to prove that $u \mapsto D(u)$
    is twice directionally differentiable at $\bar u$ w.r.t.\ regular directions.
    On the contrary,
    we are able to characterize the weak-$\star$ second subderivative,
    which, ultimately, gives rise to no-gap second-order conditions.
\end{remark}

\section{Second-order conditions}
\label{sec:second_order_condition}
In this section, we want to apply the findings of the previous two section
to study our (lifted) optimization problem \eqref{P'}.
By adapting the abstract assumption from \cref{thm:abstract_second_order_result}
to our situation,
we state the following differentiability assumption on $\FF$.

\begin{assumption}
	\label{assu:F_second_order_Taylor}
	There exist $\FF'(\bar u) \in C^1(\omega_1)$ and
	a bounded bilinear form
	$\FF''(\bar u) \colon C^1(\omega_1)\dualspace \times C^1(\omega_1)\dualspace \to \R$
	with
	\begin{equation}\label{eq:taylor2}
		\lim_{k \to \infty}
		\frac
		{\FF(\bar u + t_k h_k) - \FF(\bar u) - t_k \FF'(\bar u) h_k - \frac 1 2 t_k^2 \FF''(\bar u) h_k^2}
		{t_k^2}
		=
		0
	\end{equation}
	for all sequences
	$\seq{t_k} \subset \R^+$, $\seq{h_k} \subset C^1(\omega_1)\dualspace$
	satisfying
	$t_k \to 0, h_k \weaklystar h \in C^1(\omega_1)\dualspace$ and
	$\bar u + t_k h_k \in \dom(\wdist)$.
	Additionally, the map
	$C^1(\omega_1)\dualspace \ni h \mapsto \FF''(\bar u) h^2$ is sequentially weak-$\star$ continuous.
\end{assumption}
Under \cref{assu:F_second_order_Taylor},
we obtain
\begin{equation*}
	-\FF'(\bar u) \in \frac\alpha2 \partial \wdist(\bar u)
\end{equation*}
as a first-order optimality condition for $\bar u$
by standard arguments.
In particular,
$-\FF'(\bar u) \in C^1(\omega_1)$
is a Kantorovich potential for $\bar u$.

Now we are in position to prove
our first second-order result.
\begin{theorem}
	\label{thm:SNC}
	We assume that \cref{assu:F_second_order_Taylor} holds
	and that \cref{assu:weaker_regularity_assumptions_on_T} is satisfied
	for $\bar\varphi = -\frac{2}{\alpha}\FF'(\bar u)$.
	Then,
	\begin{equation}
		\label{eq:Phi''_positive}
		\FF''(\bar u) h^2 + \frac \alpha 2 \wdist''(\bar u,\bar\varphi;h) > 0
		\quad
		\forall h \in C^1(\omega_1)\dualspace \setminus \set{0}
	\end{equation}
	is equivalent to the existence of $\kappa,\varepsilon > 0$ such that the quadratic growth condition
	\begin{equation}
		\label{eq:quadratic_growth}
		F(u) + \frac\alpha2 D(u)
		\ge
		F(\bar u) + \frac\alpha2 D(\bar u) + \frac \kappa 2 \norm{u - \bar u}_{\BL}^2
		\quad
		\forall u \in \MM(\omega_1),
		\norm{u - \bar u}_{\BL} \le \varepsilon
	\end{equation}
	holds.
\end{theorem}
\begin{proof}
	We are going to apply \cref{thm:abstract_second_order_result}.
	We note that \eqref{NDC} follows from \Cref{thm:NDC_D_inequality} and \Cref{lem:ndc_sufficient}.
	Consequently, \cref{thm:abstract_second_order_result}
	implies that \eqref{eq:Phi''_positive} is equivalent to
	the existence of $\kappa', \varepsilon' > 0$ with
	\begin{equation*}
		\FF(\dualelement) + \frac\alpha2 \wdist(\dualelement)
		\ge
		\FF(\bar u) + \frac\alpha2 \wdist(\bar u) + \frac {\kappa'} 2 \norm{\dualelement - \bar u}_{C^1(\omega_1)\dualspace}^2
	\end{equation*}
	for all
	$\dualelement \in C^1(\omega_1)\dualspace$
	with
	$\norm{\dualelement - \bar u}_{C^1(\omega_1)\dualspace} \le \varepsilon'$.
	By definition of the lifted functionals,
	$\DD(\dualelement) = \infty$ if $\dualelement$ does not belong to the range $\Pi$,
	i.e., it cannot be represented by a measure.
	Together with the equivalence of
	$\norm{\eqclass{\cdot}{0}}_{C^1(\omega_1)\dualspace}$
	and
	$\norm{\cdot}_{\BL}$
	on $\MM(\omega_1)$,
	see \cref{lem:BLnorm_C1*norm_equivalent},
	we see that the quadratic growth condition in $C^1(\omega_1)\dualspace$
	is equivalent to \eqref{eq:quadratic_growth}.
\end{proof}
At this point it is interesting to note
that \cref{thm:SNC} can be proved
without having an explicit formula for the weak-$\star$ second subderivative 
and, therefore, the weaker \cref{assu:weaker_regularity_assumptions_on_T} 
suffices at this point, cf.\ \cref{rem:strongerassu}.
The precise computation of the weak-$\star$ second subderivative 
requires stronger assumptions, see \cref{sec:second_subderivative}.
Under these stronger assumptions, we can reformulate condition \eqref{eq:Phi''_positive}.

\begin{lemma}
	\label{lem:equivalence_of_equivalence_classes_second_order}
	We assume that \cref{assu:F_second_order_Taylor} holds
	and that \cref{assu:varphireg} is satisfied for $\bar\varphi = -\FF'(\bar u)$.
	Then, the conditions
	\begin{subequations}
		\begin{align}
			\label{eq:Phi''_positive_again}
			&\FF''(\bar u) h^2 + \frac \alpha 2 \wdist''(\bar u,\bar\varphi;h) > 0
			&&
			\forall h \in C^1(\omega_1)\dualspace \setminus \set{0}
			\\
			\label{lem:equivalence_of_equivalence_classes_second_order:i}
			&\FF''(\bar u) \eqclass{0}{\nu}^2 + \frac \alpha 2 \wdist''(\bar u,\bar\varphi;\eqclass{0}{\nu})
			>
			0
			&&
			\forall
			\nu \in \MM(\omega_1; \R^d) \setminus \set{0}
			\\
			\label{lem:equivalence_of_equivalence_classes_second_order:ii}
			&\FF''(\bar u) \eqclass{0}{\tilde \nu}^2 + \alpha g\dualspace(\tilde \nu)
			>
			0
			&&
			\forall
			\tilde \nu \in \MM(\omega_1; \R^d) \setminus \set{0}
			\\
			\label{lem:equivalence_of_equivalence_classes_second_order:iii}
			&\FF''(\bar u) \eqclass{0}{\rho \bar u}^2
			+
			\frac \alpha 2 \int_{\omega_1} \rho^\top H \rho \d\bar u
			>
			0
			&&
			\forall
			\rho \in L^2(\bar u; \R^d) \setminus \set{0}
		\end{align}
	\end{subequations}
	are equivalent,
	where we use $g\dualspace$ from \eqref{eq:g*}.
\end{lemma}
\begin{proof}
	The equivalence of \eqref{eq:Phi''_positive_again}
	and \eqref{lem:equivalence_of_equivalence_classes_second_order:i}
	follows
	since $\wdist''(\bar u, \bar\varphi; h) = \infty$
	if $h$ cannot be represented as $\eqclass{0}{\nu}$, cf.\ \cref{thm:second_subdiff_W22_epidiff}.
	Similarly, \cref{thm:second_subdiff_W22_epidiff}
	together with \eqref{eq:g*} yields
	\begin{equation*}
		\frac12 \wdist''(\bar u, \bar\varphi; \eqclass{0}{\nu})
		=
		\inf\set{ g\dualspace(\tilde\nu) \given \eqclass{0}{\nu} = \eqclass{0}{\tilde\nu} }
		.
	\end{equation*}
	Thus,
	\eqref{lem:equivalence_of_equivalence_classes_second_order:i}
	implies
	\eqref{lem:equivalence_of_equivalence_classes_second_order:ii}.
	For the converse direction, we need that
	the infimum in the previous formula is attained (whenever it is not $\infty$).
	From this, we infer that
	\eqref{lem:equivalence_of_equivalence_classes_second_order:ii}
	implies
	\eqref{lem:equivalence_of_equivalence_classes_second_order:i}.
	Finally,
	\eqref{lem:equivalence_of_equivalence_classes_second_order:iii}
	is just a reformulation of
	\eqref{lem:equivalence_of_equivalence_classes_second_order:ii}
	using the densities.
\end{proof}

Finally,
our results allow to apply 
\cite[Theorem~6.11]{WachsmuthWalter2025:1}
to obtain a strong metric subregularity result.

\begin{theorem}
	\label{thm:SMS}
	We assume that the cost function is given by
	$c(x,\eta) := \abs{x - \eta}^p$ with $p \in [2,\infty)$.
	We assume that
	mappings $\FF' \colon C^1(\omega_1)\dualspace \to C^1(\omega_1)$
	and $\FF''(\bar u) \colon C^1(\omega_1)\dualspace \to C^1(\omega_1)$
	are given
	such that
	\[
		\frac{\FF'(\bar u + t_n h_n) - \FF'(\bar u)}{t_n}
		\to
		\FF''(\bar u) h
		\qquad
		\text{in } C^1(\omega_1)
	\]
	holds
	for all sequences $\seq{t_n} \subset \R^+$, $\seq{h_n} \subset C^1(\omega_1)\dualspace$ with
	$t_n \to 0$, $h_n \weaklystar h$ and $\bar u + t_n h_n \in \dom(\wdist)$.
	Finally,
	we assume that \cref{assu:varphireg}
	is satisfied for $\bar\varphi = -\FF'(\bar u)$.

	Then, the second-order condition
	\eqref{eq:Phi''_positive}
	with $\FF''(\bar u) h^2 = \dual{h}{\FF''(\bar u) h}$
	implies the existence of
	$\nu, R > 0$ such that
	\begin{equation*}
		\norm{u - \bar u}_{\BL}
		\le
		\frac 1 \nu
		\inf\set*{
			\norm*{\FF'(u) + \frac\alpha2 \varphi}_{C^1(\omega_1)}
			\given
			\varphi \in \partial D(u) \cap C^1(\omega_1)
		}
	\end{equation*}
	for all $u \in \MM(\omega_1)$ with $\norm{u - \bar u}_{\BL} \le R$.
\end{theorem}
\begin{proof}
	We verify that the four properties from \cite[Theorem~6.11]{WachsmuthWalter2025:1}
	hold.

	The differentiability property
	\cite[Theorem~6.11(i)]{WachsmuthWalter2025:1}
	holds by assumption.

	The second-order condition
	\cite[Theorem~6.11(iv)]{WachsmuthWalter2025:1}
	is precisely \eqref{eq:Phi''_positive}.
	
	For
	\cite[Theorem~6.11(iii)]{WachsmuthWalter2025:1}
	we have to show that for all
	$\seq{t_n} \subset \R^+$,
	$\seq{h_n} \subset C^1(\omega_1)\dualspace$,
	$\seq{\psi_n} \subset C^1(\omega_1)$
	with
	$t_n \to 0$,
	$h_n \weaklystar h$ in $C^1(\omega_1)\dualspace$,
	$\psi_n \to \psi$ in $C^1(\omega_1)$
	and
	$\bar\varphi + t_n \psi_n \in \partial\wdist(\bar u + t_n h_n)$
	we have
	\begin{equation*}
		\dual{h}{\psi}
		\ge
		\wdist''(\bar u, \bar\varphi; h)
		.
	\end{equation*}
	Due to the assumption on the cost functional,
	\cref{lem:subdiff_J} implies the uniqueness of the subdifferential
	and, thus,
	$\bar u + t_n h_n = T(\bar\varphi + t_n \psi_n) \# u_0$.
	Consequently,
	$h_n = \psi_n\dualspace$ from \cref{lem:recovery}
	and we get
	$h_n = \psi_n\dualspace \weaklystar \psi\dualspace = \eqclass{0}{(H^{-1} \nabla \psi)\bar u}$,
	in particular, $h = \psi\dualspace$.
	Consequently,
	\begin{equation*}
		\dual{h}{\psi}
		=
		\dual{\psi\dualspace}{\psi}
		=
		\int_{\omega_1} \nabla\psi^\top H^{-1} \nabla\psi \d\bar u
		=
		\wdist(\bar u, \bar\varphi; h)
		,
	\end{equation*}
	where the last equality follows from \eqref{eq:Formel_4_20}
	and $h = \psi\dualspace$.

	For
	\cite[Theorem~6.11(ii)]{WachsmuthWalter2025:1}
	we have to show that
	the existence of sequences
	$\seq{t_n} \subset \R^+$,
	$\seq{h_n} \subset C^1(\omega_1)\dualspace$,
	$\seq{\psi_n} \subset C^1(\omega_1)$
	with
	$t_n \to 0$,
	$h_n \weaklystar 0$ in $C^1(\omega_1)\dualspace$,
	$\norm{h_n}_{C^1(\omega_1)\dualspace} = 1$,
	$\psi_n \to 0$ in $C^1(\omega_1)$
	and
	$\bar\varphi + t_n \psi_n \in \partial\wdist(\bar u + t_n h_n)$
	is impossible.
	Suppose that such sequences exist.
	By arguing as above, we find from \cref{lem:recovery}
	that $h_n = \psi_n\dualspace$
	and \eqref{eq:replacement_for_norm_convergence}
	implies
	$\norm{\psi_n\dualspace}_{C^1(\omega_1)\dualspace} \to 0$
	since $\psi = 0$ in our current situation.
	This is a contradiction.

	Consequently,
	\cite[Theorem~6.11]{WachsmuthWalter2025:1}
	implies the existence of $\nu', R' > 0$
	such that
	\begin{equation*}
		\norm{u - \bar u}_{C^1(\omega_1)\dualspace}
		\le
		\frac 1 {\nu'}
		\inf\set*{
			\norm*{\FF'(u) + \frac\alpha2 \varphi}_{C^1(\omega_1)}
			\given
			\varphi \in \partial \wdist(u) \cap C^1(\omega_1)
		}
	\end{equation*}
	holds for all $u \in C^1(\omega_1)\dualspace$
	with $\norm{u - \bar u}_{C^1(\omega_1)\dualspace} \le R'$.
	We can argue as in the proof of \cref{thm:SNC}
	to reduce this condition to measures.
	Indeed, $\partial \wdist(u) \ne \emptyset$
	implies that $u$ is represented by a measure
	and an application of \cref{lem:BLnorm_C1*norm_equivalent}
	yields the claim.
\end{proof}

We mention that for nice functionals $\FF$,
the bilinear form $\FF''(\bar u)$ from \cref{assu:F_second_order_Taylor}
essentially coincides with the operator $\FF''(\bar u)$
from \cref{thm:SMS}.

First, we want to show that the above can be applied to
a nonlinear optimal control problem governed by a linear equation
if the observation domain has a positive distance to the control domain.

\begin{example}
	\label{ex:nonlin_ctrl}
	Let $\Omega \subset \R^d$, $d = 2,3$, be a bounded Lipschitz domain
	such that the control domain satisfies $\omega_1 \subset \Omega$.
	For some fixed $q \in (1, d/(d-1))$,
	we want to consider the optimal control problem
	\begin{equation}
		\label{eq:nonlin_ctrl}
		\begin{aligned}
			\min \quad &\int_\Lambda a(x, y(x)) \d\lambda^d(x) + \frac\alpha2 D(u), \\
			\text{s.t.} \quad & y \in W^{1,q}_0(\Omega), \quad u \in \MM(\omega_1),\\
			& - \laplace y = u \quad \text{in } W^{-1,q}(\Omega).
		\end{aligned}
	\end{equation}
	Here, $\Lambda \in \BB(\Omega)$ is a given observation domain with positive Lebesgue measure
	and positive distance to the control domain $\omega_1$.
	The nonlinearity $a \colon \Lambda \times \R \to \R$ is assumed
	to satisfy the assumptions of
	\cite[Theorem~9]{GoldbergKampowskyTroeltzsch1992}
	such that the associated Nemytskii operator is $C^2$ on $L^\infty(\Lambda)$.
	As usual, $W^{1,q}_0(\Omega)$ denotes the closure of 
	$C^\infty_c(\Omega)$ w.r.t.\ the $W^{1,q}(\Omega)$-norm and $W^{-1,q}(\Omega)$ denotes the 
	dual of $W^{1,q'}_0(\Omega)$ where $1/q + 1/q' = 1$.

	According to \cite[Theorem~0.5]{JK95},
	we can choose the exponent $q$
	such that $-\laplace$
	is an isomorphism
	from $W_0^{1,q}(\Omega)$ to $W^{-1,q}(\Omega)$
	and
	from $W_0^{1,q'}(\Omega)$ to $W^{-1,q'}(\Omega)$.
	Consequently,
	the state equation has a unique solution $y \in W_0^{1,q}(\Omega)$
	for every $u \in \MM(\omega_1)$,
	see, e.g.,
	\cite[Remark~2.2, Lemma~3.1]{MeyerWachsmuth2025}.
	We want to show that the mapping $u \mapsto y_{|\Lambda}$
	is even continuous from $C^1(\omega_1)\dualspace$ to $L^\infty(\Lambda)$.
	To this end, we consider $u \in \MM(\omega_1)$ together with the associated solution
	$y \in W_0^{1,q}(\Omega)$.

	We prepare a duality argument.
	For some $w \in L^1(\Lambda)$, let $v \in W_0^{1,q}(\Omega)$
	be the unique solution of $-\laplace v = w \in W^{-1,q}(\Omega)$.
	Since $v$ is harmonic in a neighborhood of $\omega_1$,
	we have $v \in C^1(\omega_1)$ due to Weyl's lemma, see \cite[Proposition~2.18]{Ponce}.
	Consequently, an application of the closed graph theorem
	to the operator
	$L^1(\Lambda) \ni w \mapsto v \in \set{v \in W_0^{1,q}(\Omega) \given v_{| \omega_1} \in C^1(\omega_1) }$
	implies that this operator is bounded, in particular,
	$\norm{v}_{C^1(\omega_1)} \le C \norm{w}_{L^1(\Lambda)}$.
	We denote this operator by
	$S_\flat \colon L^1(\Lambda) \to C^1(\omega_1)$.
	Applying a similar argument with $C^2(\omega_1)$
	together with the Arzelà--Ascoli theorem
	shows that $S_\flat$ is compact.

	Now, we additionally assume
	$w \in L^{q'}(\Lambda) \subset W^{-1,q'}(\Omega)$
	which implies
	that the above solution $v$
	satisfies $v = (-\laplace)^{-1} w \in W_0^{1,q'}(\Omega)$.
	Consequently,
	\begin{align*}
		\int_\Lambda y w \d\lambda^d
		=
		\int_\Omega \nabla y \nabla v \d\lambda^d
		=
		\int_{\omega_1} v \d u
		&\le
		\norm{v}_{C^1(\omega_1)} \norm{u}_{C^1(\omega_1)\dualspace}
		\\&
		\le
		C \norm{w}_{L^1(\Lambda)} \norm{u}_{C^1(\omega_1)\dualspace}
		.
	\end{align*}
	Since $w \in L^{q'}(\Lambda)$ was arbitrary,
	this implies
	\begin{equation*}
		\norm{y}_{L^\infty(\Lambda)}
		\le
		C \norm{u}_{C^1(\omega_1)\dualspace}
	\end{equation*}
	which is the desired continuity
	of
	$C^1(\omega_1)\dualspace \ni u \mapsto y_{|\Lambda} \in L^\infty(\Lambda)$.
	Let us denote by $S \colon C^1(\omega_1)\dualspace \to L^\infty(\Lambda)$
	the associated solution operator.
	Note that the above calculation also yields
	that $S$ is the adjoint of $S_\flat$.
	Consequently, the reduced objective
	$\FF \colon C^1(\omega_1)\dualspace \to \R$
	defined via
	\begin{equation*}
		\FF(u)
		=
		\int_\Lambda a(x, (S u)(x)) \d\lambda^d
	\end{equation*}
	is $C^2$ and its derivatives are given by
	\begin{equation*}
		\FF'(u) v
		=
		\int_\Lambda \frac{\partial}{\partial y} a(x, y(x)) (S v)(x) \d\lambda^d
		=
		\int_{\omega_1} p v \d\lambda^d
	\end{equation*}
	and
	\begin{equation*}
		\FF''(u)[v_1, v_2]
		=
		\int_\Lambda \frac{\partial}{\partial y} a(x, y(x)) (S v_1)(x) (S v_2)(x) \d\lambda^d
		,
	\end{equation*}
	where $y = S u$ is the state and the adjoint state $p \in W_0^{1,q'}(\Omega)$
	satisfies
	$-\laplace p = \chi_\Lambda \frac{\partial}{\partial y} a(y)$.
	Arguing as above, we find $p \in C^1(\omega_1)$.

	Now, it is easy to check that \cref{assu:F_second_order_Taylor} is satisfied.
	Indeed,
	let $\bar u \in \MM(\omega_1)$ be given
	and denote by $\bar p$ the associated adjoint state.
	Then, the above observations yield the regularity
	$\FF'(\bar u) = \bar p_{|\omega_1} \in C^1(\omega_1)$.
	Moreover, the differentiability assumption \eqref{eq:taylor2}
	follows since $\FF$ is twice continuously differentiable on $C^1(\omega_1)\dualspace$.
	Finally,
	$C^1(\omega_1)\dualspace \ni h \mapsto \FF''(\bar u)h^2$
	is sequentially weak-$\star$ continuous,
	since $C^1(\omega_1)\dualspace \ni h \mapsto S h \in L^\infty(\Lambda)$
	is sequentially weak-$\star$-to-strong continuous,
	since $S$ is the adjoint of the compact operator $S_\flat$.
	Hence, we have verified \cref{assu:F_second_order_Taylor}.
	Consequently, if \cref{assu:weaker_regularity_assumptions_on_T}
	is satisfied by $\bar\varphi = -\frac{2}{\alpha}\bar p$,
	\cref{thm:SNC} is applicable
	and the second-order condition \eqref{eq:Phi''_positive}
	characterizes the quadratic growth \eqref{eq:quadratic_growth}.
\end{example}

Next, we briefly argue that these arguments do not apply for the control of a semilinear equation
even if the observation domain and the control domain are separated.
\begin{remark}
	\label{rem:semilinear control}
	Let us consider the control problem
	\begin{equation}
		\label{eq:optctrl_semilin}
		\begin{aligned}
			\min & \quad \frac12 \int_\Lambda \abs{y - y_d}^2 \d \lambda^d + \frac{\alpha}{2} D(u) \\
			\text{s.t.} & \quad y \in W^{1,q}_0(\Omega), \quad u \in \MM(\omega_1),\\
			& \quad - \laplace y + g(y) = u \quad \text{in } W^{-1,q}(\Omega),
		\end{aligned}
	\end{equation}
	similar to \cref{ex:nonlin_ctrl}, but including a semilinearity $g \in C^2(\R)$.
	Under some mild assumptions on $g$,
	one can prove that the control-to-state map $S \colon \MM(\omega_1) \to W_0^{1,q}(\Omega)$,
	$q \in (1,d/(d-1))$,
	is well defined and $C^2$, see \cite[Theorem~2.1 and 2.2]{CasasKunisch2014}.
	We denote by $F \colon \MM(\omega_1) \to \R$ the associated reduced objective.
	From the above work we obtain the expressions
	\begin{subequations}
		\label{eq:zweiteablF_zum_zweiten}
		\begin{align}
			\label{eq:zweiteablF_zum_zweiten:1}
			F''(u)[h_1, h_2]
			&= \int_{\Lambda} \eta_1 \eta_2 \d\lambda^d +  \int_\Lambda (y - y_d) w \d\lambda^d \\
			\label{eq:zweiteablF_zum_zweiten:2}
			&= \int_{\Lambda}  \eta_1 \eta_2 \d\lambda^d - \int_{\Omega} p  g''(y) \eta_1 \eta_2 \d \lambda^d,    
		\end{align}
	\end{subequations}
	where $y = S(u)$, while $\eta_i = S'(u) h_i$ and $w = S''(u)[h_1,h_2]$ are the solutions of
	\begin{subequations}
		\label{eq:lin_eq}
		\begin{align}
			- \laplace \eta_i + g'(y) \eta_i &= h_i & & \text{in } W^{-1,q}(\Omega), \label{eq:lin_eq:1}\\
			- \laplace w + g'(y) w &= - g''(y) \eta_1\eta_2  & & \text{in } W^{-1,q}(\Omega), \label{eq:lin_eq:2}
		\end{align}
	\end{subequations}
	and $p$ is the solution of the adjoint equation
	\begin{equation}
		\label{eq:adjeq_Zwei}
		- \laplace p + g'(y) p = \chi_\Lambda (y - y_d) \quad \text{in } W^{-1,q'}(\Omega)
		.
	\end{equation}
	Let us argue that $F''(u)[\cdot, \cdot]$ is, in general, not bounded w.r.t.\ the norm of $C^1(\omega_1)\dualspace$.

	Fix $x \in \omega_1$ and consider the function $h \in C^1(\omega_1)\dualspace$,
	$\dual{h}{\varphi} = \frac{\partial\varphi}{\partial x_1}(x)$.
	Then, one can check that $(-\laplace)^{-1} h$
	has a singularity in the neighborhood of $x$ which, essentially, coincides
	with $\frac{\partial \Phi}{\partial x_1}(\cdot - x)$,
	where $\Phi \colon \R^d \setminus \set{0} \to \R$ is the fundamental solution.
	In particular, $(-\laplace)^{-1} h$ is not $L^2$-integrable in a neighborhood of $x$.
	Similarly, we expect that the solution $\eta$ of \eqref{eq:lin_eq:1}
	(if it exists) has a similar singularity.
	With $h_1 = h_2 = h$, the term $\eta^2$ which appears in \eqref{eq:zweiteablF_zum_zweiten:2}
	and \eqref{eq:lin_eq:2}
	is not even integrable.
	That is, the expressions in \eqref{eq:zweiteablF_zum_zweiten}, \eqref{eq:lin_eq}
	cannot be used to extend $F''(u)$
	to $C^1(\omega_1)\dualspace$.
	Moreover, the above $h$ can be approximated by
	a sequence $\seq{h_k} \subset L^2(\omega_1)$
	such that $h_k \weaklystar h$ in $C^1(\omega_1)\dualspace$
	via, e.g., a mollification argument.
	Consequently, the sequence of associated linearized states
	$\eta_k := S'(u) h_k$
	cannot be uniformly $L^2$-integrable in a neighborhood of $x$,
	since they converge in a weak sense towards $\eta = S'(u) h$,
	which is not locally $L^2$-integrable.
	Now, we additionally assume that the expression
	$p g''(y)$, which appears in \eqref{eq:zweiteablF_zum_zweiten:2},
	is uniformly positive (or negative)
	in a neighborhood of $x$.
	This implies that $F''(u) h_k^2$ diverges as $k \to \infty$.
	Consequently, $F''(u) \colon \MM(\omega_1)^2 \to \R$
	is not bounded w.r.t.\ the norm of $C^1(\omega_1)\dualspace$.
	Hence, we expect that, in general, \cref{assu:F_second_order_Taylor} is not satisfied for this problem.

	Without going into details, we mention two possible modifications of \eqref{eq:optctrl_semilin}
	under which \cref{assu:F_second_order_Taylor} can be satisfied again.
	First, one could replace the nonlinearity $g(y)$ in the state equation by $\chi_\Lambda g(y)$.
	Under this modification, the problematic term $g''(y) \eta_1 \eta_2$
	appearing in \eqref{eq:zweiteablF_zum_zweiten:2} and \eqref{eq:lin_eq:2}
	becomes $\chi_\Lambda g''(y) \eta_1 \eta_2$.
	Since the right-hand side of the linearized equation \eqref{eq:lin_eq:1}
	only lives on $\omega_1$ one can show, similar to \cref{ex:nonlin_ctrl},
	that $\eta_i$ enjoys higher regularity on $\Lambda$.
	Consequently, $\chi_\Lambda g''(y) \eta_1 \eta_2$ is integrable and everything is fine again.

	Second, we could replace the right-hand side $u$ in the state equation by $B u$,
	where $B$ is a linear operator satisfying some suitable smoothing properties.
	If, for example, $B \colon C^1(\omega_1)\dualspace \to \MM(\omega_1)$ is bounded,
	the control-to-state operator becomes $u \mapsto S(B u)$,
	where $S$ is the solution operator of the semilinear equation.
	Since $S$ is $C^2$ on $\MM(\omega_1)$, the operator $S \circ B$
	is $C^2$ on $C^1(\omega_1)\dualspace$
	and, consequently, \cref{assu:F_second_order_Taylor} applies.
\end{remark}

\begin{remark}
	It would of course be desirable to verify the conditions in \cref{assu:F_second_order_Taylor} for the 
	case of an optimal control problem
	governed by a semilinear equation
	as in \cref{rem:semilinear control}.
	Our considerations, however, show that this is not possible if the space $X$ in the abstract second-order analysis 
	is chosen to be $C^1(\omega_1)\dualspace$. Instead a more regular space is needed for this.  
	On the other hand, \cref{ex:counter_example_inequality} shows that the non-degeneracy condition \eqref{NDC} 
	is in general not true in more regular spaces,
	especially in case of
	$X = W^{1,q'}(\omega_1)\dualspace = W^{-1,q}(\omega_1)$ with $q \geq d / (d - 1)$.
	It remains an open question, if \eqref{NDC} can be verified for the transportation distance, if $q$ is chosen to 
	be smaller than $d/(d-1)$. If this is the case, then it should be possible to verify \cref{assu:F_second_order_Taylor}
	for the control problem from \cref{rem:semilinear control}. This, however, is subject to future research.
\end{remark}

\printbibliography

@article{WachsmuthWachsmuth2022,
  author        = {Daniel Wachsmuth and Gerd Wachsmuth},
  title         = {Second-order conditions for non-uniformly convex integrands: quadratic growth in {$L^1$}},
  year          = {2022},
  doi = {10.46298/jnsao-2022-8733},
  publisher = {Centre pour la Communication Scientifique Directe ({CCSD})},
  volume = {3},
  journal = {Journal of Nonsmooth Analysis and Optimization}
}

@online{WachsmuthWalter2024,
  author        = {Gerd Wachsmuth and Daniel Walter},
  title         = {No-gap second-order conditions for minimization problems in spaces of measures},
  year          = {2024},
  eprint        = {2403.12001},
  eprinttype    = {arXiv}
}

@article {JK95,
    AUTHOR = {Jerison, David and Kenig, Carlos E.},
     TITLE = {The inhomogeneous {D}irichlet problem in {L}ipschitz domains},
   JOURNAL = {J. Funct. Anal.},
    VOLUME = {130},
      YEAR = {1995},
    NUMBER = {1},
     PAGES = {161--219},
      ISSN = {0022-1236,1096-0783},
       DOI = {10.1006/jfan.1995.1067},
}

@book{Ponce,
	author 		= {Ponce, Augusto C.},
	title 			= {Elliptic {PDEs}, Measures and Capacities},
	fseries 		= {EMS Tracts in Mathematics},
	series 		= {EMS Tracts Math.},
	volume 		= {23},
	year  		= {2016},
	publisher 		= {Z{\"u}rich: European Mathematical Society (EMS)},
	doi 			= {10.4171/140},
}

@BOOK{santambrogio,
 title = {Optimal Transport for Applied Mathematicians: Calculus of Variations, PDEs, and Modeling},
 ISBN = {9783319208282},
 ISSN = {2374-0280},
 DOI = {10.1007/978-3-319-20828-2},
 journal = {Progress in Nonlinear Differential Equations and Their Applications},
 publisher = {Springer International Publishing},
 author = {Santambrogio, Filippo},
 year = {2015}
}

@ARTICLE{PieperVexler2013,
 title = {A Priori Error Analysis for Discretization of Sparse Elliptic Optimal Control Problems in Measure Space},
 volume = {51},
 ISSN = {1095-7138},
 DOI = {10.1137/120889137},
 number = {4},
 journal = {SIAM Journal on Control and Optimization},
 publisher = {Society for Industrial & Applied Mathematics (SIAM)},
 author = {Pieper, Konstantin and Vexler, Boris},
 year = {2013},
 pages = {2788–2808}
}

@article {ClasonKunisch2011,
    AUTHOR = {Clason, Christian and Kunisch, Karl},
     TITLE = {A duality-based approach to elliptic control problems in
              non-reflexive {B}anach spaces},
   JOURNAL = {ESAIM Control Optim. Calc. Var.},
    VOLUME = {17},
      YEAR = {2011},
    NUMBER = {1},
     PAGES = {243--266},
       DOI = {10.1051/cocv/2010003},
}

@article {CasasKunisch2014,
    AUTHOR = {Casas, Eduardo and Kunisch, Karl},
     TITLE = {Optimal control of semilinear elliptic equations in measure
              spaces},
   JOURNAL = {SIAM J. Control Optim.},
    VOLUME = {52},
      YEAR = {2014},
    NUMBER = {1},
     PAGES = {339--364},
       DOI = {10.1137/13092188X},
}

@article{BrediesPikkarainen2013,
    AUTHOR = {Bredies, Kristian and Pikkarainen, Hanna Katriina},
     TITLE = {Inverse problems in spaces of measures},
   JOURNAL = {ESAIM Control Optim. Calc. Var.},
    VOLUME = {19},
      YEAR = {2013},
    NUMBER = {1},
     PAGES = {190--218},
       DOI = {10.1051/cocv/2011205},
}

@BOOK{AubinFrankowska2008,
 title = {Set-Valued Analysis},
 ISBN = {9780817648473},
 ISSN = {2197-1811},
 DOI = {10.1007/978-0-8176-4848-0},
 journal = {Modern Birkhäuser Classics},
 publisher = {Birkhäuser Boston, MA},
 author = {Aubin, Jean-Pierre and Frankowska, Hélène},
 year = {2008}
}

@article{Rockafellar1971,
author = "Rockafellar, R. T.",
journal = "Pacific Journal of Mathematics",
number = "2",
pages = "439--469",
publisher = "Pacific Journal of Mathematics, A Non-profit Corporation",
title = "Integrals which are convex functionals. II.",
url = "https://projecteuclid.org:443/euclid.pjm/1102969571",
volume = "39",
year = "1971"
}

@ARTICLE{ChristofWachsmuth2019,
 title = {Differential Sensitivity Analysis of Variational Inequalities with Locally Lipschitz Continuous Solution Operators},
 volume = {81},
 ISSN = {1432-0606},
 DOI = {10.1007/s00245-018-09553-y},
 number = {1},
 journal = {Applied Mathematics \& Optimization},
 publisher = {Springer Science and Business Media LLC},
 author = {Christof, Constantin and Wachsmuth, Gerd},
 year = {2019},
 pages = {23–62}
}

@ARTICLE{BorchardWachsmuth2024,
 title = {Second-order conditions for spatio-temporally sparse optimal control via second subderivatives},
 volume = {5},
 ISSN = {2700-7448},
 DOI = {10.46298/jnsao-2024-12604},
 journal = {Journal of Nonsmooth Analysis and Optimization},
 publisher = {Centre pour la Communication Scientifique Directe (CCSD)},
 author = {Borchard, Nicolas and Wachsmuth, Gerd},
 year = {2024}
}

@online{MeyerWachsmuth2025,
  author        = {Christian Meyer and Gerd Wachsmuth},
  title         = {Optimal control of the Poisson equation with transport regularization: Properties of optimal transport plans and transport map},
  year          = {2025},
  eprint        = {2506.02808},
  eprinttype    = {arXiv}
}

@BOOK{Brezis2011,
	author = {Brezis, Haim},
	title = {Functional analysis, {S}obolev spaces and partial differential
		equations},
	series = {Universitext},
	publisher = {Springer, New York},
	year = {2011},
	isbn = {978-0-387-70913-0},
	doi = {10.1007/978-0-387-70914-7},
}

@BOOK{AmbrosioGigliSavare2005,
 author = {Ambrosio, Luigi and Gigli, Nicola and Savaré, Giuseppe},
 title = {Gradient flows in metric spaces and in the space of probability measures},
 year = {2005},
 publisher = {Basel: Birkhäuser},
 doi = {10.1007/b137080}
}

@online{WachsmuthWalter2025:1,
  author        = {Gerd Wachsmuth and Daniel Walter},
  title         = {Proximal gradient methods in Banach spaces},
  year          = {2025},
  eprint        = {2509.24685},
  eprinttype    = {arXiv}
}

@ARTICLE{DeckelnickHinze2012,
    AUTHOR = {Deckelnick, Klaus and Hinze, Michael},
     TITLE = {A note on the approximation of elliptic control problems with
              bang-bang controls},
   JOURNAL = {Comput. Optim. Appl.},
    VOLUME = {51},
      YEAR = {2012},
    NUMBER = {2},
     PAGES = {931--939},
      ISSN = {0926-6003},
       DOI = {10.1007/s10589-010-9365-z},
}

@ARTICLE{WachsmuthWachsmuth2009,
	title                    = {Convergence and Regularization Results for Optimal Control Problems
		with Sparsity Functional},
	author                   = {Gerd Wachsmuth and Daniel Wachsmuth},
	journal                  = {ESAIM: Control, Optimisation and Calculus of Variations},
	year                     = {2011},
	number                   = {3},
	pages                    = {858--886},
	volume                   = {17},
	doi                      = {10.1051/cocv/2010027}
}

@online{Caja-LopezDelgadinoKitagawa2026,
  author        = {Caja-Lopez, Francisco-Unai and Matias G. Delgadino and Jun Kitagawa},
  title         = {Stability of optimal transport maps and second variation of the 2-Monge-Kantorovich distance},
  year          = {2026},
  eprint        = {2605.24232},
  eprinttype    = {arXiv}
}

@ARTICLE{ChristofWachsmuth2018,
	title= {No-Gap Second-Order Conditions via a Directional Curvature Functional},
	author= {Christof, Constantin and Wachsmuth, Gerd},
	journal = {SIAM Journal on Optimization},
	volume = {28},
	number = {3},
	pages = {2097-2130},
	year = {2018},
	doi = {10.1137/17M1140418},
}

@BOOK{SchaeferWolff1999,
 title = {Topological Vector Spaces},
 ISBN = {9781461214687},
 ISSN = {0072-5285},
 DOI = {10.1007/978-1-4612-1468-7},
 journal = {Graduate Texts in Mathematics},
 publisher = {Springer New York},
 author = {Schaefer, H. H. and Wolff, M. P.},
 year = {1999}
}

@ARTICLE{GoldbergKampowskyTroeltzsch1992,
	title                    = {On {N}emytskij operators in {$L_p$}-spaces of abstract functions},
	author                   = {Goldberg, Hyman and Kampowsky, Winfried and Tr{\"o}ltzsch, Fredi},
	journal                  = {Mathematische Nachrichten},
	year                     = {1992},
	pages                    = {127--140},
	volume                   = {155},
	doi                      = {10.1002/mana.19921550110},
	issn                     = {0025-584X}
}

@ARTICLE{Do1992,
	author = {Chi Ngoc Do},
	TITLE = {Generalized second-order derivatives of convex functions in reflexive {B}anach spaces},
	JOURNAL = {Transactions of the American Mathematical Society},
	VOLUME = {334},
	YEAR = {1992},
	NUMBER = {1},
	PAGES = {281--301},
	DOI = {10.2307/2153983}
}

@BOOK{RockafellarWets1998,
	doi = {10.1007/978-3-642-02431-3},
	year = 1998,
	publisher = {Springer Berlin Heidelberg},
	author = {Ralph Tyrrell Rockafellar and Roger J. B. Wets},
	title = {Variational Analysis}
}

@ARTICLE{CasasTroeltzsch2012,
 title = {Second Order Analysis for Optimal Control Problems: Improving Results Expected From Abstract Theory},
 volume = {22},
 ISSN = {1095-7189},
 DOI = {10.1137/110840406},
 number = {1},
 journal = {SIAM Journal on Optimization},
 publisher = {Society for Industrial & Applied Mathematics (SIAM)},
 author = {Casas, Eduardo and Tröltzsch, Fredi},
 year = {2012},
 pages = {261–279}
}

@article{Casas2012,
	Title = {Second Order Analysis for Bang-Bang Control Problems of {PDEs}},
	Author = {Eduardo Casas},
	Journal = {SIAM Journal on Control and Optimization},
	Year = {2012},
	Number = {4},
	Pages = {2355--2372},
	Volume = {50},
	DOI = {10.1137/120862892}
}

@ARTICLE{CasasWachsmuthWachsmuth2017,
	title = {Sufficient Second-Order Conditions for Bang-Bang Control Problems},
	author                   = {Eduardo Casas and Daniel Wachsmuth and Gerd Wachsmuth},
	journal                  = {SIAM Journal on Control and Optimization},
	year                     = {2017},
	volume = {55},
	number = {5},
	pages = {3066-3090},
	doi = {10.1137/16M1099674},
}

@ARTICLE{Levy1993,
 title = {Second-order epi-derivatives of integral functionals},
 volume = {1},
 ISSN = {1572-932X},
 DOI = {10.1007/bf01027827},
 number = {4},
 journal = {Set-Valued Analysis},
 publisher = {Springer Science and Business Media LLC},
 author = {Levy, A. B.},
 year = {1993},
 pages = {379–392}
}

@article {DuvalPeyre2015,
    AUTHOR = {Duval, Vincent and Peyr\'{e}, Gabriel},
     TITLE = {Exact support recovery for sparse spikes deconvolution},
   JOURNAL = {Found. Comput. Math.},
    VOLUME = {15},
      YEAR = {2015},
    NUMBER = {5},
     PAGES = {1315--1355},
       DOI = {10.1007/s10208-014-9228-6},
}

@article {ClasonKunisch2012,
    AUTHOR = {Clason, Christian and Kunisch, Karl},
     TITLE = {A measure space approach to optimal source placement},
   JOURNAL = {Comput. Optim. Appl.},
     VOLUME = {53},
      YEAR = {2012},
    NUMBER = {1},
     PAGES = {155--171},
         DOI = {10.1007/s10589-011-9444-9},
}

@ARTICLE{NeitzelPieperVexlerWalter2019,
 title = {A sparse control approach to optimal sensor placement in PDE-constrained parameter estimation problems},
 volume = {143},
 ISSN = {0945-3245},
 DOI = {10.1007/s00211-019-01073-3},
 number = {4},
 journal = {Numerische Mathematik},
 publisher = {Springer Science and Business Media LLC},
 author = {Neitzel, Ira and Pieper, Konstantin and Vexler, Boris and Walter, Daniel},
 year = {2019},
 month = Sept, pages={943–984}
}

@INBOOK{ScherzerWalch2009,
 title = {Sparsity Regularization for Radon Measures},
 ISBN = {9783642022562},
 ISSN = {1611-3349},
 DOI = {10.1007/978-3-642-02256-2_38},
 booktitle = {Scale Space and Variational Methods in Computer Vision},
 publisher = {Springer Berlin Heidelberg},
 author = {Scherzer, Otmar and Walch, Birgit},
 year = {2009},
 pages = {452–463}
}

@online{WachsmuthWachsmuth2025:1,
  author        = {Daniel Wachsmuth and Gerd Wachsmuth},
  title         = {Continuous differentiability of the signum function and Newton's method for bang-bang control},
  year          = {2025},
  eprint        = {2509.24829},
  eprinttype    = {arXiv}
}

\end{document}